\documentclass[12pt,reqno]{amsart}   	
\usepackage[letterpaper, margin=1in]{geometry}
\usepackage{graphicx}				
\usepackage{amssymb}
\usepackage{mathtools,amsmath}
\usepackage{amsthm}
\usepackage{amssymb}
\usepackage{mathrsfs}
\usepackage{epstopdf}
\usepackage{color}
\usepackage{tikz-cd}
\usepackage{subfigure}
\usepackage{enumerate}
\usepackage{comment}
\usepackage[pagebackref,colorlinks,citecolor=blue,linkcolor=magenta]{hyperref}
\usepackage{caption}

\usepackage[linesnumbered,ruled]{algorithm2e}
\RequirePackage{amsthm,amsmath,amsfonts,amssymb}
\usepackage[utf8]{inputenc}
\usepackage{chngcntr}
\usepackage{float}
\usepackage{natbib}

\theoremstyle{plain}
   \newtheorem{theorem}{Theorem}[section]
   \newtheorem{proposition}[theorem]{Proposition}
   \newtheorem{lemma}[theorem]{Lemma}
   \newtheorem{corollary}[theorem]{Corollary}

\theoremstyle{definition}
   \newtheorem{definition}[theorem]{Definition}

   \newtheorem{example}[theorem]{Example}
   
\theoremstyle{remark}
 \newtheorem{remark}[theorem]{Remark}

\newcommand{\xx}{\mathbf{x}}
\newcommand{\yy}{\mathbf{y}}
\newcommand{\zz}{\mathbf{z}}
\newcommand{\LL}{\mathcal{L}}

\newcommand{\MM}{\mathcal{M}}

\newcommand{\CC}{\mathcal{C}}
\newcommand{\RR}{\mathcal{R}}
\newcommand{\TT}{\mathcal{T}}
\newcommand{\I}{\mathcal{I}}
\newcommand{\J}{\mathcal{J}}
\newcommand{\V}{\mathcal{V}}

\newcommand{\R}{\mathbb{R}}
\newcommand{\C}{\mathbb{C}}

\newcommand{\sat}{\mathrm{Sat}}

\DeclareMathOperator{\loc}{local}
\DeclareMathOperator{\glo}{global}

\newcommand{\cc}{X_C=\xx_C}
\newcommand{\Cp}{X_{C'}=\xx_{C'}}
\newcommand{\children}{\mathrm{ch}}

\newcommand{\Lift}{\mathrm{Lift}}
\newcommand{\Quad}{\mathrm{Quad}}

\renewcommand{\subset}{\subseteq}
\renewcommand{\emptyset}{\varnothing}

\DeclareMathOperator{\per}{per}

\SetKwInput{KwInput}{Input}
\SetKwInput{KwOutput}{Output}

\newcommand\independent{\protect\mathpalette{\protect\independenT}{\perp}}
\def\independenT#1#2{\mathrel{\rlap{$#1#2$}\mkern2mu{#1#2}}}

\def\newop#1{\expandafter\def\csname #1\endcsname{\mathop{\rm
#1}\nolimits}}

\newcommand{\TTb}{\overline{\TT}}

\newop{Inv}
\newop{conv}
\newop{pa}
\newop{de}
\newop{an}
\newop{nd}
\newop{ch}
\newop{CS}


\title{Decomposable context-specific models}

\author{Yulia Alexandr}
\address{University of California, Berkeley, 1045 Evans Hall, Berkeley, CA 94720-3840 USA}
\email{yalexandr@berkeley.edu}

\author{Eliana Duarte}
\address{Universidade do Porto, Rua do Campo Alegre 687, 4169-007 Porto, Portugal}
\email{eliana.gelvez@fc.up.pt}

\author{Julian Vill}
\address{Fakult\"at f\"ur Mathematik, Otto-von-Guericke Universit\"at Magdeburg, Universitätsplatz 2, 39106 Magdeburg, Germany}
\email{julian.vill@ovgu.de}

\begin{document}

\begin{abstract}
    We introduce a family of discrete context-specific models, which we call \textit{decomposable}. We construct this family from the subclass of staged tree models known as CStree models.
    We give an algebraic and combinatorial characterization of all context-specific independence relations that hold in a decomposable context-specific model, which yields a Markov basis. 
    We prove that a directed version of the moralization operation applied to the graphical
    representation of a context-specific model does not affect the implied independence relations, thus affirming that these models are algebraically described by a finite collection of decomposable graphical models. More generally, we establish that several algebraic, combinatorial, and geometric properties of decomposable context-specific models generalize those of decomposable graphical models to the context-specific setting.
\end{abstract}

\maketitle

\section{Introduction}

A discrete graphical model $\MM(G)$ associated to a graph $G$ with $p$ nodes
is a set of joint probability distributions for a vector of discrete 
random variables $(X_1,\ldots,X_p)$. The distributions in $\MM(G)$ satisfy conditional independence (CI) relations according to the non-adjacencies of the graph $G$. The type of graph used to
encode CI relations is typically a directed acyclic graph (DAG) or an undirected 
graph (UG), although other kinds of graphs are possible \cite[]{L96}.

Graphical models are widely used in several fields of science, such as artificial intelligence, biology, and epidemiology \cite[]{KF09,pearl:1988,maathuis:2018}. However,
in some applications it is useful to consider models that encode a finer form of independence.
Context-specific independence (CSI)
is a generalization of conditional independence where the conditional independence between the random variables only holds for particular outcomes of the variables in the conditioning set. The classical graphical models based on DAGs or UGs are no longer able to capture
these more refined relations. Several extended graphical representations of CSI models have been proposed in the literature,  \cite[]{BFGK96,PZ03,CHM97,pensar:2015,SA08}. 
Apart from its usage to encode model assumptions more accurately,  context-specific independence is also important
in the study of structural causal models because the presence of more refined independence can improve the identifiability of causal links \cite[]{tikka:2019}.

A graphical model $\MM(G)$ associated to an undirected graph  $G$ is called \emph{decomposable} if $G$ is chordal. Decomposable graphical models
play a prominent role among graphical models because they exhibit optimal properties for probabilistic inference \cite[Ch. 9]{KF09}. There are several characterizations of decomposable models
 in terms of their algebraic, combinatorial, and geometric properties \cite[]{GMS06,L96,geiger:2001,DS21}.
 In this article we generalize this class of models to the context-specific setting,
by defining \emph{decomposable context-specific models}
(see Section~\ref{sec:3-vars})
and prove that they mirror many of  the properties that characterize decomposable graphical models. 
For brevity, we will refer to these models as decomposable CSmodels. These models will be constructed from a subclass of staged tree models, first introduced in \cite[]{SA08}. 
Staged tree models are a very general class of discrete (categorical) multivariate statistical models whose  applications, causal interpretation, and learning is a topic of interest in recent statistical literature
\cite{LG2022,LG2023,GBRS2018, CLRG2022r}. The distributions that belong to a staged tree model  $\MM(\TT)$ are determined by a directed tree $\TT$. The inner 
nodes of the tree are partitioned into sets called stages, and the leaves correspond to the state space of the model. The key feature of the stages is that they efficiently encode context-specific independence statements. In particular, every discrete DAG model can be represented using a staged tree model because any conditional independence statement is the union of several context-specific independence statements.

Previous work on context-specific versions
of decomposable models defines them by associating labels to the edges 
of a 
decomposable undirected graph and preserving the properties such as 
perfect elimination ordering and clique factorization \cite[]{janhunen:2017,corander:2003,NPKC2014}.
Our approach here is different in that we define  decomposable CSmodels to be staged tree models that satisfy two conditions. The first one is that the staged tree is balanced (Section~\ref{subsec:dCSmodels}), this property implies that the model is log-linear; the second one is that the staged tree is  a CStree (Section~\ref{sec:cstrees}, \cite[]{DS22}), this means that the stages satisfy additional properties which imply that the context-specific independence statements that hold for the model can be represented using a collection of DAGs. Thus each decomposable CSmodel is represented by a CStree.

Similar to discrete DAG models, there are two ways to define a 
CStree model $\MM(\TT)$, associated to the CStree $\TT$. The first approach uses a recursive factorization property
according to $\TT$, while the second one uses the local CSI relations  implied by $\TT$. 
From an algebro-geometric point of view, the
recursive factorization property is a polynomial parametrization of $\MM(\TT)$, while the polynomials associated 
to the local CSI statements in $\TT$ define the model $\MM(\TT)$ implicitly.
An important open problem that arises in the study of context-specific models is characterizing the set of all CSI statements implied
by the local CSI statements defining the model \cite[]{BFGK96,corander:2016}. Such a problem is especially amenable to algebraic techniques
because any CSI relation that holds in the model can be represented by a collection of polynomials. Algebraically, this problem can be solved by finding a prime polynomial ideal that defines  $\MM(\TT)$ implicitly \cite[]{GSS05}. Moreover, for log-linear models, the generators of the ideal that defines $\MM(\TT)$ form a Markov basis \cite[]{diaconis:98}. Our first main theorem is an algebraic characterization of the distributions that belong to a decomposable CSmodel. This theorem is similar to the Hammersley-Clifford theorem for undirected graphical models and its generalization \cite[Theorem 4.2]{GMS06}.
\begin{theorem}[Context-specific Hammersley-Clifford] \label{thm:CShc}
A distribution $f$ factorizes according to $\TT$ if and only if the polynomials associated to saturated CSI statements in $\TT$ vanish at~$f$.
Moreover, the polynomials associated to the saturated CSI statements of a decomposable CSmodel form its Markov basis.
\end{theorem}
Every decomposable CSmodel $\MM(\TT)$ is a CStree model, therefore it can be represented by a collection of
\emph{minimal context DAGs}. We also prove that the
saturated CSI statements (i.e. statements that involve all of the variables in the DAG) in Theorem~\ref{thm:CShc} are obtained as the union of saturated
$d$-separation statements that hold in each of the minimal context DAGs that represent the model $\MM(\TT)$ (Corollary~\ref{cor:saturated-statements-for-minimal-dags}). As a consequence of our algebraic characterization of decomposable CSmodels, we obtain the next two theorems.

The directed moralization of a DAG is constructed by moralizing the DAG but keeping all edges directed and directing the new edges (see Definition~\ref{def:directed_moralization}).
\begin{theorem}
    In a decomposable CSmodel, the directed moralizations of the minimal context DAGs imply the same CSI statements as the model itself. In particular, one can apply the directed moralization operation until all context DAGs are perfect.
\end{theorem}

\begin{theorem}
    Every decomposable CSmodel  is an intersection of a finite number of decomposable DAG models.
\end{theorem}

Our work also contributes to understanding the set of CSI statements that hold for certain context-specific models known as LDAGs \cite[]{pensar:2015}. Briefly, an LDAG is a context-specific model represented by a DAG with edge labels, these labels encode the extra CSI relations that hold for the model. Every CStree is an LDAG, and every LDAG is a staged tree \cite[]{DS22}. Whenever the LDAG is represented by a balanced CStree, Theorem~\ref{thm:CShc} gives a complete characterization of the CSI statements that hold for the LDAG. In general, however, describing all CSI statements that hold for LDAGs is coNP-hard \cite[]{corander:2016}. 
The decomposable models studied in \cite[]{janhunen:2017,corander:2003,NPKC2014} are defined in a similar fashion as the LDAGs, except their starting point is a decomposable undirected graph. 
We leave it as a direction for future
research to establish the relation between
these context specific decomposable models  and the decomposable CSmodels  we define.

This paper is organized as follows. In Section~\ref{sec:prelim} we present the necessary background on DAG models, staged tree models
 and CStree models. Decomposable CSmodels are defined in Section~\ref{sec:3-vars}
where we illustrate the nature of CStrees and decomposable CSmodels by presenting
a classification of all CStree models in three random variables.
A highlight from this section is Theorem~\ref{thm:three_perfect_iff_balanced}, which states that if the number $p$ of random variables equals $3$ then $\MM(\TT)$ is a decomposable CSmodel if and only if all of its minimal context DAGs
 are perfect. This is no 
longer true for $p=4$ by Example~\ref{ex:counterexample}.
In Section~\ref{sec:combinatorics}, we establish several combinatorial properties for balanced CStrees. Finally, Section~\ref{sec:algebra} contains the proofs of our main results. 

We assume the reader has some familiarity with polynomial ideals at the level of \cite[]{cox2015}. Although we introduce the basics of graphical models, we refer the reader to \cite[Ch. 3,4]{L96}, \cite[Ch. 3,4,5]{KF09}, \cite[Ch. 1,2,3]{maathuis:2018} for a more detailed  presentation. For a unified algebraic statistics perspective we suggest the book by Sullivant \cite[Ch.013]{S19}. Our methods rely heavily on properties of staged tree models. The reader may refer to \cite[]{CGS18} for a comprehensive introduction to this class of models.

\section{Preliminaries}\label{sec:prelim}
A discrete statistical model is a subset of the probability simplex. We consider
 models that are algebraic varieties intersected with the simplex. We are interested in finding their defining equations. We use the combinatorial properties of the equations to gain insight into the statistical properties of the model.

\subsection{Notation}
For any natural number $d$ we define $[d]\coloneqq\{1,2,\ldots,d\}$.
Consider a vector of discrete random variables
$X_{[p]}=(X_1,\ldots, X_p)$ where $p\in \mathbb{N}$  and for each $i\in [p]$, $d_i\in \mathbb{N}$, $[d_i]$ is the state space of  $X_i$ and $\RR=\prod_{i\in [p]}[d_i]$ is the state space of  $X_{[p]}$. 
Elements in $\RR$ are sequences $\xx=(x_1,\ldots, 
x_p)=x_1\cdots x_p$ where $x_k\in [d_k]$ for every $k\in [p]$. For $A\subset [p]$, $X_{A}$ is a subvector of discrete random variables with indices in  $A$, and $\RR_A=\prod_{i\in A}[d_i]$ is the state space of $X_A$. We also
use the bold notation $\xx$ or $\yy$ to denote elements in any marginal space
$\RR_{A}$ to avoid the excessive use of subscripts. At times it is
useful to recall which marginal space the outcome belongs to, in this case we write $\xx_A$ or $\yy_A$ for elements in $\RR_{A}$. We shall use the three notations $\xx$, $(x_1,\ldots, 
x_p)$, or $ x_1\cdots x_p$ throughout this article depending on whether it
is necessary to emphasize the outcomes $x_i\in \RR_{\{i\}}$
or not.  Whenever $\xx=(x_1,\ldots,x_p)\in \RR$, the element $\xx_A$ is the subvector $(x_i)_{i\in A}$ of $\xx$. For
two disjoint subsets $A,B\subset [p]$ and $\xx\in \RR$, the element
$\xx_{A}\xx_B=\xx_{A\cup B}$ is the subvector $(x_i)_{i\in A\cup B}$.
The notation $\xx_{A}\xx_B$ is reminiscent of concatenation of strings, nevertheless we use it to denote the element $\xx_{A\cup B}$ 
in $\RR_{A\cup B}$ so the order of $A\cup B$ is important. A probability distribution 
$f$ for $X_{[p]}$ is a tuple $(f(\xx): \xx\in\RR)$ where $f(\xx)>0$ and 
$\sum_{\xx \in \RR}f(\xx)=1$, $f(\xx)$ is the probability of the outcome
$\xx\in \RR$. The $|\RR|-1$ dimensional open probability simplex, denoted by $\Delta_{|\RR|-1}^{\circ}$ consists of all possible positive probability distributions for $X_{[p]}$.

To define a subvariety of the probability simplex we use the  polynomial ring $\R[D]:=\R[p_{\xx}\colon \xx\in \RR]$. For any subset $H \subset\R[D]$ we denote the algebraic variety
$\V(H)=\{x\in \C^{|\RR|} : g(x)=0 \text{ for all } g\in H \}$. The intersection $\V(H)\cap \Delta_{|\RR|-1}^{\circ}$ is a statistical model. In our situation, statistical models can also be defined as closed images of rational maps intersected with the probability simplex. 
A main question in algebraic statistics is to find the implicit equations that
define the parametrized model. In this statistical setting, the defining algebraic equations translate into restrictions on the distributions, these  encode  model assumptions.

\begin{example} \label{ex:adag}
Consider the case where $p=3$ and $X_1,X_2,X_3$ are binary random variables.
The graph $G=([3],1\to 2\to 3)$ imposes conditional independence relations among 
the $X_i$, $i\in[3]$. The state space of $(X_1,X_2,X_3)$ is $\RR=\{0,1\}^3$ because all random variables are binary,  hence $|\RR|=8$. The polynomial ring is $\R[D]=\R[p_{000},p_{001},p_{010},p_{011},p_{100},p_{101},p_{110},p_{111}]$. As we will see in the next section, $G$ tells us that $X_1$ is independent of $X_3$, given $X_2$, as there is no edge from $1$ to $3$. This CI relation translates into two equations, one for each outcome of the 
conditioning variable $X_2$:
\[
p_{100}p_{001}-p_{101}p_{000}=0,\quad p_{110}p_{011}-p_{111}p_{010}=0.
\]
These two equations, together with the hyperplane $\sum_{\xx\in\RR} p_\xx=1$, define a variety in 
the affine 8-dimensional space. Taking into account positivity conditions yields the model inside $\Delta_{7}^{\circ}$. By ignoring the sum-to-one hyperplane we immediately see that this model
defines a toric variety in $\mathbb{P}^7$ as it is cut out by a prime binomial ideal.
\end{example}

We now introduce the formal setup for context-specific independence. 

\subsection{ Context-specific conditional independence statements} \label{subsec:csis}
Let $A,B,C,S$ be disjoint subsets of $[p]$ and let $\xx_C\in \RR_{C}$. A distribution $f\in \Delta_{|\RR|-1}^{\circ}$ satisfies the \emph{context-specific conditional independence statement} (CSI statement)  $X_{A}\independent X_{B}|X_S,X_{C}=\xx_C$ if for all outcomes $(\xx_A,\xx_B,\xx_S)\in \RR_A\times \RR_B \times \RR_S$
 \[f(\xx_A|\xx_B,\xx_S,\xx_C)=f(\xx_A|\xx_S,\xx_C).\]
 To each CSI statement $X_A\independent X_B|X_S, \cc$ we associate the collection of polynomials
\begin{align}\label{eqn:associated-polynomial-to-CSI-statements}
    p_{\xx_A\xx_B\xx_S\xx_C+}p_{\yy_A\yy_B\xx_S\xx_C+}-p_{\xx_A\yy_B\xx_S\xx_C+}p_{\yy_A\xx_B\xx_S\xx_C+}
\end{align}
for every $\xx_A,\yy_A \in \RR_A$, $\xx_B,\yy_B \in \RR_B$ and  $\xx_S\in \RR_{S}$ where 
\[
p_{\xx_A\xx_B\xx_S\xx_C+}=\sum_{\zz\in\RR_{[p]\setminus (A\cup B\cup C\cup S)}} p_{\xx_A\xx_B\xx_S\xx_C\zz}.
\]
Note that these polynomials are the $2\times 2$ minors of the matrix $(p_{\xx_A\xx_B\xx_{S}\xx_C+})_{\xx_A\in\RR_A,\xx_B\in\RR_B}$ for all outcomes $\xx_S\in\RR_S$. We define the ideal $I_{X_A\independent X_B|X_S, \cc}$ in $\R[D]$ to be the ideal generated by all the polynomials in (\ref{eqn:associated-polynomial-to-CSI-statements}).

Given a collection $\CC$ of CSI statements, we define the CSI ideal generated by the polynomials associated to all CSI statements in $\CC$, i.e.
\[I_{\CC}=\sum_{X_A\independent X_B|X_S, \cc \in \CC} I_{X_A\independent X_B|X_S, \cc}.\]
 When the set $C$ in the conditioning of a CSI statement is empty, we recover the notion of 
 \textit{conditional independence} (CI) which are  statements of the form $X_{A}\independent X_{B}|X_S$. In this case $I_\CC$ is a
 conditional independence ideal, see \cite[Ch.4]{S19}. 
 Graphical models are a widely used class of CI models where the CI statements among random variables are captured by a  graph. More general CSI statements cannot be easily encoded using a single graph, to encode these, the staged tree model is more suitable. We define discrete graphical models associated to DAGs (also known as Bayesian networks) and staged tree models in the next section.

 \subsection{DAGs and staged trees}\label{sec:dags-and-staged-trees}
 A directed acyclic graph (DAG) $G$ is a pair 
 $([p],E)$ where~$[p]$ is the set of vertices and $E$ the set of directed edges such that there is no directed loop in~$G$. For any DAG $G$, we fix a topological ordering on its vertices, which means if $i\to j$ is an edge in $G$, then $i < j$. The DAG model $\MM(G)$ is the set of all the distributions $f\in\Delta_{|\RR|-1}^{\circ}$ that satisfy the recursive factorization property according to $G$. That is, for all $\xx\in \RR$ we~have
\begin{align*}
    p_{\xx}=f(\xx)=\prod_{v\in [p]}f(\xx_v|\xx_{\pa(v)}).
\end{align*}

On the other hand, DAG models can also be defined implicitly by using the polynomials associated to
 the CI statements via different Markov properties \cite[Section 1.8]{maathuis:2018}. The \textit{local Markov property} of a DAG $G$ is the collection of CI statements
$$\loc(G) = \{X_v \independent X_{\nd(v)}\, | \, X_{\pa(v)}: v \in V \},$$
where $\nd(v)$ denotes the set of all non-descendants of the node $v$ and $\pa(v)$ denotes the set of its parents in $G$. While  these local constraints contain information about independence relationships, they are often not enough to fully describe the set of all CI statements that hold in $G$.
Via the axioms of conditional independence, \cite[Prop. 4.1.4]{S19}, these local constraints may imply other CI statements. Hence, it is necessary to introduce the global Markov property. Any CI relation that holds for the model $\MM(G)$ is obtained from $G$ via the technical notion of $d$-separation statements \cite[Section 3.2.2]{L96}. The set of all $d$-separation statements, denoted by $\glo(G)$, defines the \textit{global Markov property} on $G$. The corresponding ideal $I_{\glo(G)}$ is not prime in general. It is, however, prime and generated by binomials if the graph is perfect (see Definition~\ref{def:perfect}).
In Example~\ref{ex:adag}, the local Markov property
and the global Markov property of $G$ coincide and the ideal
$I_{\loc(G)}=I_{\glo(G)}$ is prime and binomial.

While DAG models are well-suited to encode CI statements, they cannot encode CSI statements.
There are several models one could use instead to encode CSI statements.  In this paper, we focus on staged tree models. The definition of staged tree model we
present here is not as general as in \cite[]{SA08}. The reason for this choice is that we are interested only in representing CSI statements that hold for a vector of discrete random variables $X_{[p]}$. Thus the  staged trees we consider
represent the outcome space of $X_{[p]}$ as an event tree and we construct them as follows.

Let $\pi_1\cdots \pi_p$ be an ordering of $[p]$ and let $\TT=(V,E)$ be a rooted tree with
$V:=\{\mathrm{root}\}\cup \bigcup_{j\in[p]}\RR_{\{\pi_1\cdots \pi_j\}        }$ and set
of edges
\begin{align*}
E:=&\{\mathrm{root}\to x_{\pi_1}: x_{\pi_1}\in [d_{\pi_1}]\}\,\cup\\
&\{x_{\pi_1}\cdots x_{\pi_{k-1}}\to x_{\pi_1} \cdots x_{\pi_k}:x_{\pi_1}\cdots x_{\pi_{k-1}} \in \RR_{\pi_1,\ldots, \pi_{k-1}}, x_{k} \in [d_k],k\in [p]\}.
\end{align*}
Note that  elements in the set of vertices $V$ of $\TT$ are outcomes in the marginal outcome spaces $\RR_{\pi_1\cdots \pi_j}$
for $j\in [p]$. For several of the definitions and proofs 
we shall present,  when considering an element  $v\in V$, it is not 
always necessary to specify which marginal outcome space
contains $v$. 
Thus we simply write $v$ for an arbitrary element in $V$.

The level of a node $v\in V$ is the number of edges in the unique
path in $\TT$ from the root to $v$. The $k$th level of
$\TT$ is the set of all nodes in $\TT$ at level $k$ and is denoted by
$L_k$. For $\TT$ as defined, we see that $L_k$ is in bijection
with the outcome space $\RR_{\{\pi_1,\ldots, \pi_k\}}$  of the random vector $X_{\{\pi_1,\ldots,\pi_{k}\}}$.
Hence we identify the outcomes of the variable $X_{\pi_k}$ with the level $L_k$,
and we denote this association by $(L_1,\ldots, L_p)\sim (X_{\pi_1},\ldots, X_{\pi_p})$.
For any tree $\TT$ we write $V_\TT$ and $E_{\TT}$ for its sets of vertices and edges, respectively.
 We write $E(v)$ to denote the set of all outgoing edges from $v$. Without loss of generality, throughout this paper we will assume $\pi_i=i$ for all $i\in [p]$. In particular, this implies $(L_1,\ldots, L_p)\sim (X_{1},\ldots, X_{p})$. The tree in Figure~\ref{fig:example} represents an event tree for a vector $(X_1,X_2,X_3)$ of binary random variables. To illustrate part of the notation, in this tree we have $L_{2}=\{00,01,10,11\}$, which is exactly the marginal outcome space $\RR_{\{1,2\}}$  and $E(0):=\{0\to 00, 0\to 01\}$.
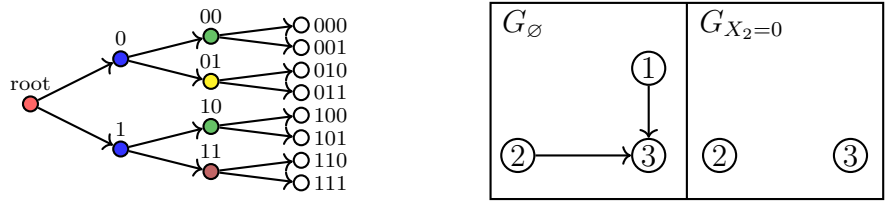
\begin{figure}
\centering
    \begin{subfigure}{}
    \begin{tikzpicture}[thick,scale=0.3]

 	 \node[circle, draw, fill=black!0, inner sep=2pt, minimum width=1pt] (w3) at (0,0)  {};
 	 \node[circle, draw, fill=black!0, inner sep=2pt, minimum width=1pt] (w4) at (0,-1) {};
 	 \node[circle, draw, fill=black!0, inner sep=2pt, minimum width=1pt] (w5) at (0,-2) {};
 	 \node[circle, draw, fill=black!0, inner sep=2pt, minimum width=1pt] (w6) at (0,-3) {};
 	 \node[circle, draw, fill=black!0, inner sep=2pt, minimum width=1pt] (v3) at (0,-4)  {};
 	 \node[circle, draw, fill=black!0, inner sep=2pt, minimum width=1pt] (v4) at (0,-5) {};
 	 \node[circle, draw, fill=black!0, inner sep=2pt, minimum width=1pt] (v5) at (0,-6) {};
 	 \node[circle, draw, fill=black!0, inner sep=2pt, minimum width=1pt] (v6) at (0,-7) {};
        
        \node[] () at (1.3,0) {\tiny{$000$}};
        \node[] () at (1.3,-1) {\tiny{$001$}};
        \node[] () at (1.3,-2) {\tiny{$010$}};
        \node[] () at (1.3,-3) {\tiny{$011$}};
        \node[] () at (1.3,-4) {\tiny{$100$}};
        \node[] () at (1.3,-5) {\tiny{$101$}};
        \node[] () at (1.3,-6) {\tiny{$110$}};
        \node[] () at (1.3,-7) {\tiny{$111$}};

	    \node[] () at (-4,0.4) {\tiny{$00$}};
	   \node[circle, draw, fill=green!60!black!60, inner sep=2pt, minimum width=2pt] (w1) at (-4,-.5) {};
        \node[] () at (-4,-1.6) {\tiny{$01$}};
 	 \node[circle, draw, fill=yellow!90, inner sep=2pt, minimum width=2pt] (w2) at (-4,-2.5) {}; 
        \node[] () at (-4,-3.6) {\tiny{$10$}};
 	 \node[circle, draw, fill=green!60!black!60, inner sep=2pt, minimum width=2pt] (v1) at (-4,-4.5) {};
        \node[] () at (-4,-5.6) {\tiny{$11$}};
 	 \node[circle, draw, fill=red!60!black!60, inner sep=2pt, minimum width=2pt] (v2) at (-4,-6.5) {};

        \node[] () at (-8,-0.6) {\tiny{$0$}};
 	 \node[circle, draw, fill=blue!80, inner sep=2pt, minimum width=2pt] (w) at (-8,-1.5) {};
        \node[] () at (-8,-4.6) {\tiny{$1$}};
 	 \node[circle, draw, fill=blue!80, inner sep=2pt, minimum width=2pt] (v) at (-8,-5.5) {};	

        \node[] () at (-12,-2.6) {\tiny{root}};
     \node[circle, draw, fill=red!60, inner sep=2pt, minimum width=2pt] (r) at (-12,-3.5) {};


 	 \draw[->]   (r) --   (w) ;
 	 \draw[->]   (r) --   (v) ;

 	 \draw[->]   (w) --  (w1) ;
 	 \draw[->]   (w) --  (w2) ;

 	 \draw[->]   (w1) --   (w3) ;
 	 \draw[->]   (w1) --   (w4) ;
 	 \draw[->]   (w2) --  (w5) ;
 	 \draw[->]   (w2) --  (w6) ;

 	 \draw[->]   (v) --  (v1) ;
 	 \draw[->]   (v) --  (v2) ;

 	 \draw[->]   (v1) --  (v3) ;
 	 \draw[->]   (v1) --  (v4) ;
 	 \draw[->]   (v2) --  (v5) ;
 	 \draw[->]   (v2) --  (v6) ;

    \end{tikzpicture}
    \end{subfigure}
    \hspace{0.5cm}
    \begin{subfigure}{}
    \begin{tikzpicture}[thick,scale=0.29]
	\draw (-1,3) -- (17,3) -- (17, -6) -- (-1, -6) -- (-1,3) -- cycle;
	\draw (8,3) -- (8,-6) ;

 	 \node[circle, draw, fill=black!0, inner sep=1pt, minimum width=1pt] (1v1) at (6.25,0) {$1$};
 	 \node[circle, draw, fill=black!0, inner sep=1pt, minimum width=1pt] (1v2) at (0.25,-4) {$2$};
 	 \node[circle, draw, fill=black!0, inner sep=1pt, minimum width=1pt] (1v3) at (6.25,-4) {$3$};

 	 \node[circle, draw, fill=black!0, inner sep=1pt, minimum width=1pt] (2v2) at (9.5,-4) {$2$};
 	 \node[circle, draw, fill=black!0, inner sep=1pt, minimum width=1pt] (2v3) at (15.5,-4) {$3$};


 	\draw[->]   (1v1) -- (1v3) ;
    \draw[->]   (1v2) -- (1v3) ;
 	 
	 \node at (-3,2) {} ;
	 \node at (0.5,2) {$G_{\emptyset}$} ;
	 \node at (10.5,2) {$G_{X_2=0}$} ;
    \end{tikzpicture}
    \end{subfigure}
    \caption{\small{A CStree for $p=3$ and its minimal context DAGs.}}
    \label{fig:example}
\end{figure}

Let $\TT=(V,E)$ be a rooted tree with levels $(L_1,\cdots,L_p)\sim (X_{1},\ldots, X_{p})$, $\LL$ a finite set of labels
 and $\theta:E\to \LL$ a labeling of the edges. The pair $(\TT,\theta)$
 is a \textit{staged tree} if 
 \begin{itemize}
     \item[(1)] $|\theta(E(v))|=|E(v)|$ for all $v\in V$, and
     \item[(2)] for any pair $v,w\in V$, $\theta(E(v))$ and
     $\theta(E(w))$
     are either equal or disjoint.
 \end{itemize}
Two vertices $v,w$ in $(\TT,\theta)$ are in the same stage if and only
if $\theta(E(v))=\theta(E(w))$. In this case we write $v\sim w$. The
equivalence relation $\sim$ on the set $V$ induces
a partition of $V$ called the \textit{staging} of $\TT$. We refer to each set in this partition
as a \emph{stage}. When depicting staged trees, such as in Figure~\ref{fig:example}, we use colors in the vertices
to indicate that two vertices are in the same stage, except for white vertices 
which always represent singleton stages. Intuitively, the vertices in a 
stage $S$ in level $L_{k-1}$ represent outcomes $x_1\cdots x_{k-1}$ for which the conditional
distributions $f(X_k|x_1\cdots x_{k-1}), x_1\cdots x_{k-1}\in S$ are all equal. In all staged trees we consider, the leaves of the tree and the root of the tree are always singleton stages, thus we often omit them when talking about the partition of $V$ into stages.
\begin{definition} \label{def:staged_tree}
Let $\TT$ be a staged tree. The staged tree model $\MM(\TT)$ has the space of parameters
\[
\Theta_{\TT}=\left\{ x\in \R^{|\LL|}\colon \forall e\in E, x_{\theta(e)}\in (0,1) \text{ and } \forall v\in V, \sum_{e\in E(v)} x_{\theta(e)}=1\right\}
\]
 and is defined to be the image of the map
 \[
 \Psi_{\TT}\colon \Theta_{\TT}\to \Delta_{|\RR|-1}^{\circ},\quad x\mapsto \left(\prod_{e\in E(\mathrm{root}\to \xx)}x_{\theta(e)}\right)_{\xx\in\RR}
 \]
 where $E(\mathrm{root}\to\xx)$ denotes the set of all edges on the path from the root to $\xx\in \RR$.
 We say a distribution $f\in \Delta_{|\RR|-1}^{\circ}$ factors according to $\TT$ if $f\in \mathrm{im}(\Psi_{\TT})$.
\end{definition}
\begin{example} \label{ex:astagedtreemodel}
    Consider the staged tree depicted in Figure~\ref{fig:example} (left) with levels $(L_1,L_2,L_3)\sim (X_1,X_2,X_3)$. It represents
    the event tree for a vector $(X_1,X_2,X_3)$
    of binary random variables with ordering $\pi=123$.
    The set of stages in this tree is
    $\{\{0,1\},\{00,10\},\{01\},\{11\}\}$.
    The stage $\{0,1\}$ is the set of blue nodes and the stage $\{00,10\}$
    is the set of green nodes in Figure~\ref{fig:example}. The blue stage encodes
    the equality $f(X_2|X_1=0)=f(X_2|X_1=1)$ and the green
    stage encodes the equality $f(X_3|X_{12}=00)=f(X_3|X_{12}=10)$. These equalities 
    of conditional probabilities correspond to the set of CSI
    statements $\CC=\{ X_2\independent X_1,X_3\independent X_1|X_2=0\}$. In this case
    $\MM(\TT)=\V(I_\CC)\cap \Delta_{7}^{\circ}$.
\end{example}

Since CI statements are collections of CSI statements when the context $C$ in Eq. (\ref{eqn:associated-polynomial-to-CSI-statements}) is empty, it is natural to see that DAG models are a particular type of staged tree model. In particular, every DAG model has a staged tree representation; see Example~\ref{ex:stagedtree_ofDAG} and \cite[Section 2.1]{DS21}. 
The Definition~\ref{def:staged_tree} of staged tree model
allows for very flexible types of stagings. In the
next section we define CStree models which are a class of staged tree models for which the staging has to satisfy additional conditions. 

\subsection{CStrees}\label{sec:cstrees}
In this section we introduce CStree models as a subclass of staged tree models.
\begin{definition} \label{def:cstree}
The staged tree $(\TT,\theta)$ is a \textit{CStree} if
 \begin{itemize}
   \item[(1)] $\theta(E(v))\neq \theta(E(w))$ if $v,w$ are in different levels,
   \item[(2)] $\theta(x_{1}\cdots x_{k-1}\to x_{1}\cdots x_{k-1}x_{k})=\theta(y_1\cdots y_{k-1}\to y_{1}\cdots y_{k-1}x_{k})$
     whenever the nodes $x_{1}\cdots x_{k-1}$ and $y_{1}\cdots y_{k-1}$ are in the same stage.
   \item[(3)] For every stage $S\subset L_{k-1}, k\in [p]$, there exists $C\subset [k-1]$ and $\xx_C\in \RR_{C}$ such that 
   \[S=\bigcup_{\yy\in \RR_{[k-1]\setminus C}}\{\xx_{C}\yy\}.\]
 \end{itemize}
\end{definition}
Condition (1) ensures that two nodes in different levels do not share
the same conditional distribution. Condition (2) forces that edges with the
same label must point to the same outcome of $X_k$. Condition $(3)$ restricts the types of CSI statements that can be encoded simultaneously in a CStree to be those described in Lemma~\ref{independence-statements-lemma}. 
  
To describe a CStree model $\MM(\TT)$ as an algebraic variety intersected with the open probability simplex we consider the ring homomorphism
\begin{align}\label{eq:psi-t}
    \psi_\TT\colon \R[D]\to\R[\Theta_\TT],\quad p_{\xx}\mapsto \prod_{e\in E(\text{root}\to \xx)}\theta(e)
\end{align}
where $\R[\Theta_{\TT}]:=\R[\theta(e)\colon e\in E]/\langle \theta-1 \rangle$ and $\langle \theta-1 \rangle:=\langle \sum_{e\in E(v)}\theta(e)-1:v\in V\rangle$ is the ideal representing the sum-to-one conditions on the parameter space. 
The ring map $\psi_{\TT}$ is the pullback of the parametrization $\Psi_{\TT}$ of the model $\MM(\TT)$ in Definition~\ref{def:staged_tree}. Using $\psi_{\TT}$ , we can write the CStree model as
$$\MM(\TT)=\V(\ker(\psi_{\TT}))\cap \Delta_{|\RR|-1}^{\circ}.$$
\begin{example}(CStree representation of a DAG model) \label{ex:stagedtree_ofDAG}
    Let $G=([p],E)$ be a DAG. The model $\MM(G)$
    is the set of all distributions in $\Delta_{|\RR|-1}^{\circ}$ that satisfy the recursive factorization property according to $G$. Following \cite[Section 2.1]{DS22}, to represent $\MM(G)$ as a CStree model we first fix an ordering
    $\pi$ of $[p]$ and then consider the
    tree $\TT_{G}$ with levels $(L_1,\ldots,L_p)\sim (X_{\pi_1},\ldots,X_{\pi_p})$. It remains to specify
    the labelling of $(\TT_{G},\theta)$. The labelling is completely determined
    by the staging, hence we specify the stages
     for each level $L_{k-1}$
    where $k\in [p]$. Fix $k\in [p]$, then $L_{k-1}$ has one
    stage $S_{\xx_{\pa(k)}}=\{\xx_{\pa(k)}\yy: \yy\in \RR_{[k-1]\setminus \pa(k)}\}$ for each outcome $\xx_{\pa(k)}\in \RR_{\pa(k)}$. Note that for each $k\in [p]$, and for each $\xx_{\pa(k)}\in \RR_{\pa(k)}$, $S_{\xx_{\pa(k)}}$
    satisfies condition (3) in the definition of a CStree, hence
    $(\TT_{G},\theta)$ is a CStree and $\MM(G)=\MM(\TT_G)$. We highlight the difference between these two combinatorial representations of the same model; while $G$ is a DAG whose vertices are random variables, $\TT_{G}$ is a directed tree whose leaves are the outcome space of the joint distribution of the random variables, which are the nodes, in $G$. The model restrictions in  $G$ are encoded via absence of edges, the same model restrictions in $\TT_G$ are encoded via stages. The map
    $\Psi_{\TT_G}$ from Definition~\ref{def:staged_tree} is
    equal to the recursive factorization according to $G$
    in the beginning of Section~\ref{sec:dags-and-staged-trees}.
    To illustrate this
    construction, consider the DAG model from Example~\ref{ex:adag}.
    In this case the ordering of the binary random variables is $(X_1,X_2,X_3)$ and the tree $\TT_{G}$ (disregarding the coloring of the vertices)
    is represented in Figure~\ref{fig:example}.
    The stages of $\TT_{G}$ in level $L_1$ are $\{0\},\{1\}$ and in $L_2$
    two these are $\{00,10\},\{01,11\}$.
\end{example}
 \begin{remark}\label{rmk:DAG-representation-remark}
For a DAG $G$, the map $\psi_{\TT_G}$ is the algebraic version of the well-known recursive factorization according to $G$.
 It is important to note that the ideal $I_{\glo(G)}$ is not always equal to the prime ideal
$\ker(\psi_{\TT_G})$. The article \cite[]{GSS05} contains several examples where equality holds as well as numerous counterexamples. The strongest
possible algebraic characterization of a model is to find the generators for 
$\ker(\psi_{\TT})$. For most graphical models, discrete and Gaussian, it is an open question to find generators of $\ker(\psi_{\TT})$. A recent overview of the state of the art is presented in \cite[Chapter 3]{maathuis:2018}. For discrete decomposable DAG models such a characterization
can be found in \cite[Theorem 4.4]{GMS06}. Our Theorem~\ref{thm:algebra-saturated-statements} characterizes $\ker(\psi_{\TT})$ in terms of CSI statements for all balanced CStree models as defined in \ref{def:balanced}.
 \end{remark}

 The next lemma describes the type of CSI statements encoded by a CStree; they are a consequence of condition $(3)$ in Definition~\ref{def:cstree}.
 
 \begin{lemma}{\cite[Lemma 3.1]{DS22}}\label{independence-statements-lemma}
 Let $\TT$ be a CStree with levels $(L_1,\ldots,L_p)\\ \sim(X_1,\ldots,X_p)$. Then for any $f\in\MM(\TT)$ and stage $S\subset L_{k-1}$,  condition $(3)$ in Definition~\ref{def:cstree} implies that $f$ entails the CSI statement $X_k\independent X_{[k-1]\setminus C}|X_{C}=\xx_C$ where $C$ is the set of all indices $\ell$ such that all elements in $S$ have the same  outcome for $X_\ell$. Hence, $\xx_C=\yy_C$
 for any $\yy\in S$.
 \end{lemma}

For any stage $S$ the context $\cc$ in Lemma~\ref{independence-statements-lemma} is called the \emph{stage-defining context} of the stage $S$. Given a stage defining context $\cc$ for a stage $S$ in level $L_{k-1}$ we recover the stage from the statement $X_k\independent X_{[k-1]\setminus C}|X_{C}=\xx_C$ as $S=\bigcup_{\yy\in \RR_{[k-1]\setminus C}}\{\xx_{C}\yy\}$.

\subsection{CStrees as collections of context DAGs} \label{subsec:contextDAGs}
An important question in the study of conditional independence statements
is to understand when a combination of CI statements implies additional
CI statements. The rules to deduce new statements are called 
conditional independence axioms \cite[Ch.4]{S19}. For instance, applying
these axioms to the local Markov property that holds for a DAG $G$ results in new
CI statements that hold for $G$ and the global Markov property encompasses all
CI statements that hold for $G$. The same question arises in the study
of CSI statements. Namely, what are all CSI statements that are implied
by a given set of CSI statements? The rules to deduce new CSI
statements from existing ones are the context-specific independence axioms, or CSI axioms for short. We use here the CSI axioms presented in
\cite[Section 3.2]{DS22} and point out that an axiomatic study of
CSI statements was carried out before in \cite[]{corander:2016}.
\begin{enumerate}
	\item \emph{symmetry.} 
 $X_A\independent X_B \mid X_S, X_C = \xx_C$ $\Longrightarrow$ $X_B\independent X_A \mid X_S, X_C = \xx_C$.
	\item \emph{decomposition.} 
 $ X_A\independent X_{B\cup D} \mid X_S, X_C = \xx_C$ $\Longrightarrow$  $ X_A\independent X_B \mid X_S, X_C = \xx_C$.
	\item \emph{weak union.} If $ X_A\independent X_{B\cup D} \mid X_S, X_C = \xx_C$ $\Longrightarrow$  $ X_A\independent X_B \mid X_{S\cup D}, X_C = \xx_C$. 
	\item \emph{contraction.} If $ X_A\independent X_B \mid X_{S\cup D}, X_C = \xx_C$ and $ X_A\independent X_D \mid X_S, X_C = \xx_C$ $\Longrightarrow$  $ X_A\independent X_{B\cup D} \mid X_S, X_C = \xx_C$. 
	\item \emph{intersection.} If $ X_A\independent X_B \mid X_{S\cup D}, X_C = \xx_C$ and $ X_A\independent X_S \mid X_{B\cup D}, X_C = \xx_C$ $\Longrightarrow$  $ X_A\independent  X_{B\cup S} \mid X_D, X_C = \xx_C$.  
	\item \emph{specialization.} If $ X_A\independent X_B \mid X_S, X_C = \xx_C$, $T\subseteq X_S$ and $\xx_T\in\RR_T$, $\Longrightarrow$  $ X_A\independent X_B \mid X_{S\setminus T}, X_{T\cup C} = \xx_{T\cup C}$.
	\item \emph{absorption.} If $ X_A\independent X_B \mid X_S, X_C = \xx_C$, and there exists $T\subseteq C$ for which $ X_A\independent X_B \mid X_S, X_{C\setminus T} = \xx_{C\setminus T}, X_T = \xx_T$ for all $\xx_T \in \RR_T$, $\Longrightarrow$  $ X_A\independent X_B \mid X_{S\cup T}, X_{C\setminus T} = \xx_{C\setminus T}$. 
\end{enumerate}
The first five axioms are a direct generalization
of the CI axioms. The specialization axiom (6) says that whenever you have a vector $X_S$ in the conditioning set you can specialize it
to a CSI relation by choosing an outcome $\xx_T\in \RR_{T}$.
The absorption axiom (7) is the opposite of specialization, it says
that if you have a collection of CSI statements such that,
in the conditioning contexts, for a certain subset $T$, all outcomes
of $X_T$ are present, then this turns into a CSI statement including
the vector $X_T$ in the conditioning set.
\begin{example}
    Consider a DAG $G=([p],E)$ and $k\in [p]$. In the staged tree
    representation $\TT_G$ of $G$ presented in Example~\ref{ex:stagedtree_ofDAG}, for each $k\in [p]$ there is one stage in level $L_{k-1}$ for each outcome $\xx_{\pa(k)}\in X_{\pa(k)}$.
    Using Lemma~\ref{independence-statements-lemma},
    such stage entails the CSI statement $X_k\independent X_{[k-1]\setminus \pa(k)}|X_{\pa(k)}=\xx_{\pa(k)}$. Thus
    the stages of $\TT_G$ in level $L_{k-1}$ correspond to
    the set of CSI statements $\{X_k\independent X_{[k-1]\setminus \pa(k)}|X_{\pa(k)}=\xx_{\pa(k)}:\xx_{\pa(k)}\in \RR_{\pa(k)}\}$.
    Using the absorption axiom this implies the CI statement
    $X_k\independent X_{[k-1]\setminus \pa(k)}|X_{\pa(k)}$. Note that
    in turn, each statement of the form $X_k\independent X_{[k-1]\setminus \pa(k)}|X_{\pa(k)}=\xx_{\pa(k)}$ is a specialization of
    $X_k\independent X_{[k-1]\setminus \pa(k)}|X_{\pa(k)}$.
\end{example}
Let $\J(\TT)$ be the set of all CSI statements implied by applying the CSI axioms \cite[Section 3.2]{DS22} to the statements in Lemma~\ref{independence-statements-lemma}. 
By the absorption axiom \cite[Lemma 3.2]{DS22}, there exists a collection $\CC_\TT:=\{\cc\}$ of contexts such that for any $X_A\independent X_B|X_S, \cc\in\J(\TT)$ with $\cc\in\CC_\TT$, there is no subset $T\subseteq C$ for which $X_A\independent X_B|X_{S\cup T}, X_{C\setminus T}=\xx_{C\setminus T}\in\J(\TT)$. We call such $\cc\in\CC_\TT$ a \textit{minimal context} of $\TT$. Also, from \cite[Lemma 3.2]{DS22}, it follows that$$\J(\TT)=\bigcup_{\cc\in\CC_\TT}\J(\cc)$$ where $\J(\cc)$ is the set of all CI statements of the form $X_A\independent X_B|X_S$ that hold in the context $\cc$. This equality establishes that the set of 
all CSI statements that are implied by $\TT$ using the CSI axioms,  is equal to the union of all the
statements that are implied, using the CSI axioms,  by the set of
statements in each of the minimal contexts.

We now construct a DAG associated to each minimal context $\cc\in\CC_\TT$ using a \emph{minimal} I-MAP \cite[Section 3.4.1]{KF09}. We say that a DAG $H$ is an I-MAP for
a set of CI statements $\I$ if the set of all CI statements implied 
by $G$,
denoted $\I(G)$, is contained in $\I$, \cite[Section 3.2.3]{KF09}. A DAG $H$ is a minimal
I-MAP for $\I$ if $H$ is an I-MAP of $\I$ and if the removal of a single
edge from $H$ renders it not an I-MAP. To each such $\cc\in\CC_\TT$, we can associate a \textit{minimal context DAG} with set of ordered nodes $[p]\setminus C$,
denoted by $G_{\cc}$, via a minimal I-MAP  of $\J(\cc)$ \cite[Section 3.2]{DS22}; this ordering should be a restriction of the ordering of the variables $X_{[p]}$
in $\TT$.  A construction of a minimal I-MAP is explained in \cite[Algorithm 3.2]{KF09}.
\begin{example}\label{ex:contextDAGsIMAP}
Consider the binary CStree $\TT$ on three random variables with ordering
$123$ given in Figure~\ref{fig:example}.
Since the two nodes in level $L_1$ are in the same stage (represented by the same colors), this CStree implies the CI statement $X_1\independent X_2$. As the nodes $00$ and $10$ are in the same stage in level $L_2$, but $01$ and $11$ are not, we get the CSI statement $X_3\independent X_1|X_2=0$. Therefore, $\J(\TT)=\{X_1\independent X_2,\, X_1\independent X_3|X_2=0\}$. Especially, we see that the set of minimal contexts $\CC$
is $\{\emptyset,X_2=0\}$ and $\J(\TT)=\J(\emptyset)\cup\J(X_2=0)$ where
$\J(\emptyset)=\{X_1\independent X_2\}$ and $\J(X_2=0)=\{X_1\independent X_3|X_2=0\}$. The minimal I-MAP of $\J(\emptyset)$ is a DAG with nodes $\{1,2,3\}$
and ordering $123$ which entails $X_1\independent X_2$. This is exactly
the DAG $G_{\emptyset}$ in Figure~\ref{fig:example}. Similarly,
the minimal I-MAP of $\J(X_2=0)$ is a DAG on the set of nodes
$\{1,3\}$, with ordering $13$, that entails $X_1\independent X_3$; this
 DAG has two nodes and no edge between them and is displayed as $G_{X_2=0}$ in Figure~\ref{fig:example}.
\end{example}

Each context DAG $G_{X_{C}=\xx_C}$ is in particular a DAG, thus by Example~\ref{ex:stagedtree_ofDAG}, it has a staged tree representation which we denote by $\TT_{G_{X_{C}=\xx_C}}$.
To relate $\TT_{G_{X_{C}=\xx_C}}$ to the original tree $\TT$, we define a
\textit{context subtree}
$\TT_{X_{C}=\xx_C}$ for each context $X_{C}=\xx_C$ .
Let $x_1\cdots x_k\in \TT$ and denote by $\TT_{x_1\cdots x_k}$ the directed subtree of $\TT$ with root node $x_1\cdots x_k$. 
For $C\subset [p]$ and $\xx_{C}\in \RR_{C}$ we construct the context subtree $\TT_{X_{C}=\xx_C}=(V_{X_{C}=\xx_{C}},E_{X_{C}=\xx_{C}})$ by
deleting all subtrees $\TT_{x_1\cdots x_k}$ and all edges $x_1\cdots x_{k-1}\to x_1\cdots x_{k}$ with $x_{k}\neq \xx_{C\cap{k}}$, and then contracting the edges
$x_1\cdots x_{k-1}\to x_1\cdots x_{k-1}(\xx_{C})_k$ for all
$x_1\cdots x_{k-1}\in \RR_{[k-1]}$, for all $k\in C$. The single node resulting from this contraction is labeled $x_1\cdots x_{k-1}(\xx_C)_k$ and it is
in the same stage as 
$x_1\cdots x_{k-1}\xx_{C\cap\{k\}}$
in $\TT$. All of the other nodes in $\TT_{X_{C}=\xx_{C}}$ inherit their staging from $\TT$. Note that
the context subtree $\TT_{\cc}$ is itself a CStree and $\TT_{x_1\cdots x_k}$ is the context subtree $\TT_{X_{[k]}=x_1\cdots x_{k}}$.

Moreover, let $\cc$ be a minimal context of $\TT$. Then by the construction of $\TT_{\cc}$ 
the CSI statements that hold in $\TT_{\cc}$ are given as
\[
\J(\TT_{\cc})=\{X_A\independent X_B|X_S, X_D=\xx_D\colon X_A\independent X_B|X_S, X_D=\xx_D, \cc\in\J(\TT)\}
\]
with $A,B,S,D\subset [p]\setminus C$.
Similarly, the CI statements on $X_{[p]\setminus C}$ implied by the DAG $G_{\cc}$ are
\[
\J(\cc)=\{X_A\independent X_B|X_S\colon X_A\independent X_B|X_S, \cc\in\J(\TT)\}
\]
which are precisely the CI statements valid in $\TT_{\cc}$.
This shows that the CI statements implied by the DAG $G_{\cc}$ are exactly the CI statements implied by the CStree $\TT_{\cc}$.
In general, the CStrees $\TT_{G_{X_{C}=\xx_C}}$ and $\TT_{X_C=\xx_C}$ are
different, since the CStree $\TT_{X_C=\xx_C}$ may imply more CSI statements, see Example~\ref{ex:small-context-tree-and-dag-example} and Example~\ref{ex:context_subtree}.
 If $\varnothing\in \CC_{\TT}$ then $G_{\varnothing}$
is a DAG that captures the CI relations implied
by $\TT$. When $\varnothing\notin C_{\TT}$, then $\TT$ entails no CI
relations, in this case we associate to $\TT$ the complete DAG
on $[p]$ nodes whose directed arrows are in agreement with the
causal ordering of $\TT$, we also denote this DAG by $G_\varnothing$. 
\begin{example} \label{ex:small-context-tree-and-dag-example}
    We illustrate the construction of the context subtree $\TT_{\cc}$ and
    the staged tree representation of a context DAG $\TT_{G_{\cc}}$ using
    the CStree in Example~\ref{ex:contextDAGsIMAP}. For the empty minimal context,  $\TT_{\emptyset}$ is equal to the
    original tree $\TT$. However, $\TT_{G_{\emptyset}}$
    is a CStree with the same vertices and edges of $\TT$ but
    with set of stages equal to $\{
    \{0,1\},\{00\},\{01\},\{10\},\{11\}\}$. Importantly, $\TT_{\emptyset}\neq \TT_{G_{\emptyset}}$. For the minimal context $X_2=0$,
    $\TT_{X_2=0}$ is a staged tree with two levels whose
    staging is the same as the staging of the tree $\TT_{G_{X_2=0}}$ on two
    levels associated to $G_{X_2=0}$.
\end{example}

\begin{example}
For the sake of intuition we present an example in Figure~\ref{fig:not_CStree} of a collection of context DAGs that do \textit{not} define a CStree. Consider the two DAGs $G_\emptyset=([3],\{2\to 3\})$ and $G_{X_1=0}=(\{2,3\},\emptyset)$ and assume all random variables are binary. These two DAGs imply the CI relation $X_1\independent X_{2,3}$ and the CSI relation $X_2\independent X_3|X_1=0$.

Let $\TT$ be a staged tree with levels $(L_1,L_2,L_3)\sim (X_1,X_2,X_3)$. We will see that the staging of $\TT$ cannot be a CStree. The vertices $10$ and $00$ are in the same stage because in the empty context DAG, $G_{\varnothing}$,
$f(X_3|X_{1,2}=10)=f(X_3|X_{1,2}=00)$. Moreover, since $X_2\independent X_3 |X_1=0$, we also have
$f(X_3|X_{1,2}=00)=f(X_3|X_{1,2}=01)$, thus
 $00$ and $01$ are also in the same stage. This implies $10,00,01$ are all in the same stage, which by definition of CStree implies that so are all vertices $00$, $01$, $10$, $11$. Thus the CI statement $X_3\independent X_{1,2}$ holds in the CStree. However, this statement is not implied by the two DAGs. 
\end{example}

\begin{figure}
\centering
    \begin{subfigure}{}
    \begin{tikzpicture}[thick,scale=0.3]

 	 \node[circle, draw, fill=black!0, inner sep=2pt, minimum width=1pt] (w3) at (0,0)  {};
 	 \node[circle, draw, fill=black!0, inner sep=2pt, minimum width=1pt] (w4) at (0,-1) {};
 	 \node[circle, draw, fill=black!0, inner sep=2pt, minimum width=1pt] (w5) at (0,-2) {};
 	 \node[circle, draw, fill=black!0, inner sep=2pt, minimum width=1pt] (w6) at (0,-3) {};
 	 \node[circle, draw, fill=black!0, inner sep=2pt, minimum width=1pt] (v3) at (0,-4)  {};
 	 \node[circle, draw, fill=black!0, inner sep=2pt, minimum width=1pt] (v4) at (0,-5) {};
 	 \node[circle, draw, fill=black!0, inner sep=2pt, minimum width=1pt] (v5) at (0,-6) {};
 	 \node[circle, draw, fill=black!0, inner sep=2pt, minimum width=1pt] (v6) at (0,-7) {};

	   \node[circle, draw, fill=green!60!black!60, inner sep=2pt, minimum width=2pt] (w1) at (-4,-.5) {};
 	 \node[circle, draw, fill=green!60!black!60, inner sep=2pt, minimum width=2pt] (w2) at (-4,-2.5) {}; 
 	 \node[circle, draw, fill=green!60!black!60, inner sep=2pt, minimum width=2pt] (v1) at (-4,-4.5) {};
 	 \node[circle, draw, fill=red!60!black!60, inner sep=2pt, minimum width=2pt] (v2) at (-4,-6.5) {};	 
		
 	 \node[circle, draw, fill=blue!80, inner sep=2pt, minimum width=2pt] (w) at (-8,-1.5) {};
 	 \node[circle, draw, fill=blue!80, inner sep=2pt, minimum width=2pt] (v) at (-8,-5.5) {};	

     \node[circle, draw, fill=red!60, inner sep=2pt, minimum width=2pt] (r) at (-12,-3.5) {};


 	 \draw[->]   (r) --   (w) ;
 	 \draw[->]   (r) --   (v) ;

 	 \draw[->]   (w) --  (w1) ;
 	 \draw[->]   (w) --  (w2) ;

 	 \draw[->]   (w1) --   (w3) ;
 	 \draw[->]   (w1) --   (w4) ;
 	 \draw[->]   (w2) --  (w5) ;
 	 \draw[->]   (w2) --  (w6) ;

 	 \draw[->]   (v) --  (v1) ;
 	 \draw[->]   (v) --  (v2) ;

 	 \draw[->]   (v1) --  (v3) ;
 	 \draw[->]   (v1) --  (v4) ;
 	 \draw[->]   (v2) --  (v5) ;
 	 \draw[->]   (v2) --  (v6) ;

    \end{tikzpicture}
    \end{subfigure}
    \hspace{0.5cm}
    \begin{subfigure}{}
    \begin{tikzpicture}[thick,scale=0.29]
	\draw (-1,3) -- (17,3) -- (17, -6) -- (-1, -6) -- (-1,3) -- cycle;
	\draw (8,3) -- (8,-6) ;

 	 \node[circle, draw, fill=black!0, inner sep=1pt, minimum width=1pt] (1v1) at (6.25,0) {$1$};
 	 \node[circle, draw, fill=black!0, inner sep=1pt, minimum width=1pt] (1v2) at (0.25,-4) {$2$};
 	 \node[circle, draw, fill=black!0, inner sep=1pt, minimum width=1pt] (1v3) at (6.25,-4) {$3$};

 	 \node[circle, draw, fill=black!0, inner sep=1pt, minimum width=1pt] (2v2) at (9.5,-4) {$2$};
 	 \node[circle, draw, fill=black!0, inner sep=1pt, minimum width=1pt] (2v3) at (15.5,-4) {$3$};


    \draw[->]   (1v2) -- (1v3) ;
 	 
	 \node at (-3,2) {} ;
	 \node at (0.5,2) {$G_{\emptyset}$} ;
	 \node at (10.5,2) {$G_{X_1=0}$} ;
    \end{tikzpicture}
    \end{subfigure}
    \caption{Staged tree that is not a CStree.}
    \label{fig:not_CStree}
\end{figure}
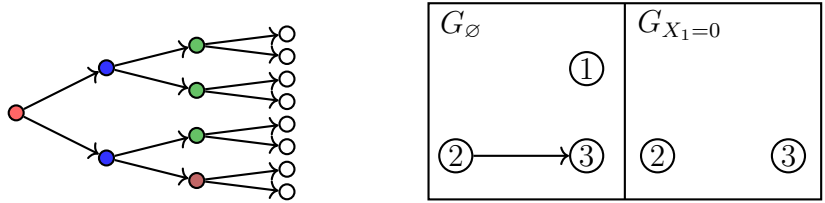

\subsection{Balanced CStrees}
\label{subsec:dCSmodels}

Decomposable graphical models are a set of graphical models for which the undirected and directed Markov properties coincide. These are characterized in many different ways: combinatorially as chordal UGs or as perfect DAGs, and geometrically as those DAG models that are discrete exponential families \cite[]{geiger:2001}, also known as toric models in the algebraic statistics literature \cite[]{S19}. 
The article \cite[]{DS22} suggests the family of balanced staged tree models as a suitable generalization of decomposable DAG models because these models are discrete exponential families. Furthermore, a DAG is perfect if and only if its CStree representation is balanced. Thus we identify the class of balanced CStrees as a good candidate for decomposable models in the context-specific setting. Our main goal is to explore to which extent the 
properties of decomposable DAG models carry over to the context-specific case.

\begin{definition}
\label{def:perfect}
A DAG $G=([p],E)$ is \emph{perfect} if the skeleton of the induced subgraph on the vertices $\pa(k)$ is a complete graph for all $k\in [p]$.

\end{definition}

There are several equivalent ways to define a perfect DAG. Another way to characterize a perfect DAG $G$ is to require that its skeleton is chordal and there is no triple $u,v,w$ of vertices such that $u\to w, v\to w$ are edges in $G$ but $u$ and $v$ are not adjacent. One can also characterize a perfect DAG via its moral graph, see \cite[]{L96}.

Let $G$ be a DAG and let $\TT_G$ be the staged tree representation of $G$. 
A characterization of perfect DAGs is also available via balanced CStrees.

\begin{definition} \label{def:balanced}
Let $\TT$ be a CStree. For any vertex $v=x_1\cdots x_{k-1}\in \TT$ we define the \textit{interpolating polynomial} at $v$ by
\[
t(v):=\sum_{\zz\in\RR_{[p]\setminus [k-1]}}\left(\prod_{e\in E(v\to v\zz)}\theta(e)\right) \,\,\in \R[\theta(e)\colon e\in E].
\]
A pair of vertices $v=x_1\cdots x_{k-1}$ and $w=y_1\cdots y_{k-1}$ in the same stage is \textit{balanced} if for all $s,r\in[d_k]$, we have the equality
$$t(vs)t(wr)=t(vr)t(ws)$$
in the polynomial ring $\R[\theta(e)\colon e\in E]$.
The tree $\TT$ is \textit{balanced} if every pair of vertices in the same stage is balanced.
\end{definition}
\begin{remark}
The polynomial $t(\mathrm{root})$  in the previous definition is called the \textit{interpolating polynomial} of $\TT$. Such polynomial is useful to study equivalence classes of staged tree models \cite[]{GS18} and enumerating the trees in the equivalence class \cite[]{GBRS18} of any given staged tree. 
\end{remark}

\begin{theorem}\cite[Theorem 3.1]{DS21}\label{thm:perfect-iff-balanced-iff-decomposable}
The DAG $G$ is perfect if and only if $\TT_{G}$ is balanced if and only if $\MM(G)$ is decomposable.
\end{theorem}
One could hope that the direct generalization of Theorem~\ref{thm:perfect-iff-balanced-iff-decomposable} is true for balanced CStrees. 
Namely that a CStree $\TT$ is balanced if and only
all of its minimal context DAGs are perfect. This equivalence only holds 
 for three random variables $(p=3)$. 
For $p=4$, Example~\ref{ex:counterexample} provides a counterexample. 
In general, only one implication holds, namely, the CStree model $\MM(\TT)$ is balanced whenever all minimal contexts are perfect (Theorem~\ref{thm:perf_implies_balanced}).

\section{Decomposable CSmodels in three variables}\label{sec:3-vars}

Our subject of study from this point forward are balanced CStree models, the combinatorics of their context-specific DAG representations and the properties of their defining equations. Our results in Section~\ref{sec:algebra}
show that the properties of these models closely mirror those of decomposable DAG models. Therefore we introduce the following definition. 
\begin{definition}
A \textit{decomposable context-specific model (decomposable CSmodel)} is a balanced CStree model.
\end{definition}
Consistent with our previous notation, we will denote such a model by $\MM(\TT)$, where $\TT$ denotes the associated balanced CStree. The goal of this section is  to provide a complete description of CStree models
in three random variables, with fixed variable ordering $123$, and to prove the generalization of Theorem~\ref{thm:perfect-iff-balanced-iff-decomposable} for $p=3$.
That is, we prove the following result.
\begin{theorem}\label{thm:three_perfect_iff_balanced}
A CStree $\TT$ with $p=3$ is balanced, i.e. $\MM(\TT)$ is a decomposable CSmodel, if and only if all minimal context DAGs of $\TT$ are perfect.
\end{theorem}

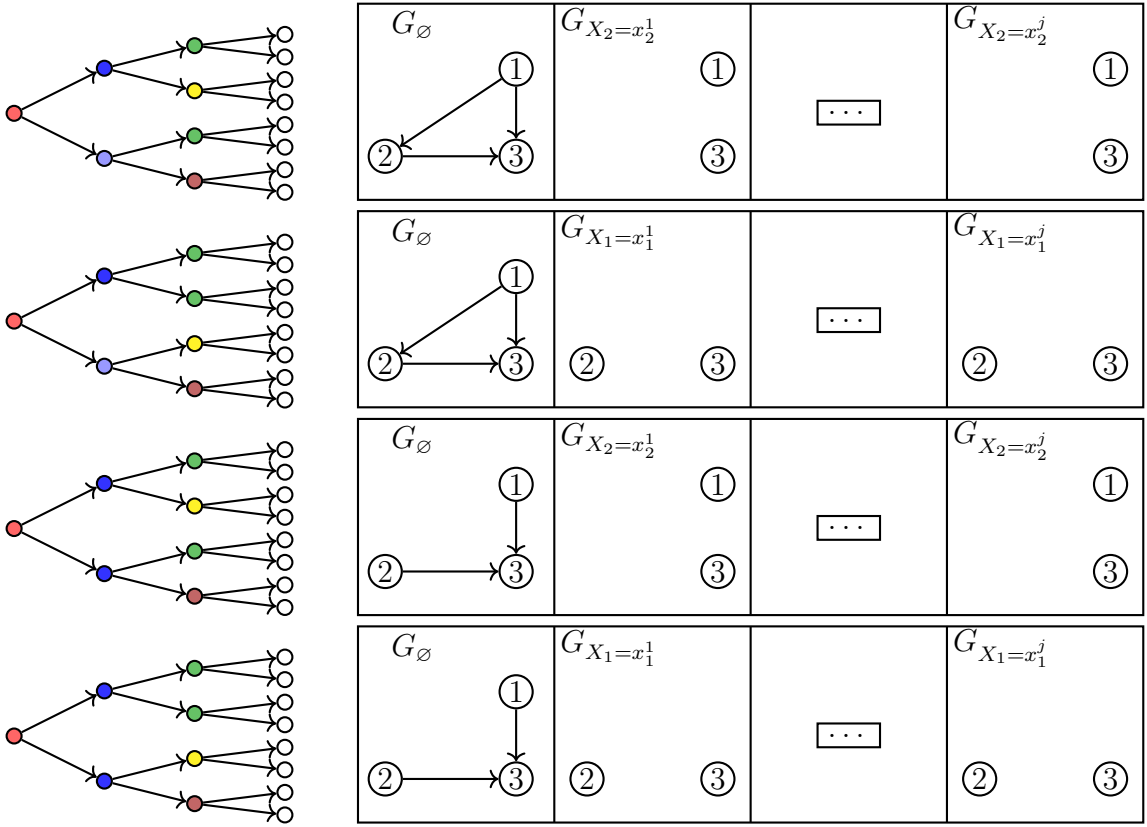
\begin{figure}
\begin{tikzpicture}[thick,scale=0.3]

 	 \node[circle, draw, fill=black!0, inner sep=2pt, minimum width=1pt] (w3) at (0,0)  {};
 	 \node[circle, draw, fill=black!0, inner sep=2pt, minimum width=1pt] (w4) at (0,-1) {};
 	 \node[circle, draw, fill=black!0, inner sep=2pt, minimum width=1pt] (w5) at (0,-2) {};
 	 \node[circle, draw, fill=black!0, inner sep=2pt, minimum width=1pt] (w6) at (0,-3) {};
 	 \node[circle, draw, fill=black!0, inner sep=2pt, minimum width=1pt] (v3) at (0,-4)  {};
 	 \node[circle, draw, fill=black!0, inner sep=2pt, minimum width=1pt] (v4) at (0,-5) {};
 	 \node[circle, draw, fill=black!0, inner sep=2pt, minimum width=1pt] (v5) at (0,-6) {};
 	 \node[circle, draw, fill=black!0, inner sep=2pt, minimum width=1pt] (v6) at (0,-7) {};

	 \node[circle, draw, fill=green!60!black!60, inner sep=2pt, minimum width=2pt] (w1) at (-4,-.5) {};
 	 \node[circle, draw, fill=yellow!90, inner sep=2pt, minimum width=2pt] (w2) at (-4,-2.5) {}; 
 	 \node[circle, draw, fill=green!60!black!60, inner sep=2pt, minimum width=2pt] (v1) at (-4,-4.5) {};
 	 \node[circle, draw, fill=red!60!black!60, inner sep=2pt, minimum width=2pt] (v2) at (-4,-6.5) {};

 	 \node[circle, draw, fill=blue!80, inner sep=2pt, minimum width=2pt] (w) at (-8,-1.5) {};
 	 \node[circle, draw, fill=blue!40, inner sep=2pt, minimum width=2pt] (v) at (-8,-5.5) {};	
	 
     \node[circle, draw, fill=red!60, inner sep=2pt, minimum width=2pt] (r) at (-12,-3.5) {};


 	 \draw[->]   (r) --   (w) ;
 	 \draw[->]   (r) --   (v) ;

 	 \draw[->]   (w) --  (w1) ;
 	 \draw[->]   (w) --  (w2) ;

 	 \draw[->]   (w1) --   (w3) ;
 	 \draw[->]   (w1) --   (w4) ;
 	 \draw[->]   (w2) --  (w5) ;
 	 \draw[->]   (w2) --  (w6) ;

 	 \draw[->]   (v) --  (v1) ;
 	 \draw[->]   (v) --  (v2) ;

 	 \draw[->]   (v1) --  (v3) ;
 	 \draw[->]   (v1) --  (v4) ;
 	 \draw[->]   (v2) --  (v5) ;
 	 \draw[->]   (v2) --  (v6) ;

\end{tikzpicture}
    \begin{tikzpicture}[thick,scale=0.29]
	\draw (-1,3) -- (35,3) -- (35, -6) -- (-1, -6) -- (-1,3) -- cycle;
	\draw (8,3) -- (8,-6) ; 
	\draw (17,3) -- (17,-6) ;
    \draw (26,3) -- (26,-6) ;

 	 \node[circle, draw, fill=black!0, inner sep=1pt, minimum width=1pt] (1v2) at (6.25,0) {$1$};
 	 \node[circle, draw, fill=black!0, inner sep=1pt, minimum width=1pt] (1v3) at (0.25,-4) {$2$};
 	 \node[circle, draw, fill=black!0, inner sep=1pt, minimum width=1pt] (1v4) at (6.25,-4) {$3$};

 	 \node[circle, draw, fill=black!0, inner sep=1pt, minimum width=1pt] (2v2) at (15.5,0) {$1$};
 	 \node[circle, draw, fill=black!0, inner sep=1pt, minimum width=1pt] (2v4) at (15.5,-4) {$3$};

 	 \node[circle, draw, fill=black!0, inner sep=1pt, minimum width=1pt] (3v2) at (33.5,0) {$1$};
 	 \node[circle, draw, fill=black!0, inner sep=1pt, minimum width=1pt] (3v4) at (33.5,-4) {$3$};

        \node[draw] at (21.5,-2) {$\dots$};

      \draw[->]   (1v2) -- (1v3) ;  
 	 \draw[->]   (1v2) -- (1v4) ;
 	 \draw[->]   (1v3) -- (1v4) ;

	 \node at (-3,2) {} ;
	 \node at (1.5,2) {$G_{\emptyset}$} ;
	 \node at (10.5,2) {$G_{X_2=x_2^1}$} ;
	 \node at (28.5,2) {$G_{X_2 =x_2^j}$} ;
\end{tikzpicture}

\begin{tikzpicture}[thick,scale=0.3]

 	 \node[circle, draw, fill=black!0, inner sep=2pt, minimum width=1pt] (w3) at (0,0)  {};
 	 \node[circle, draw, fill=black!0, inner sep=2pt, minimum width=1pt] (w4) at (0,-1) {};
 	 \node[circle, draw, fill=black!0, inner sep=2pt, minimum width=1pt] (w5) at (0,-2) {};
 	 \node[circle, draw, fill=black!0, inner sep=2pt, minimum width=1pt] (w6) at (0,-3) {};
 	 \node[circle, draw, fill=black!0, inner sep=2pt, minimum width=1pt] (v3) at (0,-4)  {};
 	 \node[circle, draw, fill=black!0, inner sep=2pt, minimum width=1pt] (v4) at (0,-5) {};
 	 \node[circle, draw, fill=black!0, inner sep=2pt, minimum width=1pt] (v5) at (0,-6) {};
 	 \node[circle, draw, fill=black!0, inner sep=2pt, minimum width=1pt] (v6) at (0,-7) {};

	 \node[circle, draw, fill=green!60!black!60, inner sep=2pt, minimum width=2pt] (w1) at (-4,-.5) {};
 	 \node[circle, draw, fill=green!60!black!60, inner sep=2pt, minimum width=2pt] (w2) at (-4,-2.5) {}; 
 	 \node[circle, draw, fill=yellow!90, inner sep=2pt, minimum width=2pt] (v1) at (-4,-4.5) {};
 	 \node[circle, draw, fill=red!60!black!60, inner sep=2pt, minimum width=2pt] (v2) at (-4,-6.5) {};

 	 \node[circle, draw, fill=blue!80, inner sep=2pt, minimum width=2pt] (w) at (-8,-1.5) {};
 	 \node[circle, draw, fill=blue!40, inner sep=2pt, minimum width=2pt] (v) at (-8,-5.5) {};	
	 
     \node[circle, draw, fill=red!60, inner sep=2pt, minimum width=2pt] (r) at (-12,-3.5) {};


 	 \draw[->]   (r) --   (w) ;
 	 \draw[->]   (r) --   (v) ;

 	 \draw[->]   (w) --  (w1) ;
 	 \draw[->]   (w) --  (w2) ;

 	 \draw[->]   (w1) --   (w3) ;
 	 \draw[->]   (w1) --   (w4) ;
 	 \draw[->]   (w2) --  (w5) ;
 	 \draw[->]   (w2) --  (w6) ;

 	 \draw[->]   (v) --  (v1) ;
 	 \draw[->]   (v) --  (v2) ;

 	 \draw[->]   (v1) --  (v3) ;
 	 \draw[->]   (v1) --  (v4) ;
 	 \draw[->]   (v2) --  (v5) ;
 	 \draw[->]   (v2) --  (v6) ;

\end{tikzpicture}
    \begin{tikzpicture}[thick,scale=0.29]
	\draw (-1,3) -- (35,3) -- (35, -6) -- (-1, -6) -- (-1,3) -- cycle;
	\draw (8,3) -- (8,-6) ; 
	\draw (17,3) -- (17,-6) ;
    \draw (26,3) -- (26,-6) ;

 	 \node[circle, draw, fill=black!0, inner sep=1pt, minimum width=1pt] (1v2) at (6.25,0) {$1$};
 	 \node[circle, draw, fill=black!0, inner sep=1pt, minimum width=1pt] (1v3) at (0.25,-4) {$2$};
 	 \node[circle, draw, fill=black!0, inner sep=1pt, minimum width=1pt] (1v4) at (6.25,-4) {$3$};

 	 \node[circle, draw, fill=black!0, inner sep=1pt, minimum width=1pt] (2v2) at (9.5,-4) {$2$};
 	 \node[circle, draw, fill=black!0, inner sep=1pt, minimum width=1pt] (2v4) at (15.5,-4) {$3$};

 	 \node[circle, draw, fill=black!0, inner sep=1pt, minimum width=1pt] (3v2) at (27.5,-4) {$2$};
 	 \node[circle, draw, fill=black!0, inner sep=1pt, minimum width=1pt] (3v4) at (33.5,-4) {$3$};

        \node[draw] at (21.5,-2) {$\dots$};

      \draw[->]   (1v2) -- (1v3) ;  
 	 \draw[->]   (1v2) -- (1v4) ;
 	 \draw[->]   (1v3) -- (1v4) ;

	 \node at (-3,2) {} ;
	 \node at (1.5,2) {$G_{\emptyset}$} ;
	 \node at (10.5,2) {$G_{X_1=x_1^1}$} ;
	 \node at (28.5,2) {$G_{X_1 = x_1^j}$} ;
\end{tikzpicture}

\begin{tikzpicture}[thick,scale=0.3]

 	 \node[circle, draw, fill=black!0, inner sep=2pt, minimum width=1pt] (w3) at (0,0)  {};
 	 \node[circle, draw, fill=black!0, inner sep=2pt, minimum width=1pt] (w4) at (0,-1) {};
 	 \node[circle, draw, fill=black!0, inner sep=2pt, minimum width=1pt] (w5) at (0,-2) {};
 	 \node[circle, draw, fill=black!0, inner sep=2pt, minimum width=1pt] (w6) at (0,-3) {};
 	 \node[circle, draw, fill=black!0, inner sep=2pt, minimum width=1pt] (v3) at (0,-4)  {};
 	 \node[circle, draw, fill=black!0, inner sep=2pt, minimum width=1pt] (v4) at (0,-5) {};
 	 \node[circle, draw, fill=black!0, inner sep=2pt, minimum width=1pt] (v5) at (0,-6) {};
 	 \node[circle, draw, fill=black!0, inner sep=2pt, minimum width=1pt] (v6) at (0,-7) {};

	 \node[circle, draw, fill=green!60!black!60, inner sep=2pt, minimum width=2pt] (w1) at (-4,-.5) {};
 	 \node[circle, draw, fill=yellow!90, inner sep=2pt, minimum width=2pt] (w2) at (-4,-2.5) {}; 
 	 \node[circle, draw, fill=green!60!black!60, inner sep=2pt, minimum width=2pt] (v1) at (-4,-4.5) {};
 	 \node[circle, draw, fill=red!60!black!60, inner sep=2pt, minimum width=2pt] (v2) at (-4,-6.5) {};

 	 \node[circle, draw, fill=blue!80, inner sep=2pt, minimum width=2pt] (w) at (-8,-1.5) {};
 	 \node[circle, draw, fill=blue!80, inner sep=2pt, minimum width=2pt] (v) at (-8,-5.5) {};	
	 
     \node[circle, draw, fill=red!60, inner sep=2pt, minimum width=2pt] (r) at (-12,-3.5) {};


 	 \draw[->]   (r) --   (w) ;
 	 \draw[->]   (r) --   (v) ;

 	 \draw[->]   (w) --  (w1) ;
 	 \draw[->]   (w) --  (w2) ;

 	 \draw[->]   (w1) --   (w3) ;
 	 \draw[->]   (w1) --   (w4) ;
 	 \draw[->]   (w2) --  (w5) ;
 	 \draw[->]   (w2) --  (w6) ;

 	 \draw[->]   (v) --  (v1) ;
 	 \draw[->]   (v) --  (v2) ;

 	 \draw[->]   (v1) --  (v3) ;
 	 \draw[->]   (v1) --  (v4) ;
 	 \draw[->]   (v2) --  (v5) ;
 	 \draw[->]   (v2) --  (v6) ;

\end{tikzpicture}
    \begin{tikzpicture}[thick,scale=0.29]
	\draw (-1,3) -- (35,3) -- (35, -6) -- (-1, -6) -- (-1,3) -- cycle;
	\draw (8,3) -- (8,-6) ; 
	\draw (17,3) -- (17,-6) ;
    \draw (26,3) -- (26,-6) ;

 	 \node[circle, draw, fill=black!0, inner sep=1pt, minimum width=1pt] (1v2) at (6.25,0) {$1$};
 	 \node[circle, draw, fill=black!0, inner sep=1pt, minimum width=1pt] (1v3) at (0.25,-4) {$2$};
 	 \node[circle, draw, fill=black!0, inner sep=1pt, minimum width=1pt] (1v4) at (6.25,-4) {$3$};

 	 \node[circle, draw, fill=black!0, inner sep=1pt, minimum width=1pt] (2v2) at (15.5,0) {$1$};
 	 \node[circle, draw, fill=black!0, inner sep=1pt, minimum width=1pt] (2v4) at (15.5,-4) {$3$};

 	 \node[circle, draw, fill=black!0, inner sep=1pt, minimum width=1pt] (3v2) at (33.5,0) {$1$};
 	 \node[circle, draw, fill=black!0, inner sep=1pt, minimum width=1pt] (3v4) at (33.5,-4) {$3$};

        \node[draw] at (21.5,-2) {$\dots$};

 	 \draw[->]   (1v2) -- (1v4) ;
 	 \draw[->]   (1v3) -- (1v4) ;

	 \node at (-3,2) {} ;
	 \node at (1.5,2) {$G_{\emptyset}$} ;
	 \node at (10.5,2) {$G_{X_2=x_2^1}$} ;
	 \node at (28.5,2) {$G_{X_2 =x_2^j}$} ;
\end{tikzpicture}

\begin{tikzpicture}[thick,scale=0.3]

 	 \node[circle, draw, fill=black!0, inner sep=2pt, minimum width=1pt] (w3) at (0,0)  {};
 	 \node[circle, draw, fill=black!0, inner sep=2pt, minimum width=1pt] (w4) at (0,-1) {};
 	 \node[circle, draw, fill=black!0, inner sep=2pt, minimum width=1pt] (w5) at (0,-2) {};
 	 \node[circle, draw, fill=black!0, inner sep=2pt, minimum width=1pt] (w6) at (0,-3) {};
 	 \node[circle, draw, fill=black!0, inner sep=2pt, minimum width=1pt] (v3) at (0,-4)  {};
 	 \node[circle, draw, fill=black!0, inner sep=2pt, minimum width=1pt] (v4) at (0,-5) {};
 	 \node[circle, draw, fill=black!0, inner sep=2pt, minimum width=1pt] (v5) at (0,-6) {};
 	 \node[circle, draw, fill=black!0, inner sep=2pt, minimum width=1pt] (v6) at (0,-7) {};

	 \node[circle, draw, fill=green!60!black!60, inner sep=2pt, minimum width=2pt] (w1) at (-4,-.5) {};
 	 \node[circle, draw, fill=green!60!black!60, inner sep=2pt, minimum width=2pt] (w2) at (-4,-2.5) {}; 
 	 \node[circle, draw, fill=yellow!90, inner sep=2pt, minimum width=2pt] (v1) at (-4,-4.5) {};
 	 \node[circle, draw, fill=red!60!black!60, inner sep=2pt, minimum width=2pt] (v2) at (-4,-6.5) {};

 	 \node[circle, draw, fill=blue!80, inner sep=2pt, minimum width=2pt] (w) at (-8,-1.5) {};
 	 \node[circle, draw, fill=blue!80, inner sep=2pt, minimum width=2pt] (v) at (-8,-5.5) {};	
	 
     \node[circle, draw, fill=red!60, inner sep=2pt, minimum width=2pt] (r) at (-12,-3.5) {};


 	 \draw[->]   (r) --   (w) ;
 	 \draw[->]   (r) --   (v) ;

 	 \draw[->]   (w) --  (w1) ;
 	 \draw[->]   (w) --  (w2) ;

 	 \draw[->]   (w1) --   (w3) ;
 	 \draw[->]   (w1) --   (w4) ;
 	 \draw[->]   (w2) --  (w5) ;
 	 \draw[->]   (w2) --  (w6) ;

 	 \draw[->]   (v) --  (v1) ;
 	 \draw[->]   (v) --  (v2) ;

 	 \draw[->]   (v1) --  (v3) ;
 	 \draw[->]   (v1) --  (v4) ;
 	 \draw[->]   (v2) --  (v5) ;
 	 \draw[->]   (v2) --  (v6) ;

\end{tikzpicture}
    \begin{tikzpicture}[thick,scale=0.29]
	\draw (-1,3) -- (35,3) -- (35, -6) -- (-1, -6) -- (-1,3) -- cycle;
	\draw (8,3) -- (8,-6) ; 
	\draw (17,3) -- (17,-6) ;
    \draw (26,3) -- (26,-6) ;

 	 \node[circle, draw, fill=black!0, inner sep=1pt, minimum width=1pt] (1v2) at (6.25,0) {$1$};
 	 \node[circle, draw, fill=black!0, inner sep=1pt, minimum width=1pt] (1v3) at (0.25,-4) {$2$};
 	 \node[circle, draw, fill=black!0, inner sep=1pt, minimum width=1pt] (1v4) at (6.25,-4) {$3$};

 	 \node[circle, draw, fill=black!0, inner sep=1pt, minimum width=1pt] (2v2) at (9.5,-4) {$2$};
 	 \node[circle, draw, fill=black!0, inner sep=1pt, minimum width=1pt] (2v4) at (15.5,-4) {$3$};

 	 \node[circle, draw, fill=black!0, inner sep=1pt, minimum width=1pt] (3v2) at (27.5,-4) {$2$};
 	 \node[circle, draw, fill=black!0, inner sep=1pt, minimum width=1pt] (3v4) at (33.5,-4) {$3$};

        \node[draw] at (21.5,-2) {$\dots$};

 	 \draw[->]   (1v2) -- (1v4) ;
 	 \draw[->]   (1v3) -- (1v4) ;

	 \node at (-3,2) {} ;
	 \node at (1.5,2) {$G_{\emptyset}$} ;
	 \node at (10.5,2) {$G_{X_1=x_1^1}$} ;
	 \node at (28.5,2) {$G_{X_1 = x_1^j}$} ;
\end{tikzpicture}
\caption{\small{All CStrees with $p=3$ and variable ordering $123$ which do not represent a DAG.}}
\label{fig:all-trees-for-3-variables}
\end{figure}

Before proving the above theorem, we classify all possible CStrees on three random variables along with their minimal contexts, taking advantage of the small value of $p$.

\begin{example}\label{ex:p=3}
We provide a list of all CStrees which are not staged tree representations of a DAG with causal ordering $123$ (see Figure~\ref{fig:all-trees-for-3-variables}). In this case there are four families of CStrees.

Consider the two pairs of DAGs $\{([3],\{1\to 2, 1\to 3, 2\to 3\}), ([3],\{1\to 3, 2\to 3\})\}$ and $\{(\{1,3\},\emptyset),(\{2,3\},\emptyset)\}$. Each of the four families is defined as follows: Choose two graphs $G_1,G_2$, one from each pair. Let $i\in\{1,2\}$ such that $i$ does not appear in the vertex set of $G_2$ and let $I\subsetneq [d_i]$. Now, consider the CStree defined by taking $G_1$ as its empty context DAG and $G_2$ as the minimal context DAG for the contexts $X_i=j$ for every $j\in I$.
Depending on the choice of $G_1,G_2$ we get exactly the four families in Figure~\ref{fig:all-trees-for-3-variables}. The first family for example corresponds to choosing $G_1$ to be the complete graph, $G_2=(\{1,3\},\emptyset)$ and some $I\subsetneq [d_2]$.

Note that any such choice does define a CStree and all contexts will be minimal contexts (except the empty context if the DAG is chosen to be the complete graph). If one would take $I=[d_i]$ this would no longer be true and the CStree would in fact be the staged tree representation of a DAG. The staged trees on the left of Figure~\ref{fig:all-trees-for-3-variables} are examples in which all random variables are binary and such that the minimal context which is not the empty context is either $X_1=x_1^1$ or $X_2=x_2^1$.

In the first two families the empty context is not a minimal context as there are no CI relations that hold. We still draw the complete graph in this example for consistency.

Moreover, we can check that the first two families are balanced CStrees whereas the latter two are not. In the first two cases we also see that all minimal context DAGs are perfect which again is not the case for the latter two (cf. Theorems~\ref{thm:perf_implies_balanced},~\ref{thm:three_perfect_iff_balanced}).

\end{example}

\begin{proposition}\label{prop:cstrees_three_variables}
    The list of CStrees in Example~\ref{ex:p=3} is a complete list of CStrees with levels $(L_1,L_2,L_3)\sim (X_1,X_2,X_3)$
    that are not staged tree representations of a DAG.
\end{proposition}
The proof can be found at the end of Section \ref{sec:combinatorics}.

\begin{lemma}\label{lemma:x1-ind-x2-DAG}
Let $\TT$ be a balanced CStree with $p=3$. If $X_1\independent X_2$, then $\TT$ represents a DAG.
\end{lemma}
\begin{proof}
Since $X_1\independent X_2$, all vertices in the first level of $\TT$ are in the same stage and $\varnothing\in\CC_\TT$. We claim that there are no other minimal contexts, besides the empty one. Since $\TT$ is balanced by assumption, for any two vertices $v$ and $u$ in level 1 and $v_i,v_j\in\children_\TT(v), u_i,u_j\in\children_\TT(u)$ with $\theta(u\to u_\ell)=\theta(v\to v_\ell)$ for $\ell\in[d_2]$, we have
$t(v_i)t(u_j)=t(v_j)t(u_i).$
Since $p=3$, we have one of the following cases:
\begin{align*}
    &1)\;t(v_i)=t(v_j)\implies v_i\sim v_j\text{ and }u_i\sim u_j\\
    &2)\;t(v_i)\neq t(v_j)\implies u_i\sim v_i\text{ and }u_j\sim v_j,
\end{align*}
where $\sim$ denotes the equivalence relation of being in the same stage. Since $\TT$ is a CStree, the first case implies that all children of any vertex $v$ in level 1 are in the same stage. But then $X_2\independent X_3|X_1$, so $X_1=\ell$ is not a minimal context for any $\ell\in[d_1]$. In the second case, we get that for any two vertices $v$ and $u$ in level 1, we have $v'\sim w'$ for some $v'\in\children_\TT (v$) and $w'\in\children_\TT (w)$. But then $X_1\independent X_3|X_2$, so again $X_2=\ell\not\in\CC_\TT$ for any $\ell\in[d_2]$. We conclude $\CC_\TT=\{\varnothing\}$, so indeed $\TT$ represents a DAG.
\end{proof}

We are now ready to prove Theorem~\ref{thm:three_perfect_iff_balanced}.
\begin{proof}[Proof of Theorem~\ref{thm:three_perfect_iff_balanced}]
We show in Theorem~\ref{thm:perf_implies_balanced} that (for any $p$) if all minimal context DAGs of $\TT$ are perfect, then $\TT$ is balanced. Hence, it suffices to show the other implication. Let $\TT$ be a balanced CStree with $p=3$. Assume there exists a minimal context DAG $G$ that is not perfect. This has to be the empty context DAG as other minimal context DAGs can only have two vertices. Hence, the empty minimal context DAG is $1\to 3 \leftarrow 2$ which implies $X_1\independent X_2$. Using Lemma~\ref{lemma:x1-ind-x2-DAG} it follows that $\TT=\TT_G$ is the staged tree representation of $G$. However, such a CStree is unbalanced by Theorem~\ref{thm:perfect-iff-balanced-iff-decomposable}.
\end{proof}

We have just observed that every balanced CStree has only perfect minimal contexts DAGs when $p=3$. This is no longer true for $p\ge 4$, as illustrated by the next example.

\begin{figure}[t]
    \centering
\begin{tikzpicture}[thick,scale=0.3]

 	 \node[circle, draw, fill=black!0, inner sep=2pt, minimum width=1pt] (w3) at (0,0)  {};
 	 \node[circle, draw, fill=black!0, inner sep=2pt, minimum width=1pt] (w4) at (0,-1) {};
 	 \node[circle, draw, fill=black!0, inner sep=2pt, minimum width=1pt] (w5) at (0,-2) {};
 	 \node[circle, draw, fill=black!0, inner sep=2pt, minimum width=1pt] (w6) at (0,-3) {};
 	 \node[circle, draw, fill=black!0, inner sep=2pt, minimum width=1pt] (v3) at (0,-4)  {};
 	 \node[circle, draw, fill=black!0, inner sep=2pt, minimum width=1pt] (v4) at (0,-5) {};
 	 \node[circle, draw, fill=black!0, inner sep=2pt, minimum width=1pt] (v5) at (0,-6) {};
 	 \node[circle, draw, fill=black!0, inner sep=2pt, minimum width=1pt] (v6) at (0,-7) {};
 	 \node[circle, draw, fill=black!0, inner sep=2pt, minimum width=2pt] (w3i) at (0,-8)  {};
 	 \node[circle, draw, fill=black!0, inner sep=2pt, minimum width=2pt] (w4i) at (0,-9) {};
 	 \node[circle, draw, fill=black!0, inner sep=2pt, minimum width=2pt] (w5i) at (0,-10) {};
 	 \node[circle, draw, fill=black!0, inner sep=2pt, minimum width=2pt] (w6i) at (0,-11) {};
 	 \node[circle, draw, fill=black!0, inner sep=2pt, minimum width=2pt] (v3i) at (0,-12)  {};
 	 \node[circle, draw, fill=black!0, inner sep=2pt, minimum width=2pt] (v4i) at (0,-13) {};
 	 \node[circle, draw, fill=black!0, inner sep=2pt, minimum width=2pt] (v5i) at (0,-14) {};
 	 \node[circle, draw, fill=black!0, inner sep=2pt, minimum width=2pt] (v6i) at (0,-15) {};

	 \node[circle, draw, fill=green!60!black!60, inner sep=2pt, minimum width=2pt] (w1) at (-4,-.5) {};
 	 \node[circle, draw, fill=green!60!black!60, inner sep=2pt, minimum width=2pt] (w2) at (-4,-2.5) {}; 
 	 \node[circle, draw, fill=red!60!black!60, inner sep=2pt, minimum width=2pt] (v1) at (-4,-4.5) {};
 	 \node[circle, draw, fill=red!60!black!60, inner sep=2pt, minimum width=2pt] (v2) at (-4,-6.5) {};	 
	 \node[circle, draw, fill=red!80!yellow!60, inner sep=2pt, minimum width=2pt] (w1i) at (-4,-8.5) {};
 	 \node[circle, draw, fill=yellow!50, inner sep=2pt, minimum width=2pt] (w2i) at (-4,-10.5) {};
 	 \node[circle, draw, fill=red!80!yellow!60, inner sep=2pt, minimum width=2pt] (v1i) at (-4,-12.5) {};
 	 \node[circle, draw, fill=yellow!50, inner sep=2pt, minimum width=2pt] (v2i) at (-4,-14.5) {};

 	 \node[circle, draw, fill=blue!80, inner sep=2pt, minimum width=2pt] (w) at (-8,-1.5) {};
 	 \node[circle, draw, fill=blue!80, inner sep=2pt, minimum width=2pt] (v) at (-8,-5.5) {};	
 	 \node[circle, draw, fill=blue!20, inner sep=2pt, minimum width=2pt] (wi) at (-8,-9.5) {};
 	 \node[circle, draw, fill=blue!20, inner sep=2pt, minimum width=2pt] (vi) at (-8,-13.5) {};
	 
     \node[circle, draw, fill=red!60, inner sep=2pt, minimum width=2pt] (r) at (-12,-3.5) {};
 	 \node[circle, draw, fill=green!60, inner sep=2pt, minimum width=2pt] (ri) at (-12,-11.5) {};

 	 \node[circle, draw, fill=black!0, inner sep=2pt, minimum width=2pt] (I) at (-16,-7.5) {};

 	 \draw[->]   (I) --    (r) ;
 	 \draw[->]   (I) --   (ri) ;

 	 \draw[->]   (r) --   (w) ;
 	 \draw[->]   (r) --   (v) ;

 	 \draw[->]   (w) --  (w1) ;
 	 \draw[->]   (w) --  (w2) ;

 	 \draw[->]   (w1) --   (w3) ;
 	 \draw[->]   (w1) --   (w4) ;
 	 \draw[->]   (w2) --  (w5) ;
 	 \draw[->]   (w2) --  (w6) ;

 	 \draw[->]   (v) --  (v1) ;
 	 \draw[->]   (v) --  (v2) ;

 	 \draw[->]   (v1) --  (v3) ;
 	 \draw[->]   (v1) --  (v4) ;
 	 \draw[->]   (v2) --  (v5) ;
 	 \draw[->]   (v2) --  (v6) ;

 	 \draw[->]   (ri) --   (wi) ;
 	 \draw[->]   (ri) -- (vi) ;

 	 \draw[->]   (wi) --  (w1i) ;
 	 \draw[->]   (wi) --  (w2i) ;

 	 \draw[->]   (w1i) --  (w3i) ;
 	 \draw[->]   (w1i) -- (w4i) ;
 	 \draw[->]   (w2i) --  (w5i) ;
 	 \draw[->]   (w2i) --  (w6i) ;

 	 \draw[->]   (vi) --  (v1i) ;
 	 \draw[->]   (vi) --  (v2i) ;

 	 \draw[->]   (v1i) --  (v3i) ;
 	 \draw[->]   (v1i) -- (v4i) ;
 	 \draw[->]   (v2i) -- (v5i) ;
 	 \draw[->]   (v2i) --  (v6i) ;

\end{tikzpicture}
\begin{tikzpicture}[thick,scale=0.29]
	\draw (-1,3) -- (26,3) -- (26, -6) -- (-1, -6) -- (-1,3) -- cycle;
	\draw (7.5,3) -- (7.5,-6) ; 
	\draw (17,3) -- (17,-6) ; 
	
	 \node[circle, draw, fill=black!0, inner sep=1pt, minimum width=1pt] (1v1) at (0.25,0) {$1$};
 	 \node[circle, draw, fill=black!0, inner sep=1pt, minimum width=1pt] (1v2) at (6.25,0) {$2$};
 	 \node[circle, draw, fill=black!0, inner sep=1pt, minimum width=1pt] (1v3) at (0.25,-4) {$3$};
 	 \node[circle, draw, fill=black!0, inner sep=1pt, minimum width=1pt] (1v4) at (6.25,-4) {$4$};

 	 \node[circle, draw, fill=black!0, inner sep=1pt, minimum width=1pt] (2v2) at (15.5,0) {$2$};
 	 \node[circle, draw, fill=black!0, inner sep=1pt, minimum width=1pt] (2v3) at (9.5,-4) {$3$};
 	 \node[circle, draw, fill=black!0, inner sep=1pt, minimum width=1pt] (2v4) at (15.5,-4) {$4$};

 	 \node[circle, draw, fill=black!0, inner sep=1pt, minimum width=1pt] (3v2) at (24.5,0) {$2$};
 	 \node[circle, draw, fill=black!0, inner sep=1pt, minimum width=1pt] (3v3) at (18.5,-4) {$3$};
 	 \node[circle, draw, fill=black!0, inner sep=1pt, minimum width=1pt] (3v4) at (24.5,-4) {$4$};

 	 \draw[->]   (1v1) -- (1v2) ;
 	 \draw[->]   (1v1) -- (1v3) ;
 	 \draw[->]   (1v1) -- (1v4) ;
 	 \draw[->]   (1v2) -- (1v4) ;
 	 \draw[->]   (1v3) -- (1v4) ;
 	 
 	 \draw[->]   (2v2) -- (2v4) ;
 	 
 	 \draw[->]   (3v3) -- (3v4) ;

	 \node at (-3,2) {} ;
	 \node at (1.5,2) {$G_{\emptyset}$} ;
	 \node at (10,2) {$G_{X_1=0}$} ;
	 \node at (19.5,2) {$G_{X_1 = 1}$} ;
\end{tikzpicture}
\caption{\small{A balanced CStree with a non-perfect minimal context.}}
\label{fig:counterexample}
\end{figure}
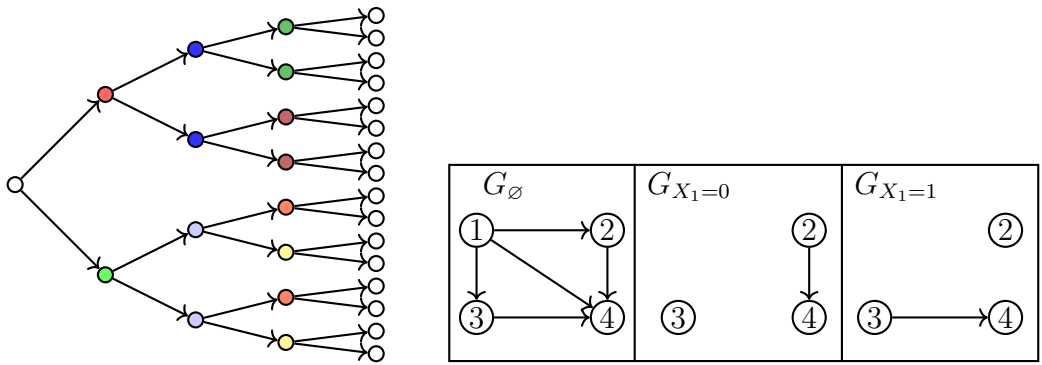

\begin{example}\label{ex:counterexample}
We consider the binary CStree in Figure~\ref{fig:counterexample} on $p=4$
binary random variables which is equivalently given by its three minimal context DAGs.
This CStree is balanced as one can check using Definition~\ref{def:balanced}, but the empty minimal context DAG $G_{\emptyset}$ is not perfect as the parents of $4$ do not form a complete graph. Therefore, a straightforward generalization of Theorem~\ref{thm:three_perfect_iff_balanced} is not true. 

This example can also  be generalized to get more counterexamples for any $p\ge 4$ and with an arbitrarily large number of minimal context. We may note that the statement $X_2\independent X_3|X_1$ (which prevents the parents of $4$ from forming a complete graph) is implied by the other two minimal contexts using absorption. We reveal why this happens in Section~\ref{sec:algebra}.
\end{example}

\begin{remark}
    In the case $p=4$ the CStree in Figure~\ref{fig:counterexample} is essentially the only binary balanced CStree with a non-perfect minimal context (up to swapping the outcomes of $X_1$ in the minimal contexts). If we do not restrict to binary CStrees there exists a family of such CStrees with a non-perfect context DAG, it can be constructed similarly to Example~\ref{ex:p=3}.
\end{remark}

\begin{remark}
Another characterizing property of
decomposable graphical models is in terms of its
maximum likelihood estimator (MLE).
Decomposable graphical models  are the only class of undirected graphical models whose MLE is a rational function \cite[Theorem 4.4]{GMS06}. The MLE of a  Decomposable CSmodel is also a rational function, this follows from the fact that they are a subclass of staged tree models and staged tree models always have this property \cite[]{DMS20}.
\end{remark}

\section{Combinatorial properties of balanced CStrees}\label{sec:combinatorics}
First, we study context subtrees of CStrees to understand which properties of CStrees are preserved when restricting to specific contexts. It turns out that any context subtree of a balanced CStree is itself balanced (Theorem~\ref{thm:balanced-for-every-context}) which can be seen as a generalization of the fact that removing a vertex from a perfect DAG results in a perfect~DAG.

Second, we saw in Example~\ref{ex:counterexample} that a CStree can be balanced without its minimal context
DAGs being perfect. The reverse implication does hold, i.e. if all minimal context DAGs are perfect, then the CStree is balanced (Theorem~\ref{thm:perf_implies_balanced}). The proof is mostly combinatorial in nature and does not make use of algebraic methods other than the definition of balancedness.
Lastly, we establish Proposition~\ref{prop:weird_formula} which will be used in the main proof of the last section. It gives an interpretation of the staging of a CStree in terms of CSI statements, as well as combinatorial conditions on minimal context DAGs for stagings to exist.

\subsection{Context subtrees}
We refer the reader back to Section~\ref{subsec:contextDAGs} for the formal definition of a context subtree and give an illustrative example here. 
For any context $X_{C}=\xx_C$, the subtree $\TT_{X_C=\xx_C}$ is
a CStree, the DAG $G_{X_{C}=\xx_C,\varnothing}$ denotes the empty context
DAG of $\TT_{X_C=\xx_C}$.

\begin{example}\label{ex:context_subtree}
    We consider the CStree $\TT$ in Figure~\ref{fig:motivatingEx} and construct the context subtree $\TT_{X_3=0}$ given in Figure~\ref{fig:context_subtree}. We remove all subtrees with root $x_1x_2 1$ and $x_1,x_2\in \{0,1\}$ and contract the edges $x_1x_2\to x_1x_2 0$. The stage of the node resulting from this contraction is the stage of the node $x_1x_2 1$. The stages of level 2 do not exist anymore and they do not have any meaning in the construction of the context subtree.
    We could now construct the minimal context DAGs from this CStree. However, we will instead do this from the minimal context DAGs of the full tree.
    We check if any minimal context is invalid in the case $X_3=0$, i.e. is only valid for $X_3=1$, and discard this DAG. This however is not the case here. Now we remove the node $3$ from any minimal context DAG, resulting in the collection of DAGs in Figure~\ref{fig:context_subtree}.
    In this case all non-empty contexts are in fact minimal contexts of the context subtree $\TT_{X_3=0}$, however this is not true in general.

    Moreover, we see that this context subtree $\TT_{X_3=0}$ is different from the tree $\TT_{G_{X_3=0,\emptyset}}$ (the staged tree representation of the empty context DAG) of which every stage is a singleton.
\end{example}

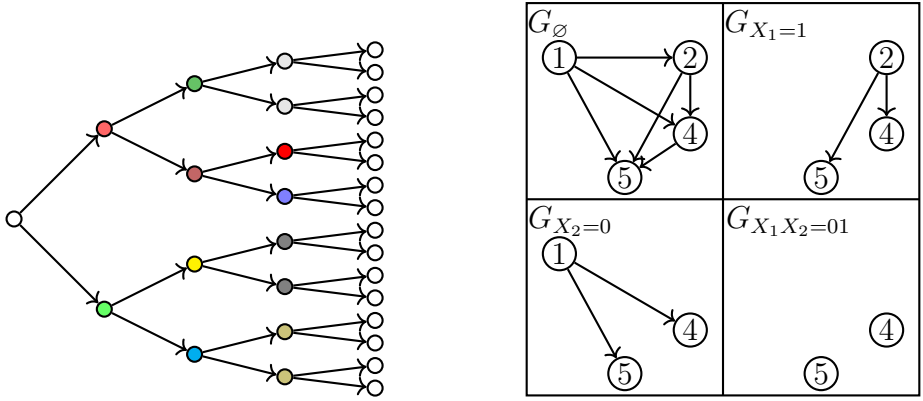
\begin{figure}[t]
    \centering
\begin{tikzpicture}[thick,scale=0.3]

 	 \node[circle, draw, fill=black!0, inner sep=2pt, minimum width=1pt] (w3) at (0,0)  {};
 	 \node[circle, draw, fill=black!0, inner sep=2pt, minimum width=1pt] (w4) at (0,-1) {};
 	 \node[circle, draw, fill=black!0, inner sep=2pt, minimum width=1pt] (w5) at (0,-2) {};
 	 \node[circle, draw, fill=black!0, inner sep=2pt, minimum width=1pt] (w6) at (0,-3) {};
 	 \node[circle, draw, fill=black!0, inner sep=2pt, minimum width=1pt] (v3) at (0,-4)  {};
 	 \node[circle, draw, fill=black!0, inner sep=2pt, minimum width=1pt] (v4) at (0,-5) {};
 	 \node[circle, draw, fill=black!0, inner sep=2pt, minimum width=1pt] (v5) at (0,-6) {};
 	 \node[circle, draw, fill=black!0, inner sep=2pt, minimum width=1pt] (v6) at (0,-7) {};
 	 \node[circle, draw, fill=black!0, inner sep=2pt, minimum width=2pt] (w3i) at (0,-8)  {};
 	 \node[circle, draw, fill=black!0, inner sep=2pt, minimum width=2pt] (w4i) at (0,-9) {};
 	 \node[circle, draw, fill=black!0, inner sep=2pt, minimum width=2pt] (w5i) at (0,-10) {};
 	 \node[circle, draw, fill=black!0, inner sep=2pt, minimum width=2pt] (w6i) at (0,-11) {};
 	 \node[circle, draw, fill=black!0, inner sep=2pt, minimum width=2pt] (v3i) at (0,-12)  {};
 	 \node[circle, draw, fill=black!0, inner sep=2pt, minimum width=2pt] (v4i) at (0,-13) {};
 	 \node[circle, draw, fill=black!0, inner sep=2pt, minimum width=2pt] (v5i) at (0,-14) {};
 	 \node[circle, draw, fill=black!0, inner sep=2pt, minimum width=2pt] (v6i) at (0,-15) {};

	 \node[circle, draw, fill=black!10, inner sep=2pt, minimum width=2pt] (w1) at (-4,-.5) {};
 	 \node[circle, draw, fill=black!10, inner sep=2pt, minimum width=2pt] (w2) at (-4,-2.5) {}; 
 	 \node[circle, draw, fill=red, inner sep=2pt, minimum width=2pt] (v1) at (-4,-4.5) {};
 	 \node[circle, draw, fill=blue!50, inner sep=2pt, minimum width=2pt] (v2) at (-4,-6.5) {};	 
	 \node[circle, draw, fill=black!50, inner sep=2pt, minimum width=2pt] (w1i) at (-4,-8.5) {};
 	 \node[circle, draw, fill=black!50, inner sep=2pt, minimum width=2pt] (w2i) at (-4,-10.5) {};
 	 \node[circle, draw, fill=olive!50, inner sep=2pt, minimum width=2pt] (v1i) at (-4,-12.5) {};
 	 \node[circle, draw, fill=olive!50, inner sep=2pt, minimum width=2pt] (v2i) at (-4,-14.5) {};

 	 \node[circle, draw, fill=green!60!black!60, inner sep=2pt, minimum width=2pt] (w) at (-8,-1.5) {};
 	 \node[circle, draw, fill=red!60!black!60, inner sep=2pt, minimum width=2pt] (v) at (-8,-5.5) {};	
 	 \node[circle, draw, fill=yellow, inner sep=2pt, minimum width=2pt] (wi) at (-8,-9.5) {};
 	 \node[circle, draw, fill=cyan, inner sep=2pt, minimum width=2pt] (vi) at (-8,-13.5) {};
	 
     \node[circle, draw, fill=red!60, inner sep=2pt, minimum width=2pt] (r) at (-12,-3.5) {};
 	 \node[circle, draw, fill=green!60, inner sep=2pt, minimum width=2pt] (ri) at (-12,-11.5) {};

 	 \node[circle, draw, fill=black!0, inner sep=2pt, minimum width=2pt] (I) at (-16,-7.5) {};

 	 \draw[->]   (I) --    (r) ;
 	 \draw[->]   (I) --   (ri) ;

 	 \draw[->]   (r) --   (w) ;
 	 \draw[->]   (r) --   (v) ;

 	 \draw[->]   (w) --  (w1) ;
 	 \draw[->]   (w) --  (w2) ;

 	 \draw[->]   (w1) --   (w3) ;
 	 \draw[->]   (w1) --   (w4) ;
 	 \draw[->]   (w2) --  (w5) ;
 	 \draw[->]   (w2) --  (w6) ;

 	 \draw[->]   (v) --  (v1) ;
 	 \draw[->]   (v) --  (v2) ;

 	 \draw[->]   (v1) --  (v3) ;
 	 \draw[->]   (v1) --  (v4) ;
 	 \draw[->]   (v2) --  (v5) ;
 	 \draw[->]   (v2) --  (v6) ;

 	 \draw[->]   (ri) --   (wi) ;
 	 \draw[->]   (ri) -- (vi) ;

 	 \draw[->]   (wi) --  (w1i) ;
 	 \draw[->]   (wi) --  (w2i) ;

 	 \draw[->]   (w1i) --  (w3i) ;
 	 \draw[->]   (w1i) -- (w4i) ;
 	 \draw[->]   (w2i) --  (w5i) ;
 	 \draw[->]   (w2i) --  (w6i) ;

 	 \draw[->]   (vi) --  (v1i) ;
 	 \draw[->]   (vi) --  (v2i) ;

 	 \draw[->]   (v1i) --  (v3i) ;
 	 \draw[->]   (v1i) -- (v4i) ;
 	 \draw[->]   (v2i) -- (v5i) ;
 	 \draw[->]   (v2i) --  (v6i) ;

\end{tikzpicture}
\hspace{1.5cm}
\begin{tikzpicture}[thick,scale=0.29]
	\draw (0,0) -- (18,0) -- (18,-18) -- (0,-18) -- cycle;
    \draw (9,0) -- (9,-18);
    \draw (0,-9) -- (18,-9);

	 \node[circle, draw, fill=black!0, inner sep=1pt, minimum width=1pt] (1v1) at (1.5,-2.5) {$1$};
 	 \node[circle, draw, fill=black!0, inner sep=1pt, minimum width=1pt] (1v2) at (7.5,-2.5) {$2$};
 	 \node[circle, draw, fill=black!0, inner sep=1pt, minimum width=1pt] (1v4) at (7.5,-6) {$4$}; 	 
 	 \node[circle, draw, fill=black!0, inner sep=1pt, minimum width=1pt] (1v5) at (4.5,-8) {$5$};

 	 \node[circle, draw, fill=black!0, inner sep=1pt, minimum width=1pt] (2v2) at (16.5,-2.5) {$2$};
 	 \node[circle, draw, fill=black!0, inner sep=1pt, minimum width=1pt] (2v4) at (16.5,-6) {$4$}; 	 
 	 \node[circle, draw, fill=black!0, inner sep=1pt, minimum width=1pt] (2v5) at (13.5,-8) {$5$};

	 \node[circle, draw, fill=black!0, inner sep=1pt, minimum width=1pt] (3v1) at (1.5,-11.5) {$1$};
 	 \node[circle, draw, fill=black!0, inner sep=1pt, minimum width=1pt] (3v4) at (7.5,-15) {$4$}; 	 
 	 \node[circle, draw, fill=black!0, inner sep=1pt, minimum width=1pt] (3v5) at (4.5,-17) {$5$};
 	 
 	 \node[circle, draw, fill=black!0, inner sep=1pt, minimum width=1pt] (4v4) at (16.5,-15) {$4$}; 	 
 	 \node[circle, draw, fill=black!0, inner sep=1pt, minimum width=1pt] (4v5) at (13.5,-17) {$5$};

 	 \draw[->]   (1v1) -- (1v2) ;
 	 \draw[->]   (1v1) -- (1v4) ;
 	 \draw[->]   (1v1) -- (1v5) ;
 	 \draw[->]   (1v2) -- (1v4) ;
 	 \draw[->]   (1v2) -- (1v5) ;
 	 \draw[->]   (1v4) -- (1v5) ;

 	 \draw[->]   (2v2) -- (2v4) ;
 	 \draw[->]   (2v2) -- (2v5) ;
 	 
 	 \draw[->]   (3v1) -- (3v4) ;
 	 \draw[->]   (3v1) -- (3v5) ;

	 \node at (1,-1) {$G_{\emptyset}$} ;
	 \node at (11,-1) {$G_{X_1=1}$} ;
	 \node at (2,-10) {$G_{X_2 = 0}$} ;
	 \node at (12,-10) {$G_{X_1X_2 = 01}$} ;
\end{tikzpicture}
\caption{\small{The context subtree of the tree in Figure~\ref{fig:motivatingEx} for the context $X_3=0$, and its minimal context DAGs.}}
\label{fig:context_subtree}
\end{figure}

\begin{lemma}
\label{lem:rooted_subtrees}
Suppose $\TT$ is a CStree with levels $(L_1,\ldots,L_p)\sim (X_1,\ldots, X_p)$.
\begin{itemize}
    \item[(1)] Every stage in $\TT_{G_\emptyset}$  is a subset of
    a stage in $\TT$.
    \item[(2)] Suppose $C\subset [p]$ is a context with maximum index $k$ and
    let $v=x_1\dots x_{q}\in V_\TT$ be a vertex of $\TT$ with $k\leq q$. 
    Then for any $\xx_{C}\in \RR_{C}$ such that $(\xx_C)_i=x_i$, we have the equality $\TT_v=(\TT_{X_{C}=\xx_{C}})_v$. 
    Since $t_\TT(v)$ only depends on the subtree $\TT_v$, it follows that $t_\TT(v)=t_{\TT_{X_{C'}=\xx_{C'}}}(v)\in\R[\Theta_{\TT_v}]$.
\end{itemize}
\end{lemma}
\begin{proof}
(1) Let $S$ be a stage in $\TT_{G_\emptyset}$ immediately preceding the level of the variable $X_k$, $k\in [p]$. Since $\TT_{G_\emptyset}$
represents a DAG, the stage defining context $X_{A}=\xx_A$ of $S$ satisfies
$A=\pa_{G_\emptyset}(k)$ and $\xx_{A}\in \RR_{A}$. Thus, as
a subset of vertices in $\TT$,
\[
S=\bigcup_{\yy \in \RR_{[k-1]\setminus A}} \{\xx_{A}\yy \}
\]
for some $\xx_A\in\RR_A$.
By the ordered Markov property in $G_\emptyset$, $G_\emptyset$ encodes
the CI relation $X_{k}\independent X_{[k-1]\setminus A}|X_{A}$. This
CI statement in $G_\emptyset$ corresponds to the CI statement
$X_{k}\independent X_{[k-1]\setminus A}|X_{A}$ in $\TT$.
Thus, by \cite[Theorem 3.3]{DS22} $X_{k}\independent X_{[k-1]\setminus A}|X_{A}$  holds in $\TT$. By  specialization to $X_{A}=\xx_A$, the statement $X_{k}\independent X_{[k-1]\setminus A}|X_{A}=\xx_A$ holds in $\TT$. The fact that this latter statement holds in $\TT$,
implies that the nodes in $S$ must be a subset of a stage in $\TT$.

(2) The vertices of the two trees are the same. A stage in the tree $\TT_v$ is defined by a statement $X_j\independent X_{[j-1]\setminus ([q]\cup D)}| X_D=\xx_D$ for some $j>q$ and $D\subset [j-1]\setminus [q]$. A stage in the tree $(\TT_{X_{C}=\xx_{C}})_v$ is defined by exactly the same kind of statement since $C\subset [k]\subset [q]$.
\end{proof}

Lemma~\ref{lem:rooted_subtrees} (1) says that every CStree $\TT$ is a coarsening of the CStree $\TT_{G_\emptyset}$, as every stage of $\TT$ is the union of possibly several stages in $\TT_{G_\emptyset}$. This coarsening is a result of other minimal context DAGs entailing more CSI statements. Hence, if $\TT=\TT_{G_\emptyset}$ all CSI statements implied by $\TT$ are specializations of CI statements also implied by $\TT$.

We recall a useful lemma to prove balancedness.
\begin{lemma}[{\cite[Lemma 3.2]{DS21}}]\label{lem:balancecheck}
Let $G=([p],E)$ be a DAG and assume $\pi=12\dotsm p$ is a linear extension of $G$. Then $\TT_{G}$ is balanced 
if and only if for every pair of vertices in the same stage with $v=x_1\dotsm x_i, w=x_1'\dotsm x_i'\in \RR_{\{i\}}$,
there exists a bijection
\begin{align*}
    \Phi:\RR_{[p]\setminus[i+1]}\times \RR_{[p]\setminus [i+1]}\ &\to \RR_{[p]\setminus[i+1]}\times \RR_{[p]\setminus [i+1]} \\
    (y_{i+2}\dotsm y_p,y_{i+2}'\dotsm y_p')&\mapsto (z_{i+2}\dotsm z_p,z_{i+2}'\dotsm z_p') 
\end{align*}
such that for all $k\ge i+2$ and all $s\neq r\in [d_{i+1}]$
\begin{equation*}
\begin{split}
f&(y_k\mid (x_1\cdots x_i,s,y_{i+2}\cdots y_p)_{\pa(k)})f(y_k^\prime\mid (x_1^\prime\cdots x_i^\prime,r,y_{i+2}^\prime\cdots y_p^\prime)_{\pa(k)}) \\
&= f(z_k\mid (x_1^\prime\cdots x_i^\prime,s,z_{i+2}\cdots z_p)_{\pa(k)})f(z_k^\prime\mid (x_1\cdots x_i,r,z_{i+2}^\prime\cdots z_p^\prime)_{\pa(k)}).
\end{split}
\end{equation*}
\end{lemma}

\begin{theorem}\label{thm:balanced-for-every-context}
If a CStree $\TT$ is balanced, then so is $\TT_{X_C=\xx_C}$ for every context $X_C=\xx_C$.
\end{theorem}
\begin{proof}
Let $\TTb:=\TT_{X_C=\xx_C}$. Let $k\in [p]\setminus C$ and suppose
$v=x_1\ldots x_{k-1}$ and $w=y_1\cdots y_{k-1} $ are two vertices in the same stage in $\TTb$ with $x_i=y_i$ for $i\in C\cap [k-1]$. Note that $v,w$ are also in the same stage in $\TT$ since $k\notin C$.

Then their children in $\TT$ and $\TTb$ are
 \begin{align*}
\ch(v)&=\{x_1\cdots x_{k-1}s: s\in [d_{k}]\},\\
\ch(w)&=\{y_1\cdots y_{k-1}s : s\in[d_{k}] \}.
\end{align*}
Let $v_1,v_2\in \ch(v)$ and $w_1,w_2\in \ch(w)$ be such that $\theta(v\to v_i)=\theta(w\to w_i)$, $(i=1,2)$. Since CStrees are compatibly labeled, then
\begin{align*}
v_1=&\,\,x_1\cdots x_{k-1}s , & v_2=&\,\,x_1\cdots x_{k-1}r, \\
w_1=&\,\,y_1\cdots y_{k-1}s , & w_2=&\,\,y_1\cdots y_{k-1}r, 
\end{align*}
for some $s,r \in [d_k]$.
We want to show $t_{\TTb}(v_1)t_{\TTb}(w_2)=t_{\TTb}(w_1)t_{\TTb}(v_2)$. Choose a monomial on the left-hand-side. This corresponds to a product of edge labels of two paths $\lambda_1'$ and $\lambda_2'$ in $\TTb$, $\lambda_1'$ is a path from $v_1$ to a leaf and $\lambda_2'$ is a path from $w_2$ to a leaf. Each leaf in  $\TTb$ is
also a leaf in $\TT$ (Section~\ref{sec:cstrees}).
In $\TT$ there exists a directed path $\lambda_1$ from $v_1$ to the leaf using all edges in $\lambda_1'$ and a directed path $\lambda_2$ from $w_2$ to the other leaf using $\lambda_2'$.

Since $v,w$ are in the same stage in $\TTb$, they are also in the same stage in $\TT$. The balanced condition in $\TT$ implies
\begin{align}
t_{\TT}(x_1\dotsm x_{k-1}s)t_{\TT}(y_1\dotsm y_{k-1}r)=t_{\TT}(x_1\dotsm x_{k-1}r)t_{\TT}(y_1\dotsm y_{k-1}s). \label{eq:balanced-for-every-context}
\end{align}
Choose the product of monomials on the left hand side of this equation
that is the product of the edge labels in the concatenation of paths $\lambda_1\lambda_2$ and  denote it by 
$\theta(\lambda_1)\theta(\lambda_2)$.
Since $\TT$ is balanced, it follows from  the bijection in Lemma~\ref{lem:balancecheck} that there exists a  product 
$\theta(\lambda_3)\theta(\lambda_4)$ corresponding to paths $\lambda_3,\lambda_4$ in $\TT$ from $v_2$ to a leaf and $w_1$ to a leaf on the right-hand side of (\ref{eq:balanced-for-every-context}) such that 
\begin{align}
\label{prod}
\theta(\lambda_1)\theta(\lambda_2)=\theta(\lambda_3)\theta(\lambda_4).
\end{align}
We claim that the paths $\lambda_3,\lambda_4$ are paths in $\TTb$, i.e.
the nodes in the paths $\lambda_3,\lambda_4$ contract to nodes in $\TTb$.
Let $j\in [p]\setminus [k]$ and denote by $e_{i,j}$ the edge of the path $i$, ($i=1,2,3,4$) from level $j$ to level $j+1$. The fact that $\TT$
is stratified and (\ref{prod}) holds, implies 
\begin{align} \label{eq:strat}
\theta(e_{1,j})\theta(e_{2,j})=\theta(e_{3,j})\theta(e_{4,j}) \text{ for all } j\in [p]\setminus [k].
\end{align}
If $j+1\in C$ then the edges $e_{1,j}, e_{2,j}$ point to the same outcome $(\xx_C)_{j+1}$. From (\ref{eq:strat}) and since $\TT$ is
compatibly labeled, different outcomes can never have equal edge labels, thus $e_{3,j}, e_{4,j}$ must also point to the outcome $(\xx_C)_{j+1}$.
This shows $\lambda_3,\lambda_4$ are paths in $\TTb$. 

Finally, if $j+1\notin C$ then  (\ref{eq:strat}) implies that the product of the edge labels of the restrictions $\lambda_3', \lambda_4'$ of $\lambda_3,\lambda_4$ to paths in $\TTb$ is equal to the product of the edge labels of $\lambda_1',\lambda_2'$. This establishes a bijection between terms on the right-hand side  and the left-hand side of 
$t_{\TTb}(v_1)t_{\TTb}(w_2)=t_{\TTb}(w_1)t_{\TTb}(v_2)$, which means $\TTb$ is balanced.
\end{proof}

\subsection{Decomposable DAG models and decomposable CSmodels}
We start  by proving the one-sided generalization of Theorem~\ref{thm:three_perfect_iff_balanced} and Theorem~\ref{thm:perfect-iff-balanced-iff-decomposable} to CStrees.

\begin{theorem}
\label{thm:perf_implies_balanced}
Let $\TT$ be a CStree with only perfect minimal contexts. Then $\TT$ is balanced.
\end{theorem}
\begin{proof}
Let $v=x_1\dots x_{k-1},w=y_1\dots y_{k-1}\in V_\TT$ be two vertices in the same stage $S$ in $\TT$. Then $S$  has a stage defining context $C$ that entails the CSI relation
\[
X_k\independent X_{[k-1]\setminus C}| X_C=\xx_C
\]
for some $\xx_C\in\RR_C$. By definition  $C\subset [k-1]$, and from  \cite[Lemma 3.2]{DS22} there exists a minimal context $C'\subset C$ such that
\[
X_k\independent X_{[k-1]\setminus C}| X_{C\setminus C'}, X_{C'}=\xx_{C'}
\]
holds with $\xx_{C'}=(\xx_C)_{C'}$. Every node in $S$ contains the
context $\xx_{C'}$, hence  every node in $S$ appears in $\TT_{X_{C'}=\xx_{C'}}$ and by construction  $S$ is a stage in $\TT_{X_{C'}=\xx_{C'}}$.
We claim that $v$ and $w$ are also in the same stage in $\TT_{G_{X_{C'}=\xx_{C'}}}$:

By \cite[Proposition 2.2]{DS21}, this holds if and only if $(v)_{\pa_{G_{X_{C'}=\xx_{C'}}}(k)}=(w)_{\pa_{G_{X_{C'}=\xx_{C'}}}(k)}$. That is, the entries of $v$ and $w$ agree for the indices in $\pa_{G_{X_{C'}=\xx_{C'}}}(k)$.
Let $i\in \pa_{G_{X_{C'}=\xx_{C'}}}(k)$ then $X_k\not\independent X_i| X_{C'}=\xx_{C'}$. Therefore $i\notin [k-1]\setminus C$, i.e. $i\in C\setminus C'$ because we are in the context $X_{C'}=\xx_{C'}$.
Since $C$ is the stage defining context of $S$, we have $x_i=y_i$.

Since $G_{X_{C'}=\xx_{C'}}$ is perfect by assumption, the nodes $v,w$ are balanced in the CStree $\TTb:=\TT_{G_{X_{C'}=\xx_{C'}}}$ by \cite[Theorem 3.1]{DS21}.
This means that for any $v_1,v_2\in\children_{\TTb}(v)$ and $w_1,w_2\in\children_{\TTb}(w)$ with $\theta_{\TTb}(v\to v_i)=\theta_{\TTb}(w\to w_i),\, (i=1,2)$ the following equation holds
\[
t_{\TTb}(v_1)t_{\TTb}(w_2)=t_{\TTb}(v_2)t_{\TTb}(w_1)
\]
in $\R[\Theta_{\TTb}]$. Since there is a surjective ring homomorphism $\R[\Theta_{\TTb}]\to \R[\Theta_{\TT_{X_{C'}=\xx_{C'}}}]$ the same equation 
\[
t_{\TT_{X_{C'}=\xx_{C'}}}(v_1)t_{\TT_{X_{C'}=\xx_{C'}}}(w_2)=t_{\TT_{X_{C'}=\xx_{C'}}}(v_2)t_{\TT_{X_{C'}=\xx_{C'}}}(w_1)
\]
holds in $\R[\Theta_{\TT_{X_{C'}=\xx_{C'}}}]$. By Lemma~\ref{lem:rooted_subtrees} (ii) we have $t_\TT(v)=t_{\TT_{X_{C'}=\xx_{C'}}}(v)$ and hence the equality
\[
t_{\TT}(v_1)t_{\TT}(w_2)=t_{\TT}(v_2)t_{\TT}(w_1)
\]
holds in $\R[\Theta_{\TT}]$, i.e. the nodes $v,w$ are balanced.
\end{proof}

\begin{proposition}
\label{prop:weird_formula}
Let $\TT$ be a CStree. Let $A,B,C\subset [p]$ be pairwise disjoint with $A\cup B\cup C=[k-1]$ and fix $\xx_A\in\RR_A, \xx_B\in\RR_B,\xx_C\in\RR_C$. Then the following rule holds for the CSI statements in $\TT$:
\[
X_k\independent X_A|X_{B\cup C}=\xx_B\xx_C \text{ and } X_k\independent X_B| X_{A\cup C}=\xx_A\xx_C \Rightarrow X_k\independent X_{A\cup B}|X_C=\xx_C.
\]
\end{proposition}
\begin{proof}
In level $k-1$ we have the two stages
\[
S_1=\bigcup_{\yy_A\in \RR_A} \{\yy_A\xx_B\xx_C\},\quad 
S_2=\bigcup_{\yy_B\in \RR_B} \{\xx_A\yy_B\xx_C\}.
\]
However, these are both contained in a single stage: 
Both contain the element $\xx_A\xx_B\xx_C$. But different stages cannot intersect, hence the two are contained in a single stage $S$.

Let $\yy_A\neq \xx_A$ and $\yy_B\neq \xx_B$. The elements $\xx_A\yy_B\xx_C$ and $\yy_A\xx_B\xx_C$ are contained in $S$ and therefore $\zz_A\zz_B\xx_C\in S$ for every $\zz_A\in \RR_A$, $\zz_B\in \RR_B$. But this means $X_k\independent X_{A\cup B}|X_C=\xx_C$.
\end{proof}

    In terms of context DAGs the last lemma says the following: If in a context DAG $G_{\cc,\emptyset}$ there is an edge $i\to j$, i.e. $X_i\not\independent X_j|\cc$, but $X_i\independent X_j|\cc,\Cp$, then for every $v\in C'$ there is an edge $v\to j$.
    The lemma can also be understood as a stronger but slightly different version in CStrees of the intersection axiom 
    \[
    X_A\independent X_B|X_{S\cup D}, \cc,\, X_A\independent X_D|X_{S\cup B}, \cc \Rightarrow X_A\independent X_{B\cup D}|X_S, \cc
    \]
    as it only requires the first two CSI statements to each hold in one context $X_D=\xx_D$ and $X_B=\xx_B$.

\begin{example}
    Consider a CStree $\TT$ with empty minimal context DAG $G_\emptyset=([4],\{1\to 2, 2\to 3, 2\to 4\})$.
Lemma~\ref{prop:weird_formula} implies that this CStree is in fact the staged tree representation of the DAG $G_\emptyset$. Indeed, there is no vertex with at least two incoming edges which implies that any CSI statement in $\TT$ is already a specialization of a CI statement.
\end{example}

Using these observations, one can see that the CStrees given in Example~\ref{ex:p=3} are in fact all CStrees on $p=3$ variables.

\begin{proof}[Proof of Proposition \ref{prop:cstrees_three_variables}] 
    Let $G$ be the empty context DAG of the CStree $\TT$. If there is no vertex with two incoming edges, the CStree is the staged tree representation of a DAG by Proposition~\ref{prop:weird_formula}. Thus either the empty context DAG ist $1\to 3\leftarrow 2$ or the complete graph on three vertices. In either case by the same observation, the only other context DAGs are DAGs on the two vertices $1,3$ or $2,3$. In order to imply a CSI statement they cannot contain the edge. 

    Assume there is a CStree with edges $1\to 3$, $2\to 3$ in the empty context DAG and at least two more minimal contexts $X_1=x_1^i$ and $X_2=x_2^j$ for some outcomes of $X_1$ and $X_2$. We claim that this is impossible. Indeed, we have the following CSI statements in $\TT$:
    \[
    X_3\independent X_2|X_1=x_1^i\quad \text{and}\quad X_3\independent X_1|X_2=x_2^j.
    \]
    Using Proposition~\ref{prop:weird_formula} again, we see that the CSI statement $X_3\independent X_{1,2}$ holds, i.e. the empty context DAG does not have the edges $1\to 3$, $2\to 3$, a contradiction.
Thus either all minimal contexts fix outcomes of $X_1$ or all fix outcomes of $X_2$.
\end{proof}

To generalize the other implication of Theorem~\ref{thm:perfect-iff-balanced-iff-decomposable} we use an algebraic approach presented in the next section.

\section{Algebraic characterization of decomposable context-specific models}\label{sec:algebra}

The core of this paper is Theorem~\ref{thm:algebra-saturated-statements}
as it provides a complete characterization of the CSI statements that hold in a decomposable CSmodel. From
this theorem we deduce the main properties of decomposable CSmodels stated in the introduction. In particular, it lays down the technical foundation upon which we build our main algebraic result, Theorem~\ref{thm:balanced_defined_by_perfect}, which 
shows that every decomposable CSmodel can be defined by a
collection of perfect DAGs.
Theorem~\ref{thm:algebra-saturated-statements}  states that for a balanced CStree $\TT$, 
the polynomials associated to saturated CSI statements are
a generating set of the prime ideal $\ker(\psi_{\TT})$ that defines $\MM(\TT)$ implicitly. This is precisely the case for perfect DAG models, see \cite[Theorem 4.4]{GMS06}, which once again highlights the important role that decomposable CSmodels play in generalizing the algebraic properties of single DAGs to collections of DAGs in the context-specific setting. The proof of this result uses the algebraic notion of the toric fiber product, first introduced in \cite[]{S07}.

For any collection $\CC$ of CSI statements in a CStree $\TT$, we define the ideal $I_\CC$ to be the ideal generated by the polynomials associated to all CSI statements in $\CC$ as defined in Section~\ref{subsec:csis}.

\subsection{Setup}
Let $\TT$ be a balanced CStree, $\TTb$ the subtree of $\TT$ up to level $p-1$ and $S_1,\ldots, S_r$ the stages in $\TT$ in level $p-1$.
Let $\TT_p=\bigcup_{i\in[r]}\mathcal{B}_i$, where each $\mathcal{B}_i$ is a one-level tree together with its edge labels as in \cite[Section 3]{AD19}.
Consider the rings
\begin{align*}
\begin{array}{lll}
\R[\TTb]& :=&\R[p_{\xx}^{i}\colon i\in [r], \xx\in S_i], \\ 
    \R[\TT_p]& :=&\R[p_{k}^{i}\colon i\in [r], k\in [d_p]],\\
\R[\TT]& :=&\R[p_{\xx k}^{i} \colon i\in [r], \xx \in S_i,k\in[d_p]] 
\end{array}
\end{align*}
with multigrading $\deg(p_{\xx}^{i})=\deg(p_{k}^{i})=\deg(p_{\xx k}^{i}), i\in[r], \xx \in S_i, k\in [d_p]$ 
where
$\mathcal{A}=\{e_1,\ldots, e_r\}$ and $e_i$ is the $i$-th standard unit vector in $\mathbb{Z}^{r}$. Note that the rings $\R[\TT]$ and
$\R[D]$ are the same, except the former is multigraded. Consider the ring homomorphism 
\[
\R[\TT]\to \R[\TTb]\otimes \R[\TT_p],\quad p_{\xx k}^i\mapsto p_{\xx}^i\otimes p_k^i\quad (i\in [r], \xx\in S_i, k\in [d_p]).
\]
Following \cite[]{S07}, we call the kernel of this map $\Quad$. It is given by
\[
\Quad=\langle p_{\xx k_1}^{i}p_{\yy k_2}^{i} - p_{\xx k_2}^{i}p_{\yy k_1}^i : k_1\neq k_2 \in [d_p], \xx, \yy \in S_i, i \in [r] \rangle.
\]
Note that the generators of $\Quad$ are the $2\times 2$ minors of the matrices $(p_{\xx k}^{i})_{\xx\in S_i,k\in [d_p]}$ for all $i\in [r]$.
Now, consider the ring homomorphism
$$\R[\TT]\to \R[\TTb]/\ker(\psi_{\TTb})\otimes \R[\TT_p],\quad p_{\xx k}^i\mapsto p_{\xx}^i\otimes p_k^i\quad (i\in [r], \xx\in S_i, k\in [d_p]),$$
where $\psi_\TT$ is the homomorphism defined in (\ref{eq:psi-t}). The kernel of this map is the \textit{toric fiber product} of $\ker(\psi_{\TTb})$ and the zero ideal $\langle 0 \rangle\subseteq \R[\TT_p]$, and is denoted by $\ker(\psi_{\TTb})\times_\mathcal{A}\langle 0\rangle$. By \cite[Proposition 3.5]{AD19}, this toric fiber product is equal to $\ker(\psi_\TT)$ when $\TT$ is balanced. The generators of $\ker(\psi_{\TT})$ are obtained from two sets, namely $\ker(\psi_\TT)=\langle\Quad,\Lift(F)\rangle$, where $F$ is a set of generators of $\ker(\psi_{\TTb})$ and
$$\Lift(F):=\{p^i_{\xx_1k_1}p^j_{\yy_1k_2}-p^i_{\xx_2k_1}p^j_{\yy_2k_2}:\xx_1,\xx_2\in S_i, \yy_1,\yy_2\in S_j, k_1,k_2\in [d_p], p^i_{\xx_1}p^j_{\yy_1}-p^i_{\xx_2}p^j_{\yy_2}\in F\}.$$
Note that the construction above relies heavily on the fact that $\TT$ is balanced since this is the only case in which 
$\ker(\psi_{\TT})$ and $\ker(\psi_{\TTb})$ are toric and  $\mathcal{A}$-homogeneous.
\subsection{Main results}
In what follows saturated CSI statements will be the main actors. Let $\TT$ be a CStree with $p$ levels and let $\mathcal{M}(\TT)$ be the associated model. A \textit{saturated CSI statement} is a CSI statement of the form $X_A\independent X_B|X_S, \cc$, where $A\cup B\cup C\cup S=[p]$.
If $\CC$ is a collection of saturated CSI statements then the ideal $I_{\CC}$ is generated by binomials. Any ideal that is generated by binomials and in addition is prime is a \textit{toric} ideal.

\begin{definition}
Let $\CC$ be any collection of CSI statements of random variables $X_1,\dots, X_p$. We define $\sat(\CC)$ to be the set of all saturated CSI statements in $\CC$. 
For a CStree $\TT$ let $\sat(\TT):=\sat(\mathcal{J}(\TT))$ where $\mathcal{J}(\TT)$ denotes the set of all CSI statements that hold in $\TT$.
For a DAG $G$ we define $\sat(G):=\sat(\TT_G)$. Since $\mathcal{J}(\TT_G)=\glo(G)$, we also get $\sat(G)=\sat(\glo(G))$. 
\end{definition}

The proofs of the results in this section rely heavily on the toric fiber product construction. We motivate these results with the following concrete example.

\begin{figure}[t]

\begin{center}
\begin{tikzpicture}[thick,scale=0.3]
\node[circle, draw, fill=black!0, inner sep=2pt, minimum width=1pt] () at (4.4,0.2)  {};
\node[circle, draw, fill=black!0, inner sep=2pt, minimum width=1pt] () at (4.4,-0.2)  {};
\node[circle, draw, fill=black!0, inner sep=2pt, minimum width=1pt] () at (4.4,-0.8)  {};
\node[circle, draw, fill=black!0, inner sep=2pt, minimum width=1pt] () at (4.4,-1.2)  {};
\node[circle, draw, fill=black!0, inner sep=2pt, minimum width=1pt] () at (4.4,-1.8)  {};
\node[circle, draw, fill=black!0, inner sep=2pt, minimum width=1pt] () at (4.4,-2.2)  {};
\node[circle, draw, fill=black!0, inner sep=2pt, minimum width=1pt] () at (4.4,-2.8)  {};
\node[circle, draw, fill=black!0, inner sep=2pt, minimum width=1pt] () at (4.4,-3.2)  {};
\node[circle, draw, fill=black!0, inner sep=2pt, minimum width=1pt] () at (4.4,-3.8)  {};
\node[circle, draw, fill=black!0, inner sep=2pt, minimum width=1pt] () at (4.4,-4.2)  {};
\node[circle, draw, fill=black!0, inner sep=2pt, minimum width=1pt] () at (4.4,-3.8)  {};
\node[circle, draw, fill=black!0, inner sep=2pt, minimum width=1pt] () at (4.4,-4.2)  {};
\node[circle, draw, fill=black!0, inner sep=2pt, minimum width=1pt] () at (4.4,-4.8)  {};
\node[circle, draw, fill=black!0, inner sep=2pt, minimum width=1pt] () at (4.4,-5.2)  {};
\node[circle, draw, fill=black!0, inner sep=2pt, minimum width=1pt] () at (4.4,-5.8)  {};
\node[circle, draw, fill=black!0, inner sep=2pt, minimum width=1pt] () at (4.4,-6.2)  {};
\node[circle, draw, fill=black!0, inner sep=2pt, minimum width=1pt] () at (4.4,-6.8)  {};
\node[circle, draw, fill=black!0, inner sep=2pt, minimum width=1pt] () at (4.4,-7.2)  {};
\node[circle, draw, fill=black!0, inner sep=2pt, minimum width=1pt] () at (4.4,-7.8)  {};
\node[circle, draw, fill=black!0, inner sep=2pt, minimum width=1pt] () at (4.4,-8.2)  {};
\node[circle, draw, fill=black!0, inner sep=2pt, minimum width=1pt] () at (4.4,-8.8)  {};
\node[circle, draw, fill=black!0, inner sep=2pt, minimum width=1pt] () at (4.4,-9.2)  {};
\node[circle, draw, fill=black!0, inner sep=2pt, minimum width=1pt] () at (4.4,-9.8)  {};
\node[circle, draw, fill=black!0, inner sep=2pt, minimum width=1pt] () at (4.4,-10.2)  {};
\node[circle, draw, fill=black!0, inner sep=2pt, minimum width=1pt] () at (4.4,-10.8)  {};
\node[circle, draw, fill=black!0, inner sep=2pt, minimum width=1pt] () at (4.4,-11.2)  {};
\node[circle, draw, fill=black!0, inner sep=2pt, minimum width=1pt] () at (4.4,-11.8)  {};
\node[circle, draw, fill=black!0, inner sep=2pt, minimum width=1pt] () at (4.4,-12.2)  {};
\node[circle, draw, fill=black!0, inner sep=2pt, minimum width=1pt] () at (4.4,-12.8)  {};
\node[circle, draw, fill=black!0, inner sep=2pt, minimum width=1pt] () at (4.4,-13.2)  {};
\node[circle, draw, fill=black!0, inner sep=2pt, minimum width=1pt] () at (4.4,-13.8)  {};
\node[circle, draw, fill=black!0, inner sep=2pt, minimum width=1pt] () at (4.4,-14.2)  {};
\node[circle, draw, fill=black!0, inner sep=2pt, minimum width=1pt] () at (4.4,-14.8)  {};
\node[circle, draw, fill=black!0, inner sep=2pt, minimum width=1pt] () at (4.4,-15.2)  {};

 	 \node[circle, draw, fill=black!10, inner sep=2pt, minimum width=1pt] (w3) at (0,0)  {};
 	 \node[circle, draw, fill=black!10, inner sep=2pt, minimum width=1pt] (w4) at (0,-1) {};
 	 \node[circle, draw, fill=purple!50, inner sep=2pt, minimum width=1pt] (w5) at (0,-2) {};
 	 \node[circle, draw, fill=purple!50, inner sep=2pt, minimum width=1pt] (w6) at (0,-3) {};
 	 \node[circle, draw, fill=red, inner sep=2pt, minimum width=1pt] (v3) at (0,-4)  {};
 	 \node[circle, draw, fill=blue!50, inner sep=2pt, minimum width=1pt] (v4) at (0,-5) {};
 	 \node[circle, draw, fill=red, inner sep=2pt, minimum width=1pt] (v5) at (0,-6) {};
 	 \node[circle, draw, fill=blue!50, inner sep=2pt, minimum width=1pt] (v6) at (0,-7) {};
 	 \node[circle, draw, fill=black!50, inner sep=2pt, minimum width=2pt] (w3i) at (0,-8)  {};
 	 \node[circle, draw, fill=black!50, inner sep=2pt, minimum width=2pt] (w4i) at (0,-9) {};
 	 \node[circle, draw, fill=orange!50, inner sep=2pt, minimum width=2pt] (w5i) at (0,-10) {};
 	 \node[circle, draw, fill=orange!50, inner sep=2pt, minimum width=2pt] (w6i) at (0,-11) {};
 	 \node[circle, draw, fill=olive!50, inner sep=2pt, minimum width=2pt] (v3i) at (0,-12)  {};
 	 \node[circle, draw, fill=olive!50, inner sep=2pt, minimum width=2pt] (v4i) at (0,-13) {};
 	 \node[circle, draw, fill=orange, inner sep=2pt, minimum width=2pt] (v5i) at (0,-14) {};
 	 \node[circle, draw, fill=orange, inner sep=2pt, minimum width=2pt] (v6i) at (0,-15) {};

	 \node[circle, draw, fill=green!60!black!60, inner sep=2pt, minimum width=2pt] (w1) at (-4,-.5) {};
 	 \node[circle, draw, fill=green!60!black!60, inner sep=2pt, minimum width=2pt] (w2) at (-4,-2.5) {}; 
 	 \node[circle, draw, fill=red!60!black!60, inner sep=2pt, minimum width=2pt] (v1) at (-4,-4.5) {};
 	 \node[circle, draw, fill=red!60!black!60, inner sep=2pt, minimum width=2pt] (v2) at (-4,-6.5) {};	 
	 \node[circle, draw, fill=yellow, inner sep=2pt, minimum width=2pt] (w1i) at (-4,-8.5) {};
 	 \node[circle, draw, fill=yellow, inner sep=2pt, minimum width=2pt] (w2i) at (-4,-10.5) {};
 	 \node[circle, draw, fill=cyan, inner sep=2pt, minimum width=2pt] (v1i) at (-4,-12.5) {};
 	 \node[circle, draw, fill=cyan, inner sep=2pt, minimum width=2pt] (v2i) at (-4,-14.5) {};

 	 \node[circle, draw, fill=pink, inner sep=2pt, minimum width=2pt] (w) at (-8,-1.5) {};
 	 \node[circle, draw, fill=blue!80, inner sep=2pt, minimum width=2pt] (v) at (-8,-5.5) {};	
 	 \node[circle, draw, fill=blue!20, inner sep=2pt, minimum width=2pt] (wi) at (-8,-9.5) {};
 	 \node[circle, draw, fill=olive, inner sep=2pt, minimum width=2pt] (vi) at (-8,-13.5) {};
	 
     \node[circle, draw, fill=red!60, inner sep=2pt, minimum width=2pt] (r) at (-12,-3.5) {};
     \node[] () at (-12,-2.5) {\tiny{$0$}};
 	 \node[circle, draw, fill=green!60, inner sep=2pt, minimum width=2pt] (ri) at (-12,-11.5) {};
     \node[] () at (-12,-10.5) {\tiny{$1$}};

 	 \node[circle, draw, fill=black!0, inner sep=2pt, minimum width=2pt] (I) at (-16,-7.5) {};

   \node[] () at (-14,1.4) {$X_1$};
   \node[] () at (-10,1.4) {$X_2$};
   \node[] () at (-6,1.4) {$X_3$};
   \node[] () at (-2,1.4) {$X_4$};
   \node[] () at (2.2,1.4) {$X_5$};

 	 \draw[->]   (I) --    (r);
 	 \draw[->]   (I) --   (ri) ;

 	 \draw[->]   (r) --   (w) ;
 	 \draw[->]   (r) --   (v) ;

 	 \draw[->]   (w) --  (w1) ;
 	 \draw[->]   (w) --  (w2) ;

 	 \draw[->]   (w1) --   (w3) ;
 	 \draw[->]   (w1) --   (w4) ;
 	 \draw[->]   (w2) --  (w5) ;
 	 \draw[->]   (w2) --  (w6) ;

 	 \draw[->]   (v) --  (v1) ;
 	 \draw[->]   (v) --  (v2) ;

 	 \draw[->]   (v1) --  (v3) ;
 	 \draw[->]   (v1) --  (v4) ;
 	 \draw[->]   (v2) --  (v5) ;
 	 \draw[->]   (v2) --  (v6) ;

 	 \draw[->]   (ri) --   (wi) ;
 	 \draw[->]   (ri) -- (vi) ;

 	 \draw[->]   (wi) --  (w1i) ;
 	 \draw[->]   (wi) --  (w2i) ;

 	 \draw[->]   (w1i) --  (w3i) ;
 	 \draw[->]   (w1i) -- (w4i) ;
 	 \draw[->]   (w2i) --  (w5i) ;
 	 \draw[->]   (w2i) --  (w6i) ;

 	 \draw[->]   (vi) --  (v1i) ;
 	 \draw[->]   (vi) --  (v2i) ;

 	 \draw[->]   (v1i) --  (v3i) ;
 	 \draw[->]   (v1i) -- (v4i) ;
 	 \draw[->]   (v2i) -- (v5i) ;
 	 \draw[->]   (v2i) --  (v6i) ;
 	 
 	 \draw[->]   (w3) --   (4,0.2) ;
 	 \draw[->]   (w3) --   (4,-0.2) ;
 	 \draw[->]   (w4) --   (4,-0.8) ;
 	 \draw[->]   (w4) --   (4,-1.2) ;
 	 \draw[->]   (w5) --   (4,-1.8) ;
 	 \draw[->]   (w5) --   (4,-2.2) ;
 	 \draw[->]   (w6) --   (4,-2.8) ;
 	 \draw[->]   (w6) --   (4,-3.2) ;
 	 \draw[->]   (v3) --   (4,-3.8) ;
 	 \draw[->]   (v3) --   (4,-4.2) ;
 	 \draw[->]   (v4) --   (4,-4.8) ;
 	 \draw[->]   (v4) --   (4,-5.2) ;
 	 \draw[->]   (v5) --   (4,-5.8) ;
 	 \draw[->]   (v5) --   (4,-6.2) ;
 	 \draw[->]   (v6) --   (4,-6.8) ;
 	 \draw[->]   (v6) --   (4,-7.2) ;
 	 \draw[->]   (w3i) --   (4,-7.8) ;
 	 \draw[->]   (w3i) --   (4,-8.2) ;
 	 \draw[->]   (w4i) --   (4,-8.8) ;
 	 \draw[->]   (w4i) --   (4,-9.2) ;
 	 \draw[->]   (w5i) --   (4,-9.8) ;
 	 \draw[->]   (w5i) --   (4,-10.2) ;
 	 \draw[->]   (w6i) --   (4,-10.8) ;
 	 \draw[->]   (w6i) --   (4,-11.2) ;
 	 \draw[->]   (v3i) --   (4,-11.8) ;
 	 \draw[->]   (v3i) --   (4,-12.2) ;
 	 \draw[->]   (v4i) --   (4,-12.8) ;
 	 \draw[->]   (v4i) --   (4,-13.2) ;
 	 \draw[->]   (v5i) --   (4,-13.8) ;
 	 \draw[->]   (v5i) --   (4,-14.2) ;
 	  \draw[->]   (v6i) --   (4,-14.8) ;
 	 \draw[->]   (v6i) --   (4,-15.2) ;

\end{tikzpicture}
\hspace{1cm}
\begin{tikzpicture}[thick,scale=0.29]
	\draw (0,0) -- (18,0) -- (18,-18) -- (0,-18) -- cycle;
    \draw (9,0) -- (9,-18);
    \draw (0,-9) -- (18,-9);

	 \node[circle, draw, fill=black!0, inner sep=1pt, minimum width=1pt] (1v1) at (1.5,-2.5) {$1$};
 	 \node[circle, draw, fill=black!0, inner sep=1pt, minimum width=1pt] (1v2) at (7.5,-2.5) {$2$};
 	 \node[circle, draw, fill=black!0, inner sep=1pt, minimum width=1pt] (1v3) at (1.5,-6) {$3$};
 	 \node[circle, draw, fill=black!0, inner sep=1pt, minimum width=1pt] (1v4) at (7.5,-6) {$4$}; 	 
 	 \node[circle, draw, fill=black!0, inner sep=1pt, minimum width=1pt] (1v5) at (4.5,-8) {$5$};

 	 \node[circle, draw, fill=black!0, inner sep=1pt, minimum width=1pt] (2v2) at (16.5,-2.5) {$2$};
 	 \node[circle, draw, fill=black!0, inner sep=1pt, minimum width=1pt] (2v3) at (10.5,-6) {$3$};
 	 \node[circle, draw, fill=black!0, inner sep=1pt, minimum width=1pt] (2v4) at (16.5,-6) {$4$}; 	 
 	 \node[circle, draw, fill=black!0, inner sep=1pt, minimum width=1pt] (2v5) at (13.5,-8) {$5$};

	 \node[circle, draw, fill=black!0, inner sep=1pt, minimum width=1pt] (3v1) at (1.5,-11.5) {$1$};
 	 \node[circle, draw, fill=black!0, inner sep=1pt, minimum width=1pt] (3v3) at (1.5,-15) {$3$};
 	 \node[circle, draw, fill=black!0, inner sep=1pt, minimum width=1pt] (3v4) at (7.5,-15) {$4$}; 	 
 	 \node[circle, draw, fill=black!0, inner sep=1pt, minimum width=1pt] (3v5) at (4.5,-17) {$5$};
 	 
 	 \node[circle, draw, fill=black!0, inner sep=1pt, minimum width=1pt] (4v3) at (10.5,-15) {$3$};
 	 \node[circle, draw, fill=black!0, inner sep=1pt, minimum width=1pt] (4v4) at (16.5,-15) {$4$}; 	 
 	 \node[circle, draw, fill=black!0, inner sep=1pt, minimum width=1pt] (4v5) at (13.5,-17) {$5$};

 	 \draw[->]   (1v1) -- (1v2) ;
 	 \draw[->]   (1v1) -- (1v3) ;
 	 \draw[->]   (1v1) -- (1v4) ;
 	 \draw[->]   (1v1) -- (1v5) ;
 	 \draw[->]   (1v2) -- (1v3) ;
 	 \draw[->]   (1v2) -- (1v4) ;
 	 \draw[->]   (1v2) -- (1v5) ;
 	 \draw[->]   (1v3) -- (1v5) ;
 	 \draw[->]   (1v4) -- (1v5) ;

 	 \draw[->]   (2v2) -- (2v3) ;
 	 \draw[->]   (2v2) -- (2v4) ;
 	 \draw[->]   (2v2) -- (2v5) ;
 	 \draw[->]   (2v3) -- (2v5) ;
 	 
 	 \draw[->]   (3v1) -- (3v3) ;
 	 \draw[->]   (3v1) -- (3v4) ;
 	 \draw[->]   (3v1) -- (3v5) ;
 	 \draw[->]   (3v3) -- (3v5) ;

 	 \draw[->]   (4v3) -- (4v5) ;

	 \node at (1.15,-1) {$G_{\emptyset}$} ;
	 \node at (11.2,-1) {$G_{X_1=1}$} ;
	 \node at (2.2,-10) {$G_{X_2 = 0}$} ;
	 \node at (12.2,-10) {$G_{X_1X_2 = 01}$} ;
\end{tikzpicture}
\caption{\small{Balanced CStree on five binary random variables whose empty context DAG is not perfect.} 
} 
\label{fig:motivatingEx}
\end{center}
\end{figure}
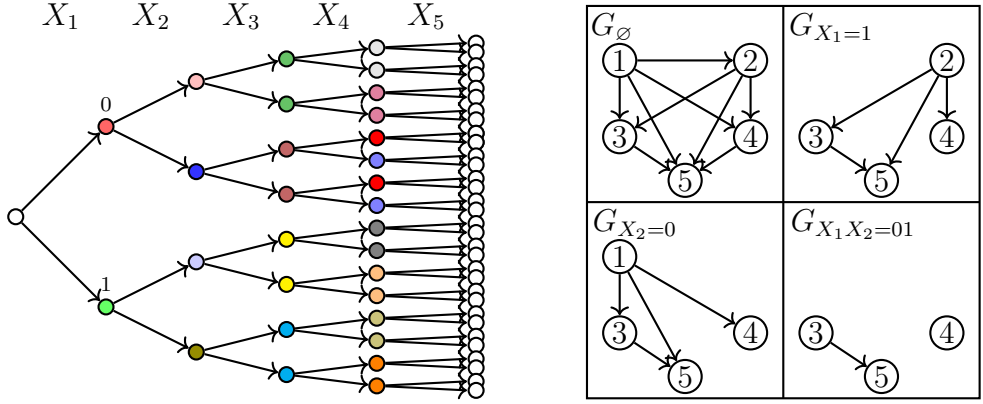

\begin{example}[Decomposable CSmodel with a non-perfect empty context] \label{ex:nonperfsparse}
Consider the decomposable CSmodel given by the CStree in Figure~\ref{fig:motivatingEx}
for $p=5$. It has four minimal contexts, namely, $$\CC_\TT=\{\varnothing,\; X_1=1,\; X_1X_2=01,\; X_2=0\}.$$ Only the non-empty minimal contexts are perfect, yet the tree is balanced. The CSI statements corresponding to the four minimal contexts are, respectively, 
\begin{align*}
    &X_3\independent X_4|X_1X_2, &X_4\independent X_5|X_2X_3,X_1=1,\\
    &X_4\independent X_5|X_3,X_1X_2=01, &X_4\independent X_5|X_1X_3,X_2=0.
\end{align*}

The last three statements, corresponding to the three perfect minimal contexts, are saturated. 
Applying the contraction axiom to each of these statements together with the appropriate specialization of the statement $X_3\independent X_4|X_1X_2$ (corresponding to the empty context), we get the following three saturated statements 
$$X_3X_5\independent X_4|X_2,X_1=1,\; X_3X_5\independent X_4|X_1X_2=01,\; X_3X_5\independent X_4|X_1,X_2=0.$$
These three saturated statements give rise to 24 polynomials, 8 of which coincide with stage-defining statements for level 5. These 8 polynomials, one of which is
$$p_{00000}^{1}p_{00011}^1-p_{00001}^1p_{00010}^1,$$
are precisely the polynomials in $\Quad$. The remaining 16 polynomials, one of which is
$$p_{00000}^1p_{00110}^2-p_{00010}^1p_{00100}^2,$$ are the polynomials in $\Lift(F)$, defined above. Hence, the 24 polynomials associated to the saturated statements are precisely the generators of $\ker(\psi_\TT)$. 
\end{example}

Turns out, the phenomenon in the example above can be generalized to all decomposable CSmodels. The next theorem is the technical foundation of this paper. It demonstrates the important role that saturated CSI statements play in the algebra of decomposable CSmodels and it also contains most of the technical work in its proof.

\begin{theorem}
\label{thm:algebra-saturated-statements}
If $\mathcal{M}(\TT)$ is a decomposable CSmodel, then $\ker(\psi_\TT)$ is generated by the quadratic binomials associated to all saturated CSI statements in $\mathcal{J}(\TT)$, i.e. 
\[
\ker(\psi_\TT)=I_{\sat(\TT)}.
\]
\end{theorem}
\begin{proof}
The containment 
$I_{\sat(\TT)}\subset \ker(\psi_{\TT})$ holds because all polynomials associated to statements in $\mathcal{J}(\TT)$ belong to 
$ \ker(\psi_{\TT})$. In particular, all binomials coming from saturated CSI statements are in 
$\ker(\psi_{\TT})$. 

For the other containment we proceed by induction on the number
of random variables in $\TT$. The statement is trivially true for
$p=1,2$. Suppose that $\TT$ has $p$ levels, and any balanced CStree with less than $p$ levels satisfies the statement. 
Let $\TTb$ be the subtree of $\TT$ up to level $p-1$. Then $\TTb$
is balanced, thus by induction hypothesis $\ker(\psi_{\TTb})$ is
generated by a set $F$ of binomials associated to saturated CSI statements in the variables $X_{[p-1]}$. Moreover, 
\[
\ker(\psi_{\TT})=\ker(\psi_{\TTb})\times_{\mathcal{A}}\langle 0\rangle =\langle\Quad, \Lift(F) \rangle.
\]
Hence, it suffices to prove 
that $\Quad$ and $\Lift(F)$ are polynomials associated to saturated 
CSI statements in the variables $X_{[p]}$. 
Let $S_1,\ldots, S_r$ be the stages of level $p-1$ in $\TT$.
For all $i\in [r]$, let $X_{C_i}=\xx_{C_i}$ be the stage defining context of the stage $S_i$. Recall that 
\[
\Quad=\langle p_{\xx k_1}^{i}p_{\yy k_2}^{i} - p_{\xx k_2}^{i}p_{\yy k_1}^i : k_1\neq k_2 \in [d_p], \xx, \yy \in S_i, i \in [r] \rangle.
\]
which, by (\ref{eqn:associated-polynomial-to-CSI-statements}), is precisely the set of binomials
associated to the saturated CSI statements $X_{p}\independent X_{[p-1]\setminus C_i}|X_{C_i}=\xx_{C_i}$ for all $i\in [r]$.

Since $\TT$ is balanced, $F$ is a set of $\mathcal{A}$-homogeneous binomials.
Let $g\in F$, then it is associated to a CSI statement $X_{A}\independent X_{B}|X_{D},X_{C}=\xx_C$ with $A\cup B\cup C \cup D =[p-1]$ and $\xx_C\in\RR_C$ as in (\ref{eqn:associated-polynomial-to-CSI-statements}). 

Choose $\yy_{A},\yy_A'\in \RR_{A}$ and $\yy_B,\yy_B'\in \RR_{B}$
such that for all $i\in A$,  $(\yy_{A})_i\neq (\yy_A')_i$,  and
for all $i\in B$, $(\yy_{B})_i\neq (\yy_B')_i$. Consider the polynomial
\[
h=p_{\yy_A \yy_B \xx_{C}\xx_{D}}^k p_{\yy_A' \yy_B' \xx_{C}\xx_{D}}^{\ell}-p_{\yy_A' \yy_B \xx_{C}\xx_{D}}^{m}p_{\yy_A \yy_B' \xx_{C}\xx_{D}}^{n}
\] associated to the same CSI statement as $g$ and its lift
\[
h_{z_1,z_2}=p_{\yy_A \yy_B \xx_{C}\xx_{D}z_1}^k p_{\yy_A' \yy_B' \xx_{C}\xx_{D}z_2}^{\ell}-
p_{\yy_A' \yy_B \xx_{C}\xx_{D}z_1}^{m}p_{\yy_A \yy_B' \xx_{C}\xx_{D}z_2}^{n},\;\; z_1,z_2\in [d_{p}].
\]
Since $h$ is $\mathcal{A}$-homogeneous,  either $(k,\ell)=(m,n)$ or $(k,\ell)=(n,m)$ . Assume it is the former. By the assignment of the grading, it follows that $(\yy_A \yy_B \xx_{C}\xx_{D}z_1)_{C_k}=(\yy_A' \yy_B \xx_{C}\xx_{D}z_1)_{C_k}$ because they are in the same stage, therefore $C_{k}\cap A =\emptyset$. The CSI statement associated to
the stage $S_k$ is $X_{p}\independent X_{[p-1]\setminus C_k}|X_{C_k}=\xx_{C_k}$, this entails $X_{p}\independent X_{A}|X_B=\yy_B, X_{C\cup D}=\xx_C\xx_D$ because  $C_{k}\cap A =\emptyset$.

For every $\zz_B,\zz_B'\in \RR_{B}$ there exist $\alpha,\beta\in [r]$ such that the binomial is either
\[
p_{\yy_A \zz_B \xx_{C}\xx_{D}z_1}^\alpha p_{\yy_A' \zz_B' \xx_{C}\xx_{D}z_2}^\beta-
p_{\yy_A' \zz_B \xx_{C}\xx_{D}z_1}^\alpha p_{\yy_A \zz_B' \xx_{C}\xx_{D}z_2}^\beta
\]
or
\[
p_{\yy_A \zz_B \xx_{C}\xx_{D}z_1}^\alpha p_{\yy_A' \zz_B' \xx_{C}\xx_{D}z_2}^\beta-
p_{\yy_A' \zz_B \xx_{C}\xx_{D}z_1}^\beta p_{\yy_A \zz_B' \xx_{C}\xx_{D}z_2}^\alpha
\]
depending on which variables have the same degree.

Case 1: For every $\zz_B$ and $\zz_B'$ entry-wise different, we have the first grading. Then by the same argument as for $\yy_B,\yy_B'$ we get the saturated CSI statement $X_{p}\independent X_{A}|X_B=\zz_B, X_{C\cup D}=\xx_C\xx_D$ for all $\zz_B$. Hence, by absorption we get 
\[
X_{p}\independent X_{A}|X_B, X_{C\cup D}=\xx_C\xx_D.
\]
Applying the contraction axiom to this statement and to $X_{A}\independent X_{B}|X_{D},X_{C}=\xx_C$ we get the saturated CSI statement
\[
X_A\independent X_{B\cup\{p\}}|X_{C\cup D}=\xx_C\xx_D.
\]
This statement entails all binomials in $\Lift(g)$, equivalently $\Lift(g)\subset I_{X_{A}\independent X_{B\cup \{p\}}|X_{C\cup D}=\xx_C\xx_D}$.

Case 2: There exists a pair $\zz_B,\zz_B'$, entry-wise different, such that the binomial has the second grading. 
Using the same argument as above with $B$ instead of $A$, this implies the statement $X_{p}\independent X_{B}|X_A=\yy_A, X_{C\cup D}=\xx_C\xx_D$. Combining this statement with $X_{p}\independent X_{A}|X_B=\yy_B, X_{C\cup D}=\xx_C\xx_D$ and Proposition~\ref{prop:weird_formula} we get
\[ 
X_p\independent X_{A\cup B}|X_{C\cup D}=\xx_C\xx_D.
\]
By the weak union axiom, we get $X_p\independent X_{A}|X_B,X_{C\cup D}=\xx_C\xx_D$. As in Case 1, we obtain the CSI statement
\[
X_A\independent X_{B\cup\{p\}}|X_{C\cup D}=\xx_C\xx_D
\]
and the conclusion $\Lift(g)\subset I_{X_{A}\independent X_{B\cup \{p\}}|X_{C\cup D}=\xx_C\xx_D}$ follows. The proof for the second choice of grading $(k,\ell)=(n,m)$ of $h$ is analogous, by swapping the roles of $A$ and $B$ in the above arguments.
\end{proof}
\begin{remark}
    The Theorem~\ref{thm:algebra-saturated-statements} implies
    that $\mathcal{V}(\ker(\psi_{\TT}))$ is the Zariski closure of $\MM(\TT)$.
    Therefore Theorem~\ref{thm:CShc} is also valid for distributions
    in the boundary of the probability simplex.
\end{remark}

For the rest of this section, we can relax the assumption of working with minimal contexts. Let $\TT$ be a CStree and $\CC$ be any collection of contexts with associated DAGs $G_{\cc}$, $\cc\in\CC$, such that $\mathcal{J}(\TT)=\cup_{\cc\in\CC}\glo(G_{\cc})$. That is, assume that the CSI statements that hold in $\CC$ are the same CSI statements that hold in $\TT$. The collection of minimal contexts is one such choice for $\CC$.

\begin{corollary}\label{cor:saturated-statements-for-minimal-dags}
    Let $\mathcal{M}(\TT)$ be a decomposable CSmodel. The ideal $\ker(\psi_{\TT})$ is generated by the binomials associated to all saturated CSI statements that hold in the context DAGs $G_{\cc}$, $\cc\in\CC$ , i.e. 
    $$\ker(\psi_\TT)=\sum_{X_C=\xx_C\in\CC}I_{\sat(G_{X_C=\xx_C})}.$$
\end{corollary}
\begin{proof}
    This follows from the fact that $\mathcal{J}(\TT)=\cup_{X_C=\xx_C\in\CC}\mathcal{J}({X_C=\xx_C})$ and Theorem~\ref{thm:algebra-saturated-statements}.
\end{proof}

\begin{corollary} \label{cor:globalsat}
Let $\mathcal{M}(\TT)$ be a decomposable CSmodel.
Then \[
\ker(\psi_{\TT})= \sum_{X_{C}=\xx_C\in \CC} I_{\glo({G_{X_C=\xx_C}})}.
\]
\end{corollary}
\begin{proof}

We show the following chain of inclusions
$$\ker(\psi_{\TT})=\sum_{X_C=\xx_C\in\CC}I_{\sat(G_{X_C=\xx_C})}\subseteq\sum_{X_{C}=\xx_C\in \CC} I_{\glo({G_{X_C=\xx_C}})}\subseteq\ker(\psi_\TT),$$
which implies the theorem. 
The equality follows from Corollary~\ref{cor:saturated-statements-for-minimal-dags} and the middle inclusion follows from the containment ${\sat(G_{X_C=\xx_C})}\subseteq {\glo(G_{X_C=\xx_C})}$ for all $X_C=\xx_C\in\CC$. 
For the last inclusion, let $J:=\sum_{X_{C}=\xx_C\in \CC} I_{\glo({G_{X_C=\xx_C}})}$. From \cite[Theorem 3.3]{DS22}, we know the
equality 
$$
\mathcal{V}(J)\cap\Delta_{|\RR|-1}^{\circ}=\mathcal{V}(\ker(\psi_{\TT}))\cap \Delta_{|\RR|-1}^{\circ}=\MM
(\TT).
$$
Since $\ker\psi_\TT$ is a prime ideal, this implies that
\begin{align*}
J&\subseteq \mathcal{I}\left(\mathcal{V}\left(J\right)\cap \Delta_{|\RR|-1}^{\circ}\right)\cap\R[D]\\
&=\mathcal{I}(\mathcal{V}(\ker(\psi_{\TT}))\cap \Delta_{|\RR|-1}^{\circ})\cap\R[D]\\
&=\mathcal{I}(\mathcal{V}(\ker\psi_\TT))\cap\R[D]=\ker \psi_{\TT}
\end{align*}
where $\mathcal{I}(V )$ denotes the set of polynomials in $\C[D]$ that vanish on a set $V\subseteq \C^{|\RR|}$.

\end{proof}

\subsection{Directed moralization for decomposable CSmodels}
To create perfect DAGs from non-perfect ones we define a directed version of the moralization operation. We use this operation to show that decomposable CSmodels can be described by perfect DAGs. We start by recalling the definition of moralization and its connection to d-separation.

\begin{definition}\label{def:moralization}
Let $G=([p],E)$ be a DAG. The \textit{moralization} of $G$, denoted by $G^m$, is the undirected graph with the vertex set $[p]$ that has an undirected edge for every directed edge in $E$, and an undirected edge $(u,v)$ whenever $u\to w$, $v\to w$ are edges in $G$.
\end{definition}
\begin{proposition}[{\cite[Proposition 3.25]{L96}}]\label{prop:d-sep-moralization}
Let $G=([p],E)$ be a DAG and $A,B,C$ be disjoint subsets of $[p]$. Then $C$ $d$-separates $A$ and $B$ in $G$ if and only if $C$ separates $A$ and $B$ in the moralization $(G_{\mathrm{an}(A\cup B\cup C)})^m$, where $G_{\mathrm{an}(A\cup B\cup C)}$ is the induced subgraph on the ancestors of $A\cup B\cup C$.
\end{proposition}

\begin{definition}\label{def:directed_moralization}
Let $G=([p],E)$ be a DAG. The \textit{directed moralization} of $G$, denoted by $G^{dm}$, is the directed graph with the vertex set $[p]$ that has a directed edge for every directed edge in $E$, and a directed edge $u\to v$ whenever $u\to w$, $v\to w$ are edges in $G$ and $u<v$.
\end{definition}

\begin{remark}
Note that the directed moralization of every DAG $G$ is indeed a DAG since our variables are topologically ordered, i.e. if $i\to j$ is an edge in $G$, then $i < j$. 
In case one wants to use directed moralization on a DAG $(V,E)$ with vertices $V$ and edges $E$ one has to fix an ordering of the variables.

Moreover, directed moralization produces a perfect DAG after applying the operation sufficiently many times. Since we can only add edges, applying the directed moralization $\binom{p}{2}$ times results in a perfect DAG.
\end{remark}

\begin{definition}
Let $G=([p],E)$ be a DAG. We denote by $G^{\per}$ the perfect DAG created from $G$ after applying the directed moralization $\binom{p}{2}$ times.
\end{definition}

\begin{remark}
To produce a chordal graph to which a distribution in a DAG model $G$ is Markov, one can moralize $G$ and then take a triangulation of the resulting undirected graph. If an ordering of the vertices is fixed, this undirected graph can then be directed according to this ordering.

The graph $G^{\per}$ is one possible perfect graph one may produce using a particular triangulation. Iterating directed moralizations in Definition~\ref{def:directed_moralization} produces this triangulation on the skeleton of $G^{\per}$.
\end{remark}

The goal is to prove the following theorem about the generators of $\ker(\psi_{\TT})$. It implies that in any balanced CStree $\TT$ we may replace all context DAGs $H$ by $H^{\per}$ without changing the model.

\begin{theorem}
\label{thm:balanced_defined_by_perfect}
A CStree $\TT$ is balanced if and only if 
\[
\ker(\psi_{\TT})=\sum_{X_C=\xx_C\in\CC}I_{\sat((G_{X_C=\xx_C})^{\per})}.
\]
In particular, every decomposable CSmodel can be described by a collection of perfect DAGs.
\end{theorem}

\begin{lemma}\label{lem:saturated_CI_induced_subgraph}
Let $G=([p],E)$ be a DAG and let $S=X_i\independent X_j|X_{[p]\setminus \{i,j\}}$ be a saturated CI statement entailed by $G$. Then $S\in\sat(G)\setminus\sat(G^{dm})$ if and only if at least one of the following two statements holds.
\begin{enumerate}
    \item There exist $k,l\in [p]$ with $k>i,j$ and $l>k$ such that $G$ contains one of the following two graphs on the vertices $\{i,j,k,l\}$ as an induced subgraph.
    
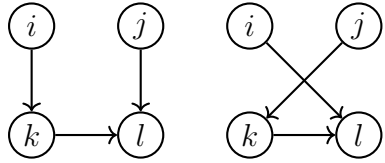
\begin{center}
    \begin{tikzpicture}[thick,scale=0.29]
	 \node[circle, draw, fill=black!0, inner sep=1pt, minimum width=1pt, minimum size=0.6cm] (i1) at (0,0) {$i$};
 	 \node[circle, draw, fill=black!0, inner sep=1pt, minimum width=1pt, minimum size=0.6cm] (j1) at (5,0) {$j$};
 	 \node[circle, draw, fill=black!0, inner sep=1pt, minimum width=1pt, minimum size=0.6cm] (k1) at (0,-5) {$k$};
 	 \node[circle, draw, fill=black!0, inner sep=1pt, minimum width=1pt, minimum size=0.6cm] (l1) at (5,-5) {$l$};

 	 \draw[->]   (i1) -- (k1) ;
   \draw[->]   (j1) -- (l1) ;
   \draw[->]   (k1) -- (l1) ;

	 \node[circle, draw, fill=black!0, inner sep=1pt, minimum width=1pt, minimum size=0.6cm] (i2) at (10,0) {$i$};
 	 \node[circle, draw, fill=black!0, inner sep=1pt, minimum width=1pt, minimum size=0.6cm] (j2) at (15,0) {$j$};
 	 \node[circle, draw, fill=black!0, inner sep=1pt, minimum width=1pt, minimum size=0.6cm] (k2) at (10,-5) {$k$};
 	 \node[circle, draw, fill=black!0, inner sep=1pt, minimum width=1pt, minimum size=0.6cm] (l2) at (15,-5) {$l$};

 	 \draw[->]   (i2) -- (l2) ;
   \draw[->]   (j2) -- (k2) ;
   \draw[->]   (k2) -- (l2) ;

\end{tikzpicture}
\end{center}
    \item There exist $l_1,l_2,k\in [p]$ with $k>i,j$ and $l_1,l_2>k$ such that $G$ contains the following graph on the vertices $\{i,j,k,l_1,l_2\}$ as an induced subgraph.
    
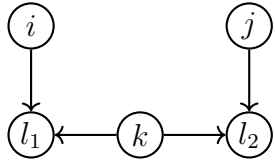
\begin{center}
    \begin{tikzpicture}[thick,scale=0.29]
	   \node[circle, draw, fill=black!0, inner sep=1pt, minimum width=1pt, minimum size=0.6cm] (i) at (0,0) {$i$};
 	 \node[circle, draw, fill=black!0, inner sep=1pt, minimum width=1pt, minimum size=0.6cm] (j) at (10,0) {$j$};
 	 \node[circle, draw, fill=black!0, inner sep=1pt, minimum width=1pt, minimum size=0.6cm] (k) at (5,-5) {$k$};
 	 \node[circle, draw, fill=black!0, inner sep=1pt, minimum width=1pt, minimum size=0.6cm] (l1) at (0,-5) {$l_1$};
      \node[circle, draw, fill=black!0, inner sep=1pt, minimum width=1pt, minimum size=0.6cm] (l2) at (10,-5) {$l_2$};

 	 \draw[->]   (i) -- (l1) ;
   \draw[->]   (j) -- (l2) ;
   \draw[->]   (k) -- (l1) ;
   \draw[->]   (k) -- (l2) ;

\end{tikzpicture}
\end{center}
\end{enumerate}
\end{lemma}
\begin{proof}
Since $S$ holds in $G$, the vertices $i$ and $j$ do not have a common child by Proposition~\ref{prop:d-sep-moralization}. However, they do have a common child after directed moralization. Assume this child is labeled $k$ with $k\in [p]$, $k>i,j$. 
If $k$ is a child of one of $i$ or $j$ and the other edge is added by the directed moralization we have one of the graphs in (1) as a subgraph of $G$. Moreover, no other edge can exist in the induced subgraph on these four vertices since then $i,j$ would have a common child or be connected via an edge.
If $k$ is not a child of either $i$ or $j$ in $G$ then the graph in (2) is a subgraph of $G$. Again other edges cannot exist for the same reason, hence this graph is the induced subgraph on these five vertices.

For the other direction note that if any of these graphs is an induced subgraph of $G$, then $(G^{dm})^m$ has an edge between $i$ and $j$. Thus $S\notin\sat(G^{dm})$ by Proposition~\ref{prop:d-sep-moralization}.
\end{proof}

Let $G=([p],E)$ be a DAG and let $i,j\in [p]$ with no common child and not adjacent. We denote the number of pairs $k,l$ as in Lemma~\ref{lem:saturated_CI_induced_subgraph}(1) by $n_1^{i,j}(G)$ and by $n_2^{i,j}(G)$ the number of triples $k,l_1,l_2$ as in Lemma~\ref{lem:saturated_CI_induced_subgraph}(2).

\begin{lemma}\label{lem:reduction_to_subgraphs}
     Let $G=([p],E)$ be a DAG and let $i,j\in [p]$ with no common child and not adjacent. Let $k,l$ as in Lemma~\ref{lem:saturated_CI_induced_subgraph}(1). Let $H$ be a subgraph of $G$ obtained by removing at least one of the edges in the induced subgraph on $\{i,j,k,l\}$. Then
    \begin{enumerate}
        \item $n_1^{i,j}(H)<n_1^{i,j}(G)$ and
        \item $n_1^{i,j}(H)+n_2^{i,j}(H)\leq n_1^{i,j}(G)+n_2^{i,j}(G)$.
    \end{enumerate}
    Assume $n_1^{i,j}(G)=0$. Let $k,l_1,l_2$ as in Lemma~\ref{lem:saturated_CI_induced_subgraph}(2). Let $H$ be a subgraph of $G$ obtained by removing at least one of the edges in the induced subgraph on $\{i,j,k,l_1,l_2\}$. Then
    \begin{enumerate}
        \item[(3)] $n_1^{i,j}(H)=0$ and
        \item[(4)] $n_2^{i,j}(H)<n_2^{i,j}(G)$.
    \end{enumerate}
\end{lemma}
\begin{proof}
    (1): Assume there is a pair $k,l$ such that the induced subgraph on $\{i,j,k,l\}$ in $G$ is not of the form in Lemma~\ref{lem:saturated_CI_induced_subgraph}(1) but the induced subgraph on these vertices in $H$ does have that form. Then the induced subgraph on $\{i,j,k,l\}$ in $G$ has at least one more edge.
    Since $i,j$ are not adjacent in $G$ by assumption, it has to be one of the edges $i\to l, j\to k$. But in both cases $i,j$ have a common child.
    When removing an edge contained in an induced subgraph on vertices $\{i,j,k,l\}$ as in Lemma~\ref{lem:saturated_CI_induced_subgraph}(1) then the total number of such pairs in $H$ is strictly smaller.
    
    (2) Assume we have a triple $k,l_1,l_2$ such that the induced subgraph on $\{i,j,k,l\}$ in $G$ is not of the form in Lemma~\ref{lem:saturated_CI_induced_subgraph}(2) but the induced subgraph on these vertices in $H$ does have that form. 
    With the same reasoning as above, exactly one of the two edges $i\to k$ or $j\to k$ must have been removed. Therefore, either the induced subgraph on $\{i,j,k,l_1\}$ or on $\{i,j,k,l_2\}$ in $G$ is of the form in Lemma~\ref{lem:saturated_CI_induced_subgraph}(1). This shows that $n_2^{i,j}(H)-n_2^{i,j}(G)$ is at most $n_1^{i,j}(G)-n_1^{i,j}(H)$.

    (3) We already saw in the proof of (1) that no new pair $k,l$ can emerge.

    (4) If there was a triple $\{k,l_1,l_2\}$ in $H$ that has the form in Lemma~\ref{lem:saturated_CI_induced_subgraph}(2), then it was already there in $G$ as the only edges that can be added to this subgraph not leading to a common child are $i\to k$ or $j\to k$. However, both imply $n_1^{i,j}(G)\geq 1$, a contradiction.
\end{proof}

\begin{proposition}
\label{prop:induction_step}
Let $\TT$ be a balanced CStree and let $X_i\independent X_j|X_{[p]\setminus \{i,j\}}\in\sat(G_\emptyset)\setminus\sat((G_\emptyset)^{dm})$. For every $\xx_{[p]\setminus \{i,j\}}\in\RR_{[p]\setminus \{i,j\}}$ the CSI statement $X_i\independent X_j|X_{[p]\setminus \{i,j\}}=\xx_{[p]\setminus \{i,j\}}$ is implied by some context DAG $G_{X_D=\xx_D}$ with $D\neq \emptyset$.
\end{proposition}
\begin{proof}
    By Lemma~\ref{lem:saturated_CI_induced_subgraph} one of the graphs in Lemma~\ref{lem:saturated_CI_induced_subgraph} is contained in $G$ as a subgraph.
    Assume first we are in case (1). By changing the roles of $i$ and $j$ we can assume the graph is
    
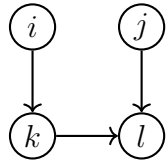
\begin{center}
    \begin{tikzpicture}[thick,scale=0.29]
	 \node[circle, draw, fill=black!0, inner sep=1pt, minimum width=1pt, minimum size=0.6cm] (i1) at (0,0) {$i$};
 	 \node[circle, draw, fill=black!0, inner sep=1pt, minimum width=1pt, minimum size=0.6cm] (j1) at (5,0) {$j$};
 	 \node[circle, draw, fill=black!0, inner sep=1pt, minimum width=1pt, minimum size=0.6cm] (k1) at (0,-5) {$k$};
 	 \node[circle, draw, fill=black!0, inner sep=1pt, minimum width=1pt, minimum size=0.6cm] (l1) at (5,-5) {$l$};

 	 \draw[->]   (i1) -- (k1) ;
   \draw[->]   (j1) -- (l1) ;
   \draw[->]   (k1) -- (l1) ;
\end{tikzpicture}
\end{center}
Let $C:=[p]\setminus \{i,j,k,l\}$ and let $\xx_C\in\RR_C$ be arbitrary. By Theorem~\ref{thm:balanced-for-every-context} the CStree $\TT_{\cc}$ is still balanced. Let $G$ be the empty context DAG of this CStree, i.e. $G:=G_{\TT_{\cc},\emptyset}$.
We claim that at least one of the three edges is missing in $G$:

\textit{Proof of claim}:
Let $\xx_i\in\RR_i$ be any outcome and consider the balanced CStree $\TT_{\cc,X_i=\xx_i}$. By Theorem~\ref{thm:three_perfect_iff_balanced} its empty context DAG is perfect and thus one of the edges $j\to l$ or $k\to l$ is missing in $G_{\TT_{\cc,X_i=\xx_i},\emptyset}$. The argument for both cases is analogous, thus we do the proof for $j\to l$ only.
In this case the CSI statements $X_l\independent X_j|X_k,\cc,X_i=\xx_i$ and $X_l\independent X_i|X_{\{j,k\}},\cc$ hold. Using Proposition~\ref{prop:weird_formula} these imply the statement $X_l\independent X_{i,j}|X_k,\cc$ which is not true as there is an edge $j\to l$. \qedsymbol

Depending on which edge is missing, one of the CSI statements $X_i\independent X_k|X_{j,l},\cc$, $X_k\independent X_l|X_{i,j},\cc$, $X_j\independent X_l|X_{i,k},\cc$ holds in $\TT$.
We do the proof for the first one, the other two work analogously. The statement $X_i\independent X_k|X_{C\cup\{j,l\}}$ does not hold in $G_\emptyset$ therefore there is a context DAG $G_{X_D=\xx_D}$ with $D\subset C$, $(\xx_C)_D=\xx_D$ and $D\neq\emptyset$ such that the CSI statement $X_i\independent X_k|X_{C\setminus D\cup\{j,l\}}, X_D=\xx_D$ is entailed by $G_{X_D=\xx_D}$. There is now at least one less subgraph of type (1) in the empty context DAG of $\TT_{X_D=\xx_D}$ by Lemma~\ref{lem:reduction_to_subgraphs}.
Continuing this process will result in a $D$ with no subgraphs of type (1) and only subgraphs of type (2).

The argument for (2) works similarly. To receive a DAG on three vertices as in the argument above we fix $X_j$ and $X_{l_2}$. In the end, again by Lemma~\ref{lem:reduction_to_subgraphs} there are no such induced subgraphs at all and thus the CSI statement $X_i\independent X_j|X_{C\setminus D\cup\{j,l\}}, X_D=\xx_D$ holds in $G_{X_D=\xx_D}$, hence also $X_i\independent X_j|X_{k,l},\cc$.
\end{proof}

\begin{proof}[{Proof of \ref{thm:balanced_defined_by_perfect}}]
Assume the equality holds. Since $G_{X_C=\xx_C}^{\per}$ is perfect, the ideal $I_{\sat(G_{X_C=\xx_C}^{\per})}$ is a toric ideal for each context $X_C=\xx_C\in\CC$. Therefore, $\ker(\psi_{\TT})$ is a binomial ideal. Since it is also prime, we conclude that $\ker(\psi_\TT)$ is toric. This is equivalent to $\TT$ being balanced by \cite[Theorem 10]{DG20}.
For the additional statement we note that we may take $\CC:=\CC_{\TT}$ and by moralizing we find a collection as required.

For the other direction it suffices to prove that using directed moralization once on an arbitrary context DAG does not alter the set of saturated CSI statements in $\TT$. Let $X_C=\xx_C\in\CC$. By replacing $\TT$ with $\TT_{X_C=\xx_C}$ which is still balanced by Theorem~\ref{thm:balanced-for-every-context} we may assume that we applied directed moralization to the empty context DAG. 
We want to prove
\[
\sum_{X_C=\xx_C\in\CC}I_{\sat((G_{X_C=\xx_C}))}=\sum_{X_C=\xx_C\in\CC\setminus\{\emptyset\}}I_{\sat((G_{X_C=\xx_C}))}+I_{\sat((G_{\emptyset})^{dm})}.
\]
Let $X_A\independent X_B|X_S$ be a saturated CSI statement implied by $G_\emptyset$ but not by $(G_\emptyset)^{dm}$. It suffices to show that $X_i\independent X_j|X_{[p]\setminus \{i,j\}}$ is implied by some other context DAG for every $i\in A, j\in B$ by using the intersection axiom.

By Proposition~\ref{prop:induction_step} for every $\xx_{[p]\setminus \{i,j\}}\in\RR_{[p]\setminus \{i,j\}}$ the CSI statement $X_i\independent X_j|X_{[p]\setminus \{i,j\}}=\xx_{[p]\setminus \{i,j\}}$ is implied by some context DAG $G_{X_F=\xx_F}$ with $X_F=\xx_F\in\CC$ and $F\neq\emptyset$. Using absorption we see that all polynomials associated to the statement $X_A\independent X_B|X_S$ are contained in $\sum_{X_C=\xx_C\in\CC\setminus\{\emptyset\}}I_{\sat((G_{X_C=\xx_C}))}$.
\end{proof}

\begin{example}
    In Figure~\ref{fig:directed_moralization} we give an example of a binary, balanced CStree where we can use directed moralization on the empty context DAG twice to obtain a perfect DAG. First, the edge $3\to 4$ is added and then the edge $2\to 3$ is added. This example can easily be generalized to obtain a DAG on $p$ vertices where we can apply directed moralization $p-3$ times and obtain an additional edge each time. For this we pick a path ending in $p$ and starting with 2, and omitting 3. The edge $3\to p$ is then added. Lastly, we connect 1 to everything.
    The two other context contexts should remove the edges $3\to 5$ and $p-1\to p$ respectively.
    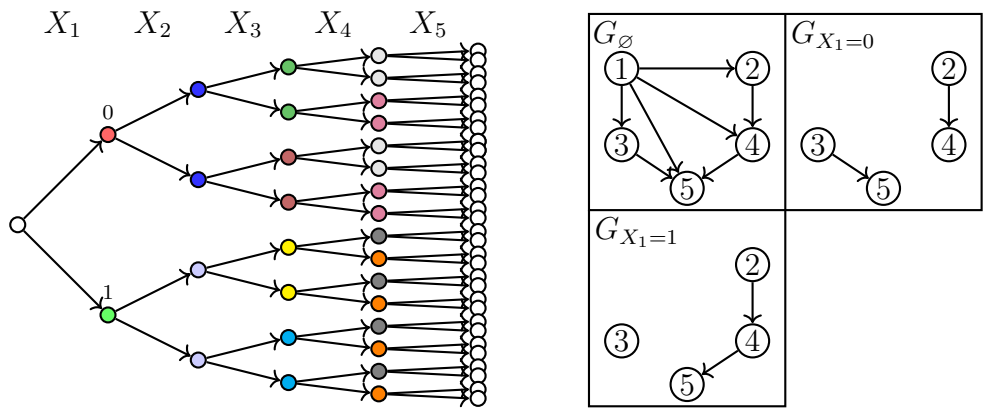
\begin{figure}[H]

\begin{center}
\begin{tikzpicture}[thick,scale=0.3]
\node[circle, draw, fill=black!0, inner sep=2pt, minimum width=1pt] () at (4.4,0.2)  {};
\node[circle, draw, fill=black!0, inner sep=2pt, minimum width=1pt] () at (4.4,-0.2)  {};
\node[circle, draw, fill=black!0, inner sep=2pt, minimum width=1pt] () at (4.4,-0.8)  {};
\node[circle, draw, fill=black!0, inner sep=2pt, minimum width=1pt] () at (4.4,-1.2)  {};
\node[circle, draw, fill=black!0, inner sep=2pt, minimum width=1pt] () at (4.4,-1.8)  {};
\node[circle, draw, fill=black!0, inner sep=2pt, minimum width=1pt] () at (4.4,-2.2)  {};
\node[circle, draw, fill=black!0, inner sep=2pt, minimum width=1pt] () at (4.4,-2.8)  {};
\node[circle, draw, fill=black!0, inner sep=2pt, minimum width=1pt] () at (4.4,-3.2)  {};
\node[circle, draw, fill=black!0, inner sep=2pt, minimum width=1pt] () at (4.4,-3.8)  {};
\node[circle, draw, fill=black!0, inner sep=2pt, minimum width=1pt] () at (4.4,-4.2)  {};
\node[circle, draw, fill=black!0, inner sep=2pt, minimum width=1pt] () at (4.4,-3.8)  {};
\node[circle, draw, fill=black!0, inner sep=2pt, minimum width=1pt] () at (4.4,-4.2)  {};
\node[circle, draw, fill=black!0, inner sep=2pt, minimum width=1pt] () at (4.4,-4.8)  {};
\node[circle, draw, fill=black!0, inner sep=2pt, minimum width=1pt] () at (4.4,-5.2)  {};
\node[circle, draw, fill=black!0, inner sep=2pt, minimum width=1pt] () at (4.4,-5.8)  {};
\node[circle, draw, fill=black!0, inner sep=2pt, minimum width=1pt] () at (4.4,-6.2)  {};
\node[circle, draw, fill=black!0, inner sep=2pt, minimum width=1pt] () at (4.4,-6.8)  {};
\node[circle, draw, fill=black!0, inner sep=2pt, minimum width=1pt] () at (4.4,-7.2)  {};
\node[circle, draw, fill=black!0, inner sep=2pt, minimum width=1pt] () at (4.4,-7.8)  {};
\node[circle, draw, fill=black!0, inner sep=2pt, minimum width=1pt] () at (4.4,-8.2)  {};
\node[circle, draw, fill=black!0, inner sep=2pt, minimum width=1pt] () at (4.4,-8.8)  {};
\node[circle, draw, fill=black!0, inner sep=2pt, minimum width=1pt] () at (4.4,-9.2)  {};
\node[circle, draw, fill=black!0, inner sep=2pt, minimum width=1pt] () at (4.4,-9.8)  {};
\node[circle, draw, fill=black!0, inner sep=2pt, minimum width=1pt] () at (4.4,-10.2)  {};
\node[circle, draw, fill=black!0, inner sep=2pt, minimum width=1pt] () at (4.4,-10.8)  {};
\node[circle, draw, fill=black!0, inner sep=2pt, minimum width=1pt] () at (4.4,-11.2)  {};
\node[circle, draw, fill=black!0, inner sep=2pt, minimum width=1pt] () at (4.4,-11.8)  {};
\node[circle, draw, fill=black!0, inner sep=2pt, minimum width=1pt] () at (4.4,-12.2)  {};
\node[circle, draw, fill=black!0, inner sep=2pt, minimum width=1pt] () at (4.4,-12.8)  {};
\node[circle, draw, fill=black!0, inner sep=2pt, minimum width=1pt] () at (4.4,-13.2)  {};
\node[circle, draw, fill=black!0, inner sep=2pt, minimum width=1pt] () at (4.4,-13.8)  {};
\node[circle, draw, fill=black!0, inner sep=2pt, minimum width=1pt] () at (4.4,-14.2)  {};
\node[circle, draw, fill=black!0, inner sep=2pt, minimum width=1pt] () at (4.4,-14.8)  {};
\node[circle, draw, fill=black!0, inner sep=2pt, minimum width=1pt] () at (4.4,-15.2)  {};

 	 \node[circle, draw, fill=black!10, inner sep=2pt, minimum width=1pt] (w3) at (0,0)  {};
 	 \node[circle, draw, fill=black!10, inner sep=2pt, minimum width=1pt] (w4) at (0,-1) {};
 	 \node[circle, draw, fill=purple!50, inner sep=2pt, minimum width=1pt] (w5) at (0,-2) {};
 	 \node[circle, draw, fill=purple!50, inner sep=2pt, minimum width=1pt] (w6) at (0,-3) {};
 	 \node[circle, draw, fill=black!10, inner sep=2pt, minimum width=1pt] (v3) at (0,-4)  {};
 	 \node[circle, draw, fill=black!10, inner sep=2pt, minimum width=1pt] (v4) at (0,-5) {};
 	 \node[circle, draw, fill=purple!50, inner sep=2pt, minimum width=1pt] (v5) at (0,-6) {};
 	 \node[circle, draw, fill=purple!50, inner sep=2pt, minimum width=1pt] (v6) at (0,-7) {};
 	 \node[circle, draw, fill=black!50, inner sep=2pt, minimum width=2pt] (w3i) at (0,-8)  {};
 	 \node[circle, draw, fill=orange, inner sep=2pt, minimum width=2pt] (w4i) at (0,-9) {};
 	 \node[circle, draw, fill=black!50, inner sep=2pt, minimum width=2pt] (w5i) at (0,-10) {};
 	 \node[circle, draw, fill=orange, inner sep=2pt, minimum width=2pt] (w6i) at (0,-11) {};
 	 \node[circle, draw, fill=black!50, inner sep=2pt, minimum width=2pt] (v3i) at (0,-12)  {};
 	 \node[circle, draw, fill=orange, inner sep=2pt, minimum width=2pt] (v4i) at (0,-13) {};
 	 \node[circle, draw, fill=black!50, inner sep=2pt, minimum width=2pt] (v5i) at (0,-14) {};
 	 \node[circle, draw, fill=orange, inner sep=2pt, minimum width=2pt] (v6i) at (0,-15) {};

	 \node[circle, draw, fill=green!60!black!60, inner sep=2pt, minimum width=2pt] (w1) at (-4,-.5) {};
 	 \node[circle, draw, fill=green!60!black!60, inner sep=2pt, minimum width=2pt] (w2) at (-4,-2.5) {}; 
 	 \node[circle, draw, fill=red!60!black!60, inner sep=2pt, minimum width=2pt] (v1) at (-4,-4.5) {};
 	 \node[circle, draw, fill=red!60!black!60, inner sep=2pt, minimum width=2pt] (v2) at (-4,-6.5) {};	 
	 \node[circle, draw, fill=yellow, inner sep=2pt, minimum width=2pt] (w1i) at (-4,-8.5) {};
 	 \node[circle, draw, fill=yellow, inner sep=2pt, minimum width=2pt] (w2i) at (-4,-10.5) {};
 	 \node[circle, draw, fill=cyan, inner sep=2pt, minimum width=2pt] (v1i) at (-4,-12.5) {};
 	 \node[circle, draw, fill=cyan, inner sep=2pt, minimum width=2pt] (v2i) at (-4,-14.5) {};

 	 \node[circle, draw, fill=blue!80, inner sep=2pt, minimum width=2pt] (w) at (-8,-1.5) {};
 	 \node[circle, draw, fill=blue!80, inner sep=2pt, minimum width=2pt] (v) at (-8,-5.5) {};	
 	 \node[circle, draw, fill=blue!20, inner sep=2pt, minimum width=2pt] (wi) at (-8,-9.5) {};
 	 \node[circle, draw, fill=blue!20, inner sep=2pt, minimum width=2pt] (vi) at (-8,-13.5) {};
	 
     \node[circle, draw, fill=red!60, inner sep=2pt, minimum width=2pt] (r) at (-12,-3.5) {};
     \node[] () at (-12,-2.5) {\tiny{$0$}};
 	 \node[circle, draw, fill=green!60, inner sep=2pt, minimum width=2pt] (ri) at (-12,-11.5) {};
     \node[] () at (-12,-10.5) {\tiny{$1$}};

 	 \node[circle, draw, fill=black!0, inner sep=2pt, minimum width=2pt] (I) at (-16,-7.5) {};

   \node[] () at (-14,1.4) {$X_1$};
   \node[] () at (-10,1.4) {$X_2$};
   \node[] () at (-6,1.4) {$X_3$};
   \node[] () at (-2,1.4) {$X_4$};
   \node[] () at (2.2,1.4) {$X_5$};

 	 \draw[->]   (I) --    (r);
 	 \draw[->]   (I) --   (ri) ;

 	 \draw[->]   (r) --   (w) ;
 	 \draw[->]   (r) --   (v) ;

 	 \draw[->]   (w) --  (w1) ;
 	 \draw[->]   (w) --  (w2) ;

 	 \draw[->]   (w1) --   (w3) ;
 	 \draw[->]   (w1) --   (w4) ;
 	 \draw[->]   (w2) --  (w5) ;
 	 \draw[->]   (w2) --  (w6) ;

 	 \draw[->]   (v) --  (v1) ;
 	 \draw[->]   (v) --  (v2) ;

 	 \draw[->]   (v1) --  (v3) ;
 	 \draw[->]   (v1) --  (v4) ;
 	 \draw[->]   (v2) --  (v5) ;
 	 \draw[->]   (v2) --  (v6) ;

 	 \draw[->]   (ri) --   (wi) ;
 	 \draw[->]   (ri) -- (vi) ;

 	 \draw[->]   (wi) --  (w1i) ;
 	 \draw[->]   (wi) --  (w2i) ;

 	 \draw[->]   (w1i) --  (w3i) ;
 	 \draw[->]   (w1i) -- (w4i) ;
 	 \draw[->]   (w2i) --  (w5i) ;
 	 \draw[->]   (w2i) --  (w6i) ;

 	 \draw[->]   (vi) --  (v1i) ;
 	 \draw[->]   (vi) --  (v2i) ;

 	 \draw[->]   (v1i) --  (v3i) ;
 	 \draw[->]   (v1i) -- (v4i) ;
 	 \draw[->]   (v2i) -- (v5i) ;
 	 \draw[->]   (v2i) --  (v6i) ;
 	 
 	 \draw[->]   (w3) --   (4,0.2) ;
 	 \draw[->]   (w3) --   (4,-0.2) ;
 	 \draw[->]   (w4) --   (4,-0.8) ;
 	 \draw[->]   (w4) --   (4,-1.2) ;
 	 \draw[->]   (w5) --   (4,-1.8) ;
 	 \draw[->]   (w5) --   (4,-2.2) ;
 	 \draw[->]   (w6) --   (4,-2.8) ;
 	 \draw[->]   (w6) --   (4,-3.2) ;
 	 \draw[->]   (v3) --   (4,-3.8) ;
 	 \draw[->]   (v3) --   (4,-4.2) ;
 	 \draw[->]   (v4) --   (4,-4.8) ;
 	 \draw[->]   (v4) --   (4,-5.2) ;
 	 \draw[->]   (v5) --   (4,-5.8) ;
 	 \draw[->]   (v5) --   (4,-6.2) ;
 	 \draw[->]   (v6) --   (4,-6.8) ;
 	 \draw[->]   (v6) --   (4,-7.2) ;
 	 \draw[->]   (w3i) --   (4,-7.8) ;
 	 \draw[->]   (w3i) --   (4,-8.2) ;
 	 \draw[->]   (w4i) --   (4,-8.8) ;
 	 \draw[->]   (w4i) --   (4,-9.2) ;
 	 \draw[->]   (w5i) --   (4,-9.8) ;
 	 \draw[->]   (w5i) --   (4,-10.2) ;
 	 \draw[->]   (w6i) --   (4,-10.8) ;
 	 \draw[->]   (w6i) --   (4,-11.2) ;
 	 \draw[->]   (v3i) --   (4,-11.8) ;
 	 \draw[->]   (v3i) --   (4,-12.2) ;
 	 \draw[->]   (v4i) --   (4,-12.8) ;
 	 \draw[->]   (v4i) --   (4,-13.2) ;
 	 \draw[->]   (v5i) --   (4,-13.8) ;
 	 \draw[->]   (v5i) --   (4,-14.2) ;
 	  \draw[->]   (v6i) --   (4,-14.8) ;
 	 \draw[->]   (v6i) --   (4,-15.2) ;

\end{tikzpicture}
\hspace{1cm}
\begin{tikzpicture}[thick,scale=0.29]
	\draw (0,0) -- (18,0) -- (18,-9) -- (0,-9) -- cycle;
    \draw (9,0) -- (9,-18) -- (0,-18) -- (0,-9);

	 \node[circle, draw, fill=black!0, inner sep=1pt, minimum width=1pt] (1v1) at (1.5,-2.5) {$1$};
 	 \node[circle, draw, fill=black!0, inner sep=1pt, minimum width=1pt] (1v2) at (7.5,-2.5) {$2$};
 	 \node[circle, draw, fill=black!0, inner sep=1pt, minimum width=1pt] (1v3) at (1.5,-6) {$3$};
 	 \node[circle, draw, fill=black!0, inner sep=1pt, minimum width=1pt] (1v4) at (7.5,-6) {$4$}; 	 
 	 \node[circle, draw, fill=black!0, inner sep=1pt, minimum width=1pt] (1v5) at (4.5,-8) {$5$};

 	 \node[circle, draw, fill=black!0, inner sep=1pt, minimum width=1pt] (2v2) at (16.5,-2.5) {$2$};
 	 \node[circle, draw, fill=black!0, inner sep=1pt, minimum width=1pt] (2v3) at (10.5,-6) {$3$};
 	 \node[circle, draw, fill=black!0, inner sep=1pt, minimum width=1pt] (2v4) at (16.5,-6) {$4$}; 	 
 	 \node[circle, draw, fill=black!0, inner sep=1pt, minimum width=1pt] (2v5) at (13.5,-8) {$5$};

	 \node[circle, draw, fill=black!0, inner sep=1pt, minimum width=1pt] (3v2) at (7.5,-11.5) {$2$};
 	 \node[circle, draw, fill=black!0, inner sep=1pt, minimum width=1pt] (3v3) at (1.5,-15) {$3$};
 	 \node[circle, draw, fill=black!0, inner sep=1pt, minimum width=1pt] (3v4) at (7.5,-15) {$4$}; 	 
 	 \node[circle, draw, fill=black!0, inner sep=1pt, minimum width=1pt] (3v5) at (4.5,-17) {$5$};

 	 \draw[->]   (1v1) -- (1v2) ;
 	 \draw[->]   (1v1) -- (1v3) ;
 	 \draw[->]   (1v1) -- (1v4) ;
 	 \draw[->]   (1v1) -- (1v5) ;
 	 \draw[->]   (1v2) -- (1v4) ;
 	 \draw[->]   (1v3) -- (1v5) ;
 	 \draw[->]   (1v4) -- (1v5) ;

 	 \draw[->]   (2v2) -- (2v4) ;
 	 \draw[->]   (2v3) -- (2v5) ;
 	 
 	 \draw[->]   (3v2) -- (3v4) ;
 	 \draw[->]   (3v4) -- (3v5) ;

	 \node at (1.15,-1) {$G_{\emptyset}$} ;
	 \node at (11.2,-1) {$G_{X_1=0}$} ;
	 \node at (2.2,-10) {$G_{X_1 = 1}$} ;
\end{tikzpicture}
\caption{\small{A balanced CStree on whose empty context DAG we have to use directed moralization twice to receive a perfect DAG.} 
} 
\label{fig:directed_moralization}
\end{center}
\end{figure}
\end{example}

{\bf Acknowledgements}.
The authors thank the editors and the anonymous referees for the careful reading of the manuscript and their many insightful comments and suggestions. In particular, we thank Referee \#2 for Remark 5.12.
The authors also thank the Max-Planck-Institute for Mathematics in the Sciences in Leipzig for their hospitality in the Summer of 2022.

YA was supported by the National Science Foundation Graduate Research Fellowship Grant No. DGE 2146752. ED was supported by the FCT grant 2020.01933.CEECIND, and partially supported by CMUP under the FCT grant UIDB/00144/2020.

\setcitestyle{numbers}
\bibliographystyle{alpha} 
\bibliography{refs.bib}

\end{document}